 \documentclass{amsart}

\usepackage{pdfsync}

 \usepackage[plainpages=false]{hyperref}
 \usepackage{amsfonts,amsmath,amssymb,amsthm,accents}
 \usepackage{latexsym,lscape,rawfonts,mathrsfs}

 \usepackage[all]{xy}
 \usepackage{eufrak}
 \usepackage{graphicx,psfrag}
 \usepackage{pstool}

 \usepackage{array,tabularx}

 \usepackage{appendix}

\usepackage{txfonts}

 \newcommand{\R}{\ensuremath{\mathbb{R}}}

 \newcommand{\ba}{\begin{align}}
 \newcommand{\ea}{\end{align}}
 \newcommand{\bal}{\begin{align*}}
 \newcommand{\eal}{\end{align*}}

 \DeclareMathOperator{\supp}{supp}

 \DeclareMathOperator{\diam}{diam}

 \newcommand{\Rm}{\mathcal{R}m}
 \newcommand{\Rc}{\mathcal{R}c}
 \newcommand{\Sc}{\mathcal{R}}

 \newcommand{\dvol}{\text{d}V}
\newcommand{\dmu}{\text{d}\mu}
\newcommand{\tdmu}{\text{d}\tilde{\mu}}
 \newcommand{\darea}{\text{d}\sigma}
 \newcommand{\area}{\mathcal{A}}
 \newcommand{\vol}{Vol}
 
 \newcommand{\volfm}{\bar{\mu}}

\newcommand{\Euler}{\mathcal{P}_{\chi}}
\newcommand{\bEuler}{\mathcal{TP}_{\chi}}

\newcommand{\rad}{\mathbf{d}}

\renewcommand{\epsilon}{\varepsilon}

 \makeatletter
 \def\ExtendSymbol#1#2#3#4#5{\ext@arrow 0099{\arrowfill@#1#2#3}{#4}{#5}}
 
 \makeatother

 \makeatletter
 \def\ExtendSymbol#1#2#3#4#5{\ext@arrow 0099{\arrowfill@#1#2#3}{#4}{#5}}
 
 \makeatother

 \definecolor{hao}{rgb}{1,0.5,0}
 \definecolor{miao}{cmyk}{0.5,0,0.2,0.2}
 \definecolor{qiao}{gray}{0.96}

\newtheorem{prop}{Proposition}[section]

\newtheorem{proposition}[prop]{Proposition}

\newtheorem{theorem}[prop]{Theorem}

\newtheorem{lemma}[prop]{Lemma}
\newtheorem{claim}[prop]{Claim}

\newtheorem{corollary}[prop]{Corollary}

\newtheorem{remark}[prop]{Remark}

\newtheorem{definition}[prop]{Definition}

\newtheorem{conjecture}[prop]{Conjecture}

\newtheorem*{theorem*}{Theorem}

\newenvironment{dedication}
  {   
   \itshape             
   \raggedleft          
  }
  {\par 
  }

\numberwithin{equation}{section}
\title[\(\varepsilon\)-regularity for 4-d Ricci shrinkers]{$\epsilon$-Regularity and Structure of $4$-Dimensional Shrinking Ricci Solitons}
\author[Shaosai Huang]{Shaosai Huang}
\email{sshuang@math.wisc.edu}
\address{Department of Mathematics, University of Wisconsin - Madison, 480 Lincoln Drive, Madison, WI, 53706, U.S.A.}
\keywords{Ricci soliton, $\varepsilon$-regularity, collapsing, Bakry-\'Emery Ricci curvature}
\subjclass{53Cxx, 58Hxx}

\begin{document} 
\date{\today}

\maketitle
\begin{dedication}
----- to Kelly.
\end{dedication}
\begin{abstract}
A closed four dimensional manifold cannot possess a non-flat Ricci soliton
metric with arbitrarily small $L^2$-norm of the curvature. In this paper, we
localize this fact in the case of shrinking Ricci solitons by proving an
$\varepsilon$-regularity theorem, thus confirming a conjecture of
Cheeger-Tian~\cite{CT05}. As applications, we will also derive structural
results concerning the degeneration of the metrics on a family of complete
non-compact four dimensional shrinking Ricci solitons \emph{without} a uniform
entropy lower bound. In the appendix, we provide a detailed account of the
equivariant good chopping theorem when collapsing with locally bounded
curvature happens.
\end{abstract}

\tableofcontents

\section{Introduction}
A four dimensional gradient shrinking Ricci soliton (a.k.a. 4-d Ricci shrinker)
is a four dimensional Riemannian manifold $(M^4,g)$ equipped with a potential
function $f$ satisfying the defining equation (with a fixed scaling)
\begin{align}
\Rc_g+\nabla^2 f=\frac{1}{2}g,
\label{eqn: defn}
\end{align}
where $\Rc_g$ denotes the Ricci curvature of the Riemannian metric $g$.

We intend to study uniform behaviors of complete non-compact $4$-d Ricci
shrinkers through their moduli spaces, whose compactification is of
foundamental importance. Poineered by the work of
Cao-\v{S}e\v{s}um~\cite{CaoSes07}, Xi Zhang~\cite{XZhang06}, Brian
Weber~\cite{Weber11}, Chen-Wang~\cite{CW12} and Zhenlei Zhang~\cite{ZLZhang10} in this direction, the most satisfactory
compactness results to date, obtained by Robert Haslhofer and Reto M\"uller
(n\'e Buzano)~\cite{HM10}~\cite{HM14}, assume a uniform entropy lower bound.
In fact, Bing Wang has conjectured that a $4$-d Ricci shrinker should have an
\emph{a priori} entropy lower bound, depending solely on some topological
restrictions. (See Conjecture~\ref{thm: conjecture_entropy}.) To confirm this
conjecture, however, we need to study the degeneration of the metrics along
sequences of $4$-d Ricci shrinkers \emph{without} uniform entropy lower bound,
and then use contradiction arguments to rule out the potential occurrence of
such a situation. (For the relation between entropy lower bound and no local
collapsing property, see~\cite{Perelman02} and~\cite{KL08}.)

The obvious analogy between Ricci solitons and Einstein manifolds brings us the
foundational work of Cheeger-Tian~\cite{CT05}, which, built on Anderson's
$\epsilon$-regularity with respect to collapsing~\cite{A92}, obtains a new
$\epsilon$-regularity theorem for any four dimensional Einstein manifolds.
Cheeger-Tian conjectured (in Section 11 of~\cite{CT05}) that a similar result
should hold for four dimensional Ricci solitons, moreover:
 $$\textit{``Of particular interest is the case of shrinking Ricci
 solitons.''}$$

Our first theorem confirms their conjecture for 4-d Ricci shrinkers:
\begin{theorem}[$\epsilon$-regularity for 4-d Ricci shrinkers]
Let $(M^4,g,f)$ be a complete non-compact four dimensional shriking Ricci
soliton and fix $R>0$. Then there exists $r_R>0$, $\epsilon_R>0$ and $C_R>0$,
such that for any $r\in (0, r_R)$ and $B(p,r)\subset B(p_0,R)$, the weighted
local $L^2$-curvature control
\begin{align*}
\int_{B(p,r)}|\Rm_g|^2\ e^{-f}\dvol_g\ \le\ \epsilon_R 
\end{align*}
implies the local boundedness of curvature 
\begin{align*}
\sup_{B(p,\frac{1}{4}r)}|\Rm_g|\ \le\ C_R\ r^{-2}.
\end{align*}
\label{thm: main}
\end{theorem}
\noindent Here $p_0\in M$ is a minimum point of $f$, see Lemma~\ref{lem:
potential_growth} for more details.

In order to further motivate our theorem, we notice that Cheeger-Tian's
$\epsilon$-regularity theorem could be viewed as a non-trivial localization of
the fact that for a closed four dimensional Einstein manifold $(M,g)$, its
Euler characteristic can be computed as
$$\chi(M)=\frac{1}{8\pi^2}\int_M|\Rm_g|^2\ \dvol_g.$$ If $\|\Rm_g\|_{L^2(M)}$
is sufficiently small, then the integrity of $\chi(M)$ will force $\chi(M)=0$,
whence the flatness of $(M,g)$.

Similarly, on a closed four dimensional Ricci soliton $(M,g,f)$, we have
$$\chi(M)=\frac{1}{8\pi^2}\int_M\left(|\Rm_g|^2-|\mathring{\Sc}c_g|^2\right)\
\dvol_g,$$ where $\mathring{\Sc}c_g$ is the traceless Ricci tensor. Now if
$\|\Rm_g\|_{L^2(M)}<\pi$, as $|\mathring{\Sc}c_g|^2\le 2|\Rm_g|^2,$ we must
have $-1<\chi(M)<1$, and thus $\chi(M)=0$. It follows that $(M,g,f)$ must be a
steady or an expanding Ricci soliton: otherwise, were $(M,g,f)$ a 4-d Ricci
shrinker, $\pi_1(M)$ must be finite by~\cite{WW}, leading to $\chi(M)\ge 2$
that contradicts the vanishing of $\chi(M)$. But a closed steady or expanding
four dimensional Ricci soliton must be Einstein, so $|\mathring{\Sc}c|\le
|\nabla^2f|\equiv 0$ by (\ref{eqn: defn}), and
$\|\Rm_g\|^2_{L^2(M)}=8\pi^2\chi(M)=0$, which means that $(M,g)$ must be flat.

So our $\epsilon$-regularity theorem for 4-d Ricci shrinkers is a localization
of the above rigidity of closed four dimensional Ricci solitons, and
particularly suits the study of non-compact ones. Notice however, as
pointed out in~\cite{HM10}, that \emph{``most interesting singularity models
are non-compact, the cylinder being the most basic example''.}

 We also need to notice that the dependence of the constants in our
 $\epsilon$-regularity theorem is a new feature caused by the presence of the
 potential function: the allowence of the existence of non-compact Ricci
 shrinkers --- in fact, they strongly resemble positive Einstein manifolds,
 which are compact, scaling rigid, and which, may only admit families of
 metrics that are either collapsing everywhere, or else nowhere. Similar
 phenomenon occuring for geodesic balls of fixed size centered at the base
 point in a non-compact 4-d Ricci shrinker, our $\varepsilon$-regularity
 theorem only applies within a fixed distance from the base point. Moreover, in
 our future presentations, we will fix such a distance and do not elaborate on
 writing down scaling invariant formulae.

The proof of Theorem~\ref{thm: main} is based on the recent advances in the
study of shrinking Ricci solitons, and the comparison geometry of Bakry-\'Emery
Ricci curvature lower bound (see, among
others,~\cite{HM10}~\cite{HM14},~\cite{Cao10},
~\cite{BLChen09},~\cite{Lott03},~\cite{MWang15},~\cite{WW}
and~\cite{WangZhu13}, etc.). Here we briefly outline the proof of
Theorem~\ref{thm: main}, which follows the strategy of Cheeger-Tian~\cite{CT05}
in the Einstein case. We will indicate the necessary improvements in order to
deal with the lack of the Einstein's equation.

\subsubsection*{\textbf{Starting point}}Our starting point is a 4-d Ricci
shrinker version of Anderson's $\varepsilon$-regularity with respect to
collapsing~\cite{A92}, see Proposition~\ref{prop: Anderson}. For any $r\le 1$
and $B(p,r)\subset B(p_0,R)$, let the renormalized energy of $B(p,r)$ be
defined as (see Definition~\ref{defn: renormalized_energy}) $$I_{\Rm}^f(p,r)\
:=\ \frac{\bar{\mu}_{R}(r)}{\mu_f(B(p,r))}\int_{B(p,r)}|\Rm|^2\ \dmu_f,$$ where
$\dmu_f:=e^{-f}\dvol_g$ and we will denote $\mu_f(U)=\int_U1\ \dmu_f$ for any
$U\subset M$. We notice that it is continuously increasing in $r$. Anderson's
theorem asserts the existence of positive constants $\varepsilon_A(R)$ and
$C_A(R)$, such that
\begin{align*} 
I_{\Rm}^f(p,r)\ \le\ \varepsilon_A(R)\ \Longrightarrow\
\sup_{B(p,\frac{r}{2})}|\Rm|\ \le\  C_A(R)\ r^{-2}I_{\Rm}^f(p,r)^{\frac{1}{2}}.
\end{align*}
However, the input of our $\varepsilon$-regularity theorem seems to be quite
far from fulfilling the smallness of $I_{|Rm|}^f(p,r)$ required by this
theorem: when collapsing happens, the smallness of the energy
$\int_{B(p,r)}|\Rm|^2\ \dmu_f$ may be caused by the smallness of
$\mu_f(B(p,r))$. This difficulty is overcome in two steps: firstly the key
estimate guarantees a uniform bound of $I_{|\Rm|}^f(p,2r)$ from the smallness
of $\int_{B(p,r)}|\Rm|^2\ \dmu_f$, as assumed in Theorem~\ref{thm: main}; then
the fast decay proposition guarantees that after a definite number, say $j_R$
times, of bisecting the given scale $r$, $I_{|\Rm|}^f(p,2^{1-j_R}r)$ is small
enough so that Anderson's theorem applies. Throughout the introduction we will
let $B(U,s)$ denote the $s$-tubular neighborhood around any set $U\subset M$,
and $A(U;s,r)=B(U,r)\backslash B(U,s)$ for $r>s>0$.

\subsubsection*{\textbf{Key estimate}}The key estimate (Proposition~\ref{prop:
KE}) follows from an interation argument, in each step of which, the energy
over a domain $U$ is roughly bounded by the $\frac{3}{4}$-power of the energy
on some $s$-tubular neighborhood of $U$, with some carefully chosen small $s\in
(0,r)$:
\begin{align}
\int_{U}|\Rm|^2\ \dmu_f\ \le\ C(R)\
\mu_f(B(U,s))\left(s^{-4}+\left(\frac{s^{-\frac{4}{3}}}{\mu_f(B(U,s))}\int_{B(U,s)}|\Rm|^2\ \dmu_f\right)^{\frac{3}{4}}\right),
\label{eqn: intro_0}
\end{align}
 see also the estimates (\ref{eqn: Rm_L2}) and (\ref{eqn: Rm_L2i}). Now we
 briefly explain how to obtain this estimate.

If $U$ is collapsing with locally bounded curvature (see Definition~\ref{def:
truncated_collapsing}), Cheeger-Tian proved in Section 2 of~\cite{CT05} that a
slightly larger neighborhood $U'$ of $U$ acquires a nilpotent structure, which
implies the vanishing of the Euler characteristic of $U'$:
\begin{align}
0\ =\ \chi(U')\ =\ \int_{U'}\Euler\ +\int_{\partial U'}\bEuler.
\label{eqn: intro_1}
\end{align}
For 4-d Ricci shrinkers,
$8\pi^2\Euler=\left(|\Rm|^2-|\mathring{\nabla}^2f|^2\right)\dvol_g$ since
$\mathring{\Sc}c_g=\mathring{\nabla}^2f$ by the defining equation (\ref{eqn:
defn}); and $\bEuler$ is a three form on $\partial U'$ with coefficients
determined by $\Rm_g|_{\partial U'}$ and $II_{\partial U'}$, the second
fundamental form of $\partial U'$, see (\ref{eqn: boundary_Pfaffian}). The
integral of $|\Rm|^2-|\mathring{\nabla}^2f|^2$ over $U'$, using (\ref{eqn:
intro_1}), is then pushed to the boundary integral $\int_{\partial U'}\bEuler$.

The control of  $\int_{\partial U'}\bEuler$ relies on the equivariant good
chopping theorem, (stated and used in Theorem 3.1 of~\cite{CT05}, also see
Appendix A for a detailed proof,) which enables us to choose $U'$ so that
$\partial U'$ is saturated by the nilpotent structure, and essentially bounds
$|II_{\partial U'}|$ by $|\Rm|^\frac{1}{2}$. It follows that $|\bEuler|$ is
then controlled by $|\Rm|^{\frac{3}{2}}$, which, via (\ref{eqn: intro_1}),
improves the integration of $|\Rm|^2$ over $U'$ to integrating
$|\Rm|^{\frac{3}{2}}$ over $\partial U'$. Averaging on a tubular neighborhood
of $\partial U$ and using a maximal function argument, Cheeger-Tian then
obtained:
\begin{align}
\left|\int_{\partial U'}\bEuler\right|\le C(R)\mu_f(A(U;0,s))
\left(s^{-4}+\left(\frac{s^{-\frac{4}{3}}}{\mu_f(A(U;0,s))}\int_{A(U;\frac{1}{4}s,\frac{3}{4}s)}|\Rm|^2\ \dvol_g\right)^{\frac{3}{4}}\right).
\label{eqn: intro_2}
\end{align}
 Invoking (\ref{eqn: intro_1}), we obtain a control of
 $\int_{U}|\Rm|^2-|\mathring{\nabla}^2f|^2\ \dmu_f$ by the right-hand side of
 the above estimate, since $\dmu_f$ is comparable to $\dvol_g$ in $B(p_0,R)$ in
 a uniform way.

In the Einstein case, since $|\nabla^2f|\equiv 0$, the above dominating term on
the right-hand side suffices to provide the desired control of
$\int_{U}|\Rm|^2\ \dmu_f$ in (\ref{eqn: intro_0}). For 4-d Ricci shrinkers,
however, $|\mathring{\nabla}^2 f|^2$ does not vanish and the control of
$\int_{U}|\nabla^2 f|^2\ \dmu_f$ relies on the gradient estimate $|\nabla f|\le
R\slash 2+\sqrt{2}$ (see Lemma~\ref{lem: gradient_estimate}), as well as
Cheeger-Colding's cut-off function (see Lemma~\ref{lem: cutoff}):
\begin{align}
\int_{U}|\nabla^2f|^2\ \dmu_f\ \le\ C(R)\ \mu_f(B(U,s))\ s^{-2},
\label{eqn: intro_3}
\end{align}
see Lemma~\ref{lem: Hessian_control}. When $s>0$ is very small, the right-hand
side of this estimate is dominated by $\mu_f(B(U,s))\ s^{-4}$, so replacing
$\mu_f(A(U;0,s))$ by $\mu_f(B(U,s))$ in (\ref{eqn: intro_2}), we could obtain
the desired energy estimate (\ref{eqn: intro_0}) for the iteration argument.

\subsubsection*{\textbf{Fast decay}} The fast decay proposition
(Proposition~\ref{prop: decay}) asserts the existence of some gap $\eta_R\in
(0,1)$, such that if the energy $\int_{B(p,r)}|\Rm|^2\ \dmu_f$ and volume
$\mu_f(B(p,r))$ of a ball $B(p,r)$ is sufficiently small, then
$$I_{\Rm}^f(p,r\slash 2)\ \le\ (1-\eta_R)\ I_{\Rm}^f(p,r).$$ This is proved by
a contradiction argument. Suppose on the contrary, for positive $\eta\to 0$,
there are counterexamples $B(p,r)\subset B(p_0,R)$ of vanishing $\mu_f$-volume
$\mu_{f}B(p,r)\to 0$ and $$I_{\Rm}^{f}(p,r\slash 2)\ >\ (1-\eta)\
I_{\Rm}^{f}(p,r),$$ we could use volume comparison to see that $A(p;r\slash
2,r)$ is an almost $\mu_{f}$-volume annulus (\ref{eqn: volume_cone}). By the
theory of Cheeger-Colding (Lemma~\ref{lem: Cheeger-Colding}), this property
implies that a slightly smaller annular region in $A(p;r\slash 2,r)$ is an
almost metric cone, whose radial distance approximated by some smooth function
$\tilde{u}$. The approximation is in the $C^0$-sense, as well as the
\emph{average} $H^2$-sense, see (\ref{eqn: C_0}) -- (\ref{eqn: H_2}).

Moreover, the key estimate implies the almost vanishing (\ref{eqn:
annulus_curvature_vanishing}) and regularity (\ref{eqn: annulus_regularity}) of
the curvature on the annulus $A(p;r\slash 2,r)$, and all derivative controls of
$\tilde{u}$, see (\ref{eqn: u_i_regularity}).

Let $W=B(x,3r\slash 2)\cup \tilde{u}^{-1}(r\slash 2, a)$ for some regular value
$a\in (3r\slash 4, r)$ of $\tilde{u}$. On the one hand, $\int_W|\Rm|^2\
\dvol_g$ is positive but very small (as assumed by the $\varepsilon$-regularity
theorem), say $$0\ <\ \int_W|\Rm|^2\ \dvol_g\ <\ \frac{1}{4};$$ on the other
hand, $\partial W=\tilde{u}^{-1}(a)$ smoothly approximates the outer boundary
of an annulus $A_{\infty}$ in a flat cone. Intuitively, since the cone is flat,
we know that the second fundamental form of its outer boundary,
$II_{\partial_+A_{\infty}}$, is positive, and its boundary Gauss-Bonnet-Chern
term $|\bEuler|_{\partial A_{\infty}}|\equiv 1$. Thus the smoothness of the
approximation $\tilde{u}^{-1}(a)\rightarrow \partial_+A_{\infty}$ together with
the vanishing of curvature (\ref{eqn: annulus_curvature_vanishing}) will imply
the positivity of coefficients of $\bEuler|_{\partial
W}=\bEuler|_{\tilde{u}^{-1}(a)}$, and the collapsing implies the smallness of
its integral, say
\begin{align}
0\ <\ \int_{\partial W}\bEuler\ <\ \frac{1}{4}.
\label{eqn: intro_4}
\end{align}
 In this way, for Einstein manifolds, we have obtained a smooth bounded domain
 $W$ whose Euler characteristic $\chi(W)$ satisfies $$0\ <\ \chi(W)\ =\
 \frac{1}{8\pi^2}\int_W|\Rm|^2\ \dvol_g+\int_{\partial W}\bEuler\ <\
 \frac{1}{2}.$$ This is impossible.

More specifically, in the Einstein case, Cheeger-Tian appealed to the theory of
Cheeger-Colding-Tian (see Theorem 3.7 of~\cite{CCT02}), which controls the
average error $|II_{\tilde{u}^{-1}(a)}-II_{\partial_+A_{\infty}}|$ on the level
set $\tilde{u}^{-1}(a)$, see (8.14) -- (8.19) of~\cite{CT05}. This was
implemented by lifting to a local covering, which is non-collapsing, and where,
since (\ref{eqn: H_1}) and (\ref{eqn: H_2}) are estimates of the
\emph{average}, similar estimates (\ref{eqn: covering_H_1}) and (\ref{eqn:
covering_H_2}) hold.

In the case of 4-d Ricci shrinkers, the control of $\int_W|\Rm|^2\ \dmu_f$ does
not impose a control of $\int_W\Euler$ directly, due to the presence of the
term $|\mathring{\nabla}^2f|^2$. However, we could further assume
$I_{|\Rm|}^f(p,r\slash 2)>\varepsilon_A(R)$, since otherwise there is no need
of proving the proposition. This assumption gives us a lower bound of
$\int_W|\Rm|^2\ \dmu_f$ proportional to $\mu_f(B(p,r))\ r^{-4}$. On the other
hand, recall that we have the estimate (\ref{eqn: intro_3}) of $|\nabla^2
f|^2$, with $s=4$ and $U=B(p,r)$. Thus whenever $r$ is sufficiently small,
$\int_W|\nabla^2 f|^2\ \dmu_f$ is dominated by the energy $\int_W|\Rm|^2\
\dmu_f$, so $$0\ <\ \frac{1}{8\pi^2}\int_W|\Rm|^2-|\mathring{\nabla}^2f|^2\
\dvol_g\ =\ \int_W\Euler.$$ This mainly accounts for the bound of scale $r_R$
in the statement of Theorem~\ref{thm: main}. The small upper bound of
$\int_W\Euler$ follows easily from the assumption of the
$\varepsilon$-regularity theorem.

In controlling the boundary Gauss-Bonnet-Chern integral $\int_{\partial
W}\bEuler$, we notice that the theory of Cheeger-Colding-Tian~\cite{CCT02} is
not available for 4-d Ricci shrinkers. (Though we expect a version of this
theory to hold in the case of Bakry-\'Emery Ricci curvature bounded below.) We
turn to the regularity (\ref{eqn: u_i_regularity}) of $\tilde{u}$: we could
find a fine enough net $\{x_j\}$ in a slightly smaller annulus contained in
$A(p;r\slash 2,r)$, such that at each point of the net we have
$$II_{\tilde{u}^{-1}(\tilde{u}(x_j))}(x_j)\ >\ \frac{1}{2}\ I_3,$$ where $I_3$
is the $3\times 3$-identity matrix; by the regularity of $\tilde{u}$ (\ref{eqn:
u_i_regularity}) and the closeness of points in the net, we could then obtain a
bound $$\forall a\in (0.7r,0.8r),\quad II_{\partial \tilde{u}^{-1}(a)}\ >\
\frac{1}{4}\ I_3.$$ This, together with the vanishing of the curvature
(\ref{eqn: annulus_curvature_vanishing}), give the desired point-wise
positivity and upper bound of coefficients of $\bEuler|_{\partial
\tilde{u}^{-1}(a)}$, for any $a\in (0.7r,0.8r)$. Integrating over $\partial
\tilde{u}^{-1}(a)$ (for some $a\in (0.7r,0.8r)$) and using the volume
collapsing of $\partial \tilde{u}^{-1}(a)$, we could obtain the desired bound
(\ref{eqn: intro_4}), see (\ref{eqn: good_principal_curvature_x_i}) --
(\ref{eqn: small_boundary}), thus concluding the proof of the fast decay
proposition.\\

Our $\epsilon$-regularity theorem sees a few applications in understanding the
moduli space of complete non-compact 4-d Ricci shrinkers. Our second theorem is
in this direction:
\begin{theorem}
Let $(M_i,g_i,f_i)$, be a sequence of complete non-compact 4-d Ricci shrinkers
with $\Sc_{g_i}\le \bar{S}$ and $|\chi(M_i)|\le \bar{E}$, then there exist
positive numbers $\bar{R}=\bar{R}(\bar{S})$ and
$\bar{J}=\bar{J}(\bar{E},\bar{S})$ together with the following data:
\begin{enumerate} 
\item a subsequence, still denoted by $(M_i,g_i,f_i)$,
\item marked points $\{p_i^1,\cdots,p_i^J\}\subset B(p^0_i,\bar{R})$ with $J\le
\bar{J}$, and
\item a length space $(X,d_{\infty})$ with marked points
$\{x_{\infty}^1,\cdots,x_{\infty}^J\}$,
\end{enumerate}
such that $(M_i,g_i,p_i^1,\cdots, p_i^J)\rightarrow
(X,d_{\infty},x^1_{\infty},\cdots,x_{\infty}^J)$ in the sense of strong
multi-pointed Gromov-Hausdorff convergence.
\label{thm: moduli}
\end{theorem}

\noindent Here we notice that $f$ has a global minimum point $p_0$
(see~\cite{CaoZhou10} and~\cite{HM10}), which will be our designated base point.
For more details about the \emph{``strong multi-pointed Gromov-Hausdorff
convergence''}, see Definition~\ref{defn: strong_convergence}. Notice that once
we are given a global upper bound of scalar curvature, then as we will show
later, $\chi(M)$ is a finite number, so our consideration is well-posed. The
condition on bounded Euler characteristic is topological in nature, while the
assumption on the scalar curvature, although being natural in the K\"ahler
setting, is technical and we hope to remove in our future work.

The paper is organized as following: after the preliminary results in Section
2, we will discuss in Section 3 the regularity and collapsing of 4-d Ricci
shrinkers with locally bounded curvature; Section 4 consists of the proof of
the $\epsilon$-regularity theorem for $4$-d Ricci shrinkers, while its
applications in studying the moduli space are in Section 5; conjectures raised
in the final section, the paper concludes with an appendix: a detailed
discussion of the good chopping theorem, under the context of collapsing with
locally bounded curvature.

\subsubsection*{\textbf{Notations}}
Throughout this paper the following notations are employed: 

\begin{enumerate}





\item $p_0\in M$ denotes the base point of $M$; also use $p_i^0\in M_i$ for a
sequence $\{M_i\}$.

\item $\Rm_g$, $\Rc_g$ and $\Sc_g$ denote the Riemannian curvature, the Ricci
curvature, and the scalar curvature of a given Riemannian metric $g$,
respectively. For the sake of simplicity, we will write $\Rm$, $\Rc$ and $\Sc$
when there is no confusion.

\item For any $E\subset M$ and $r>0$, define
$$
B(E,r):=\left\{x\in M:\ \exists y\in E,\ d(x,y)<r\right\}.
$$
For any $E\subset M$ and $0<r_1<r_2$, define
$$
A(E;r_1,r_2):=\left\{x\in M: \forall y\in E,\ d(x,y)>r_1,\ \text{and}\ \exists
z\in E,\ d(x,z)<r_2\right\}.
$$ 
Especially, $B(x,r)$ is the geodesic ball of radius $r$ around $x\in M$ and
$A(x;r_1,r_2)$ is the geodesic annulus around $x\in M$, with inner and outer
radii specified by $r_1$ and $r_2$ respectively.

\item $\Psi(\alpha,\beta\ |\ a,b,c)$ will denote some positive function
depending on $\alpha,\beta,a,b,c$ such that for any fixed $a,b,c$, $$
\lim_{\alpha,\beta\rightarrow 0}\Psi(\alpha,\beta\ |\ a,b,c)\ =\ 0.
$$
Notice that the specific value of $\Psi$ may change from line to line.

\item We will use bold-face letter to denote a vector in $\R^4$, e.g. the origin
is denoted by $\mathbf{0}$ and a vector is denoted by $\mathbf{v}$.
\end{enumerate}

\textbf{Acknowledgements:} \emph{I am grateful to Bing Wang for introducing me
the topics discussed in this paper, and his many encouragements. I am also
thankful to Jeff Cheeger who shared with me his insights on the subject and
helped me improve the exposition of this paper. Many thanks going to Gao Chen,
Xiuxiong Chen, Bing Wang, Yuanqi Wang and Ruobing Zhang for helpful discussions
on the subject, I would also like to thank Huai-Dong Cao for helpful comments
on the first version of this paper. Last by not least, I would like to thank an
anonymous referee for careful proofreading and valuable suggestions.}

\section{Preliminaries}
Given a 4-d Ricci shrinker $(M,g,f)$, in this section we record those properties needed throughout the paper. The results, except for the equations concerning the Euler characteristic, are valid in general dimensions, but we present them in the four dimensional setting for the sake of simplicity.
\subsection{Basic properties of 4-d Ricci shrinkers}
This subsection collects the differential equations satisfied on 4-d Ricci shrinkers, as well as the growth estimates of the potential function and volume. A good overall reference on topics covered here is Cao's Lecture notes~\cite{Cao10}.
\subsubsection{\textbf{Equations of the potential}} We start with taking trace of the defining equation (\ref{eqn: defn}) to get
\begin{align}
\Sc+\Delta f=2.
\label{eqn: scalar_f}
\end{align}
We also notice the fundamental observation due to Hamilton~\cite{Ham93} states that the quantity $\Sc+|\nabla f|^2-f$ is a constant on $M$, and in this paper we will make the following normalization for the potential function:
\begin{align}
\Sc+|\nabla f|^2=f.
\label{eqn: normalization}
\end{align}
Subtracting (\ref{eqn: normalization}) from (\ref{eqn: scalar_f}), we will get an elliptic equation of $f$ that does not involve any curvature term:
\begin{align}
\Delta f-|\nabla f|^2=2-f.
\label{eqn: elliptic}
\end{align}
This equation is of fundamental importance for our argument to obtain various estimates in later sections, since it gives a the Weitzenb\"ock formula of $f$:
\begin{align}
\Delta^f|\nabla f|^2\ =\ 2|\nabla^2f|^2-|\nabla f|^2,
\label{eqn: Bochner}
\end{align}
where the drifted Laplacian $\Delta^f:=\Delta-\nabla f\cdot \nabla$, and we used the defining equation (\ref{eqn: defn}), the elliptic equation (\ref{eqn: elliptic}), together with the equality $\nabla |\nabla f|^2\cdot \nabla f=2\nabla^2f(\nabla f,\nabla f)$,.

\subsubsection{\textbf{Equations of the curvature}} On the other hand, the curvature satisfies the following elliptic equations (see~\cite{PW10}):
\begin{align}
\Delta \Sc-\nabla f\cdot \nabla \Sc\ &=\ \Sc-2|\Rc|^2,
\label{eqn: scalar_laplace}\\
\Delta \Rc-\nabla f\cdot \nabla \Rc\ &=\ \Rc- 2\Rm\ast \Rc,\quad \text{and}\\
\Delta \Rm-\nabla f\cdot \nabla \Rm\ &=\ \Rm+\Rm\ast\Rm.
\label{eqn: Rm_laplace}
\end{align}
By the maximum principle applied to (\ref{eqn: scalar_laplace}), it was observed in~\cite{BLChen09}:
\begin{lemma}
$\Sc> 0$ unless $(M,g)$ is flat.
\label{lem: R_nonnegative}
\end{lemma}
Also see~\cite{LiWang17} for a uniform lower bound only depending on the entropy.

\subsubsection{\textbf{Equations of the Euler characteristic}} Moreover, on the topological side, the 4-dimensional Riemannian manifold $(M,g)$ has the localized Euler characteristic of any open subset $U\subset M$ expressed as
\begin{align}
\chi(U)=\int_U\Euler+\int_{\partial U}\bEuler,
\label{eqn: Euler}
\end{align}
provided that the integrals are defined.
Here the Pfaffian 4-form $\Euler$ is given by
\begin{align}
\Euler=\frac{1}{8\pi^2}\left(|\mathcal{W}|^2-\frac{1}{2}\left|\mathring{\mathcal{R}}c\right|^2+\frac{\Sc^2}{24}\right)\dvol_g
=\frac{1}{8\pi^2}\left(|\Rm|^2-\left|\mathring{\nabla}^2 f\right|^2\right)\dvol_g,
\label{eqn: Pfaffian}
\end{align}
where $\mathcal{W}$ is the Weyl tensor of $\Rm$, $\mathring{\nabla}^2f=\nabla^2 f-\frac{\Delta f}{4}g$ is the traceless Hessian of $f$, and we have used the defining equation (\ref{eqn: defn}) in the second equality.
For the boundary 3-form $\bEuler$, if we denote the area form of $\partial U$ by $\text{d}\sigma$, and let $\{e_i\}$ ($i=1,2,3$) be an orthonormal local frame tangent to $\partial U$ diagonalizing its second fundamental form $II_{\partial U}$, then we have
\begin{align}
\bEuler=\frac{1}{4\pi^2}\left(2k_1k_2k_3-k_1\mathcal{K}_{23}-k_2\mathcal{K}_{13}-k_3\mathcal{K}_{12}\right)\text{d}\sigma,
\label{eqn: boundary_Pfaffian}
\end{align}
where for $i,j=1,2,3$, $\mathcal{K}_{ij}=\Rm(e_i,e_j,e_j,e_i)$ is the sectional curvature along the tangent plane spanned by $e_i$ and $e_j$, and $k_i=II_{\partial U}(e_i,e_i)$ is the principal curvature of $\partial U$, see~\cite{HM10}.

\subsubsection{\textbf{Potential growth}} The potential function $f$ obeys a
very nice growth control by distance function both from below and above. This
was first proved by Cao-Zhou~\cite{CaoZhou10}. Here, we shall need the improved
version by Haslhofer-M\"uller~\cite{HM10}:
\begin{lemma}[Potential growth]
Let $(M,g,f)$ be a 4-d Ricci shrinker such that the normalization condition
(\ref{eqn: normalization}) is satisfied. Then there exists a point $p_0\in M$
where $f$ attains its infimum and
\begin{align}
\forall x\in M,\quad \frac{1}{4}\left(\max\{\rad(x)-20,0\}\right)^2\le f(x)\le
\frac{1}{4}\left(\rad(x)+2\sqrt{2}\right)^2,
\label{eqn: potential_growth}
\end{align}
where $\rad(x):=d(x,p_0)$. Moreover, all minimum points of $f$ is contained in
the geodesic ball $B\left(p_0,10+2\sqrt{2}\right)$.
\label{lem: potential_growth}
\end{lemma}

From the normalization (\ref{eqn: normalization}), the non-negativity of scalar
curvature Lemma~\ref{lem: R_nonnegative} and the above growth control of
potential (\ref{eqn: potential_growth}), we have the following gradient estimate:
\begin{lemma}[Gradient estimate for potential]
Let $(M,g,f)$ be a 4-d Ricci shrinker such that the normalization condition
(\ref{eqn: normalization}) is satisfied, then
\begin{align}
|\nabla f|\le \frac{\rad}{2}+\sqrt{2}.
\label{eqn: gradient_estimate}
\end{align}
\label{lem: gradient_estimate}
\end{lemma}

\subsubsection{\textbf{Volume growth}} Moreover, the following control of
volume growth is discovered by~\cite{CaoZhou10} and~\cite{Mun09}:
\begin{lemma}
Let $(M,g,f)$ be a complete non-compact 4-d Ricci shrinker, then there exists
some constant $C_{CMZ}>0$, such that $\forall r>10,$ $$\quad
\vol_g(B(p_0,r))\le C_{CMZ} r^4.$$ Moreover, if $u$ is any function on $M$
satisfying $|u|\le Ae^{\alpha \rad^2}$ for some $\alpha \in [0,\frac{1}{4})$
and $A>0$, then $$\int_M|u|e^{-f}\dvol_g<\infty.$$ Especially, the weighted
volume of $M$ is finite, i.e.
$\int_Me^{-f}\dvol_g<\infty.$
\label{lem: volume_growth}
\end{lemma}
\subsection{Comparison geometry for 4-d Ricci shrinkers} 
Compared to Einstein manifolds, one drawback of Ricci solitons comes from the
lack of a uniform Ricci curvature lower bound, whence the lack of volume ratio
monotonicity. However, Ricci solitons do satisfy the Bakry-\'Emery Ricci
curvature bounds. If we define $\Rc_f:=\Rc+\nabla^2f$, then the defining
equation (\ref{eqn: defn}) becomes
\begin{align}
\Rc_f=\frac{1}{2}g,
\label{eqn: BER}
\end{align}
saying that the Bakry-\'Emery Ricci curvature of a 4-d Ricci shrinker is half
the metric tensor. This subsection explores the analogy of 4-d Ricci shrinkers
and manifolds with uniform Ricci lower bound, basic references
being~\cite{Lott03} and~\cite{WW}.
\subsubsection{\textbf{Weighted volume comparison}} The measure compatible with
the Bakry-\'Emery Ricci curvature is the weighted measure
$\text{d}\mu_g:=e^{-f}\dvol_g$, and according to~\cite{WW}, there stands a
weighted volume comparison theorem, the counterpart of the Bishop-Gromov volume
comparison for the Ricci lower bound case. For a 4-d Ricci shrinker $(M,g,f)$
viewed as a metric measure space $(M,g,\dmu_f)$, we define its comparison
metric measure space as following:
\begin{definition}[Metric measure space form]
For any $R>2\sqrt{2}$ fixed, define the metric measure space $M^4_R:=
(\mathbb{R}^4, g_{Euc}, \text{d}\bar{\mu}_R)$, the four dimensional Euclidean
space equipped with the weighted measure $\text{d}\volfm_R$. Here we define
$\text{d}\bar{\mu}_R(\mathbf{x}):=e^{R|\mathbf{x}|}\text{d}x^1\wedge\cdots\wedge
\text{d}x^4$ for any $\mathbf{x}=\left(x^1,x^2,x^3,x^4\right)^T\in
\mathbb{R}^4$. Also let the weighted volume of radius $r$ ball centered at the
origin of $M_R^4$ be defined as $$\volfm_R(r):=\int_{B(\mathbf{0},r)}1\
d\volfm_R.$$

 Moreover, define the area function $\volfm_R'(r):=2\pi^2e^{Rr}r^3$.
\end{definition} 

We immediately notice that 
\begin{align}
\omega_4r^4\le \mu_R(r)\le e^{R^2}\omega_4r^4,
\label{eqn: Euclidean_comparison}
\end{align}
where $\omega_4$ is the volume of the unit ball in the four dimensional
Euclidean space.

 The following monotonicity formula follows directly from~\cite{WW}.
\begin{lemma}[Monotonicity of area and volume ratio]
Let $(M,g,f)$ be a 4-d Ricci shrinker and fix $R>2\sqrt{2}$. For any $p\in
B(p_0,R)$ and any unit tangent vector $\mathbf{v}$ at $p$, let
$\mathcal{A}(\mathbf{v},r)$ be the area form of the geodesic sphere at
$\exp_p(r\mathbf{v})$, then
\begin{align}
  0<s<r<d(p,\partial B(p_0,R))\quad \Rightarrow\quad
  \frac{\area_f(\mathbf{v},r)}{\volfm'_R(r)}\ \le\
  \frac{\area_f(\mathbf{v},s)}{\volfm'_R(s)}.
\label{eqn: area_monotone}
\end{align}

Moreover, for any $B(p,r_1)\subset B(p,r_2)\subset B(p_0,R)$ and
$B(p,s_1)\subset B(p,s_2)\subset B(p_0,R)$,
\begin{align}
0<s_1< r_1\ \text{and}\ 0<s_2< r_2\quad \Rightarrow\quad
\frac{\mu_f(A(p;r_1,r_2))}{\volfm_R(r_2)-\volfm_R(r_1)}\ \le\
\frac{\mu_f(A(p;s_1,s_2))}{\volfm_R(s_2)-\volfm_R(s_1)}
\label{eqn: volume_monotone}
\end{align}
\label{lem: volume_comparison}
\end{lemma}
\begin{proof}
Recall that (\ref{eqn: gradient_estimate}) implies 
$$\sup_{B(p_0,R)}|\nabla f|\le \frac{R}{2}+\sqrt{2}.$$
When $R>2\sqrt{2}$, we have the radial derivative $\partial_r f\ge -R$ on
$B(p_0,R)$. Now we can apply (4.8) of~\cite{WW} directly to obtain (\ref{eqn:
area_monotone}). See also Theorem 3.1 and (4.3) of~\cite{WW}.

For (\ref{eqn: volume_monotone}), integrate (\ref{eqn: area_monotone}) along
geodesics gives the directional comparison
$$\frac{\int_{r_1}^{r_2}\area_f(\mathbf{v},t)\ \text{d}t}{\int_{r_1}^{r_2}\volfm'_R(t)\ \text{d}t}\ \le\ \frac{\int_{s_1}^{s_2}\area_f(\mathbf{v},t)\ \text{d}t}{\int_{s_1}^{s_2}\volfm'_R(t)\ \text{d}t},$$ and integrating the above inequality in $\mathbf{v}\in S_pM$ gives the desired inequality.
\end{proof}
Notice that the doubling property of the weighted measure follows easily from
the above monotonicity: for any $R>2\sqrt{2}$ fixed,
\begin{align}
B(p,2r)\subset B(p_0,R)\ \Rightarrow\ \mu_f(B(p,2r))\le C_D(R) \mu_f(B(p,r)),
\label{eqn: doubling}
\end{align}
with the doubling constant $C_D(R)=16e^{R^2}$.
\subsubsection{\textbf{Sobolev inequality}} Another important consequence of
the monotonicity (\ref{eqn: area_monotone}) is the segment inequality,
originally due to Cheeger-Colding~\cite{ChCo} for manifolds with uniform Ricci
lower bound. We will provide a proof here as this is the first time the segment
inequality appears in the context of Bakry-\'Emery Ricci curvature bounded
below.
\begin{lemma}[Segment inequality]
Let $(M,g,f)$ be a 4-d Ricci shrinker, and fix $R>0$. For any $U\subset
B(p_0,R)$ and any non-negative $u\in C^0(U)$, there is a constant
$C_{ChCo}(R)>0$ such that if a subset $A$ of $U$ sees almost all pairs of its
points connected by minimal geodesics contained in $U$, then
\begin{align}
\int_{A\times A}\mathcal{F}_u(x,y)\ \dmu_f(x)\dmu_f(y)\ \le\ C_{ChCo}(R)\
\mu_f(A) \diam U\int_{U}u\ \dmu_f,
\label{eqn: segment}
\end{align}
where $$\mathcal{F}_u(x,y):=\inf_{\gamma_{xy}}\int_0^{d(x,y)}u(\gamma_{xy}(t))\
\text{d}t,$$ the infimum being taken over all minimal geodesics $\gamma_{xy}$
connecting $x$ and $y$.
\end{lemma}
\begin{proof}
We may consider $\mathcal{F}_u(x,y)=\mathcal{F}_u^+(x,y)+\mathcal{F}_u^-(x,y)$
where 
$$\mathcal{F}_u^+(x,y):=\inf_{\{\gamma_{xy}\}}\int_{\frac{d(x,y)}{2}}^{d(x,y)}u(\gamma_{xy}(f))\ \text{d}t\quad \text{and}\quad \mathcal{F}_u^-(x,y):=\inf_{\{\gamma_{xy}\}}\int_0^{\frac{d(x,y)}{2}}u(\gamma_{xy}(f))\ \text{d}t.$$ Since $\mathcal{F}^+_u(x,y)=\mathcal{F}^-_u(y,x)$, by Fubini's theorem, $$\int_{A\times A}\mathcal{F}_u^+(x,y)\ \dmu_f(x)\dmu_f(y)\ =\ \int_{A\times A}\mathcal{F}_u^-(x,y)\ \dmu_f(x)\dmu_f(y),$$
and so we only need to do the estimate for $\mathcal{F}_u^+$. For any $x\in A$
and any $\mathbf{v}\in S_xM$ fixed, define $d_{x,\mathbf{v}}:=\min\{t>0:
\exp_x(t\mathbf{v})\in \partial U\}$, also denote
$\gamma_{\mathbf{v}}(t)=\exp_x(t\mathbf{v})$. Then $\forall t\in
(0,d_{x,\mathbf{v}})$, by the area ratio monotonicity (\ref{eqn:
area_monotone}),
\begin{align*}
\mathcal{F}_u^+(\gamma_{\mathbf{v}}(t\slash 2),\gamma_{\mathbf{v}}(t))\
\dmu_f(\gamma_{\mathbf{v}}(t))\ &\le\
\left(\int_{\frac{t}{2}}^tu(\gamma_{\mathbf{v}}(s))\ \text{d}s\right)\
\area_f(\mathbf{v},t)\text{d}t\\
&\le\  8e^{R^2}\left(\int_{\frac{t}{2}}^tu(\gamma_{\mathbf{v}}(s))\
\area_f(\mathbf{v},s)\text{d}s\right)\ \text{d}t.
\end{align*}
By the assumption on $A\subset U$, for almost every $y\in A$, there exists some
$\mathbf{v}\in S_xM$ such that $\gamma_{\mathbf{v}}(d(x,y))=y$, we have
\begin{align*}
\int_{A}\mathcal{F}_u^+(x,y)\ \dmu_f(y)\ &\le\
\int_{S_xM}\int_0^{d_{x,\mathbf{v}}}
\mathcal{F}_u^+(\gamma_{\mathbf{v}}(t\slash
2),\gamma_{\mathbf{v}}(t))\area_f(\mathbf{v},t)\ \text{d}t\text{d}\mathbf{v}\\
&\le\ 8e^{R^2}\diam U
\int_{S_xM}\int_0^{d_{x,\mathbf{v}}}u(\gamma_{\mathbf{v}}(s))\
\area_f(\mathbf{v},s)\ \text{d}s\text{d}\mathbf{v}\\
&\le\ 16e^{R^2}\diam U \int_{U}u\ \dmu_f.
\end{align*}
Finally, integrate the above inequality for $x\in A$, we get
\begin{align*}
\int_{A}\int_{A}\mathcal{F}_u^+(x,y)\ \dmu_f(y)\dmu_f(x)\ &\le\
16e^{R^2}\mu_f(A)\diam U \int_U u\ \dmu_f.
\end{align*}
\end{proof}

 Iterating the segment inequality, one easily obtains the local
 $L^2$-Poincar\'e inequality, whose constants are determined by $C_{ChCo}(R)$:
\begin{lemma}[Poincar\'e inequality]
Let $(M,g,f)$ be a 4-d Ricci shrinker and fix $R>2\sqrt{2}$. There exists a
positive constant $C_P(R)>0$ such that for any $B(p,r)\subset B(p_0,R)$ and any
$u\in C^1(B(p,r))$,
\begin{align}
\int_{B(p,r)}\left|u-\fint_{B(p,r)}u\ \dmu_f\right|^2\ \dmu_f\ \le\ C_P(R)\
r^2\int_{B(p,r)}|\nabla u|^2\ \dmu_f.
 \label{eqn: Poincare}
\end{align}
\end{lemma}

It is well-know that the doubling property (\ref{eqn: doubling}) and the
$L^2$-Poincar\'e inequality (\ref{eqn: Poincare}) implies a local Sobolev
inequality, with whose constants are determined by $C_D(R)$ and $C_P(R)$,
see~\cite{SCoste92}:
\begin{lemma}[Sobolev inequality]
 Let $(M,g,f)$ be a 4-d Ricci shrinker and fix $R>2\sqrt{2}$. For any
 $B(p,r)\subset B(p_0,R)$ and $u\in C^1_c(B(p,r))$,
\begin{align}
\left(\int_{B(p,r)}u^4\ \dmu_f\right)^{\frac{1}{2}}\le \frac{C_S(R)\
r^2}{\mu_f(B(p,r))^{\frac{1}{2}}}\int_{B(p,r)}\left(|\nabla
u|^2+r^{-2}u^2\right)\ \dmu_f.
\end{align}
\label{lem: Sobolev}
\end{lemma}

\subsubsection{\textbf{Cheeger-Colding theory}} The theory of
Cheeger-Colding~\cite{ChCo}~\cite{ChCoII} provides powerful tools in studying
the structure of manifolds with uniform lower Ricci bounds. In the context of
lower bounded Bakry-\'Emery Ricci curvature, a similar theory has been
developed in~\cite{WangZhu13}, where the study is focused on
\emph{non-collapsing} manifolds. Yet our major concern is the \emph{collapsing}
phenomenon. Still, some of their lemmas see a few applications in our situation.

The existence of a cut-off function with controlled gradient and Laplacian will
play a fundamental role in our local $L^2$-Ricci curvature estimate.
In~\cite{WangZhu13}, such a cut-off function on a \emph{unit ball} has been
constructed following~\cite{ChCo}. However, noticing that the equation
(\ref{eqn: BER}) is not scaling invariant, we need a more careful argument when
dealing with the general case, see also~\cite{CoNa11}.
\begin{lemma}[Existence of good cut-off function]
For any $R>10$, there is a constant $C(R)>0$ such that for any $r\in (0,1)$,
and any compact $K\subset B(p_0,R-r)$, there is a smooth cut-off function
$\varphi$ supported on $B(K,r)$, with $\varphi\equiv 1$ on $B(K,\frac{r}{2})$,
$\varphi\equiv 0$ outside $B(K,\frac{3r}{4})$, and $r|\nabla
\varphi|+r^2|\Delta^f \varphi|\le C(R)$.
\label{lem: cutoff}
\end{lemma}
\begin{proof}
Fix $r\in (0,0.1)$. When $K=\{x_0\}\subset B(p_0,R-r)$, the construction of
such a cut-off function originates in the work of Cheeger-Colding~\cite{ChCo},
and a Bakry-\'Emery version was constructed in~\cite{WangZhu13}. For shrinking
Ricci solitons, consider the rescaled metric $\tilde{g}=4r^{-2}g$, then
$\Rc_{\tilde{g}}+\tilde{\nabla}^2f=\frac{r^2}{8}\tilde{g}$, or the
Bakry-\'Emery Ricci curvature satisfies
$\Rc_{\tilde{g}}^f=\frac{r^2}{8}\tilde{g}\ge 0$ as symmetric two tensors.
Moreover, $|\tilde{\nabla} f|=\frac{r}{2}|\nabla f|\le R+2$ since $r<1$.
 Then we can apply Lemma 1.5 of~\cite{WangZhu13} to obtain a cut-off function
 $\varphi$ supported on $\tilde{B}(x_0,2)$, $\varphi\equiv 1$ on
 $\tilde{B}(x_0,\frac{5}{4})$ and $\varphi\equiv 0$ outside
 $\tilde{B}(x_0,\frac{7}{4})$, moreover
\begin{align}
|\tilde{\nabla} \varphi|+|\tilde{\Delta}^f \varphi|\le C(R).
\label{eqn: rescaled_cutoff}
\end{align} 
Notice that the constant $C(R)$ depends on the lower Bakry-\'Emery Ricci
curvature bound, which is $0$, thus scaling invariant, and it also depends on
an uniform upper bound of $|\tilde{\nabla} f|$ on $\tilde{B}(p_0,r^{-1}R)$,
which is uniformly bounded above by $R+2$, regardless of the scaling by $r$ as
long as $r<1$.  In the original metric, (\ref{eqn: rescaled_cutoff}) reads
$r|\nabla \varphi|+ r^2|\Delta^f \varphi|\le C(R)$.

Now suppose $K\subset B(p_0,R-r)$, let a maximal subset of points
$\{x_i\}\subset B(K,\frac{r}{2})$ with $d(x_i,x_j)>\frac{r}{20}$. Then the
maximality implies that $B(K,\frac{r}{2})\subset \cup_iB(x_i,\frac{r}{10})$.
Moreover, if $x\in \cap_{j=1}^kB(x_{i_j},\frac{1}{5}r_{i_j})$, then by
Lemma~\ref{lem: volume_comparison}, relations $$B(x,r\slash 5)\ \subset\
B(x_{i_j},2r\slash 5)\ \subset\ B(x,3r\slash 5),\quad \text{and}\quad
B(x_{i_j},r\slash 40)\cap B(x_{i_{j'}},r\slash 40)\ =\ \emptyset,$$ bound the
multiplicity of the covering $\{B(x_i,\frac{r}{10})\}$ by some $m(R)$.

Then we use the first step of the lemma on each $B(x_i,\frac{r}{5})$, to
construct cutoff functions $\varphi_i$ supported on $B(x_i,\frac{r}{5})$ such
that $\varphi_i|_{B(x_i,\frac{r}{10})}\equiv 1$, and $r|\nabla
\varphi_i|+r^2|\Delta^f \varphi_i|\le c(R)$. Let
$\bar{\varphi}=\sum_i\varphi_i$, then $1\le \bar{\varphi}\le m(R)$ on
$B(K,\frac{r}{2})$, and vanishes outside $B(K,\frac{7r}{10})$. Let
$u:[0,\infty)\rightarrow [0,1]$ be a smooth function that vanishes near zero
and constantly equals one on $[1,\infty)$, then $\varphi=u(\bar{\varphi})$ is
the desired cutoff function.
\end{proof}

A fundamental tool of Cheeger-Colding theory is a controlled smoothing of the
distance function using solutions to the Poisson equations with prescribed
Dirichlet boundary conditions given by the distance function. In the case of
Bakry-\'Emery Ricci curvature uniformly bounded below, similar estimates were
obtained in~\cite{WangZhu13}:
\begin{lemma}
For any $\eta>0$ and $\epsilon>0$, let $(M,g,f)$ be a 4-dimensional smooth
Riemannian manifold with $\Rc^f\ge 0$ and $|\nabla f|\le \epsilon\ A$.
Suppose that 
$$\frac{\mu_f(\partial B(p,s))}{\mu_f(\partial B(p,r))}\ge
(1-\eta)\frac{\volfm_{\epsilon A}'(r)}{\volfm_{\epsilon A}'(s)},$$ and that $u$
solves the following Poisson-Dirichlet problem $$\Delta^fu=4\quad \text{on}\
\text A(p;r,s),\quad \quad u|_{\partial
B(p,r)}=\frac{r^2}{2}\quad\text{and}\quad u|_{\partial B(p,s)}=\frac{s^2}{2}.$$
Then for $r<r_1<r_2<s_2<s_1<s$, denoting $d^2_p(x):=d^2(p,x)$ and
$\tilde{u}:=\sqrt{2u}$, then $u$ and $\tilde{u}$ satisfies the following
estimates:
\begin{enumerate}
\item $\sup_{A(p;r_1,s_1)}|\tilde{u}-d_p|\ \le\ \Psi(\eta,\epsilon\ |\
A,r,s,r_1,s_1)$;
\item $\fint_{A(p;r,s)}|\nabla \tilde{u}-\nabla d_p|^2\ \dmu_f\ \le\
\Psi(\eta,\epsilon\ |\ A,r,s)$; and
\item $\fint_{A(p;r_2,s_2)}|\nabla^2 u-g|^2\ \dmu_f\ \le\ \Psi(\eta,\epsilon\
|\ A,r,r_1,r_2,s,s_1,s_2)$.
\end{enumerate}
\label{lem: Cheeger-Colding}
\end{lemma}
 Basically, this lemma states that when $f$ is approximately a constant
 function, the situation is reduced to the Ricci lower bound case and
 corresponding estimates follow from the work of Cheeger-Colding~\cite{ChCo}.
\subsubsection{\textbf{Anderson's theorem}} Anderson's $\epsilon$-regularity
with respect to collapsing~\cite{A92} is the starting point of Cheeger-Tian's
$\epsilon$-regularity theorem for four dimensional Einstein manifolds~\cite{CT05}.
By the bound on the Sobolev constant for $\dmu_f$, as obtained in
Lemma~\ref{lem: Sobolev}, the proof of this theorem is by now standard using
Moser iteration, see~\cite{A92} and~\cite{HM10} for the original work.
\begin{proposition}[Weighted $\epsilon$-regularity with respect to collapsing] 
Let $(M,g,f)$ be a 4-d Ricci shrinker. There exist $\epsilon_A(R)>0$ and
$C_A(R)>0$ such that if $B(p,r)\subset B(p_0,R)$, then
\begin{align}
\frac{\volfm_R(r)}{\mu_f(B(p,r))}\int_{B(p,r)}|\Rm|^2\ \dmu_f\ \le\ \epsilon_A(R)
\label{eqn: small_Rm_L2}
\end{align}
implies that 
$$\sup_{B(p,\frac{r}{2})}|\Rm|\ \le\ C_A(R)\
r^{-2}\left(\frac{\volfm_R(r)}{\mu_f(B(p,r))}\int_{B(p,r)}|\Rm|^2\
\dmu_f\right)^{\frac{1}{2}}.$$
\label{prop: Anderson}
\end{proposition}

This proposition basically says that even if a geodesic ball has no uniform
volume lower bound, and consequently no uniform estimate from the Sobolov
inequality, when the local energy is sufficiently small --- much smaller
compared to the volume --- we still have uniform curvature control. Adapted to
this phenomenon, we define the ``renormalized energy" as following:
\begin{definition}
Fix $r\in (0,1]$. For any $p\in B(p_0,R)$, define the scale $r$ renormalized
energy as
$$I^f_{\Rm}(p,r):=\frac{\volfm_{R+1}(r)}{\mu_f(B(p,r))}\int_{B(p,r)}|\Rm|^2\
\dmu_f.$$
\label{defn: renormalized_energy}
\end{definition}
So Proposition~\ref{prop: Anderson} says that for $p\in B(p_0,R)$, 
\begin{align}
I_{\Rm}^f(p,r)<\epsilon_A(R) \Rightarrow \sup_{B(p,\frac{r}{2})}|\Rm|\ \le\
C_A(R)\ r^{-2}I_{\Rm}^f(p,r)^{\frac{1}{2}}.
\label{eqn: Anderson}
\end{align}
Moreover, we immediately notice the following key properties of the
renormalized energy:
\begin{enumerate}
\item $I_{\Rm}^f$ is invariant under rescaling, so is (\ref{eqn: Anderson});
\item  $I^f_{\Rm}$ is continuous and monotonically \emph{non-decreasing} in
radius $r$.
\end{enumerate}

\subsection{Convergence and collapsing of Riemannian manifolds} 
In this subsection, we start by introducing various convergence concepts of metric spaces, whose canonical reference is~\cite{GLP}, then discuss Fukaya's structural results about the collapsing limit under bounded curvature, see~\cite{Fukaya87b} and~\cite{Fukaya88}. See also~\cite{GMR90} for a relevant result concerning the local structure of Riemannian manifolds.
\subsubsection{\textbf{Weak convergence}} Given a sequence of metric spaces $(X_i,d_i)$ with diameter bounded above by $R$, we say that $(X_i,d_i)\rightarrow_{GH}(X_{\infty,d_{\infty}})$ if when $i\rightarrow \infty$, the Gromov-Hausdorff distance, $d_{GH}((X_i,d_i),(X_{\infty},d_{\infty}))\rightarrow 0$. Recall that $d_{GH}((X_i,d_i),(X_{\infty},d_{\infty}))$ is defined as the infimum of the Hausdorff distance between $X$ and $Y$ in $X\sqcup Y$, equipped with all possible metrics. If $(X_i,d_i)\rightarrow_{GH}(X_{\infty,d_{\infty}})$, we could then find maps $G_i: X_i\rightarrow X_{\infty}$ and $H_i:X_{\infty}\rightarrow X_i$ such that for any $\epsilon>0$, there exists some $i_{\epsilon}$ so that $\forall i>i_{\epsilon}$, $\forall x_i,x_i'\in X_i$ and $\forall x_{\infty},x'_{\infty}\in X_{\infty}$,
\begin{enumerate}
\item $\left|d_i(x_i,x_i')-d_i(H_i\circ G_i(x_i),H_i\circ G_i(x_i'))\right|\ <\ \epsilon$, and
\item $\left|d_{\infty}(x_{\infty},x_{\infty}')-d_{\infty}(G_i\circ H_i(x_{\infty}),G_i\circ H_i(x_{\infty}'))\right|\ <\ \epsilon$.
\end{enumerate}

Gromov's fundamental observation says that if $\{(X_i,d_i)\}$ has uniformly bounded Hausdorff dimension, diameter and volume doubling property, then there exists some metric space $(X_{\infty},d_{\infty})$ with the same diameter bound, such that a subsequence Gromov-Hausdorff converges to $(X,d)$. Notice that if $(X_i,d_i)\subset B(p_i^0,R)\subset (M_i,g_i,f_i)$ with $d_i$ induced by $g_i|_{X_i}$, then by the uniform doubling property (\ref{eqn: doubling}) for $\mu_{f_i}$:
\begin{lemma}
Suppose $\{(X_i,d_i)\subset B(p_i^0,R)\subset (M_i,g_i,f_i)\}$ is a sequence of uniformly bounded domains in 4-d Ricci shrinkers, possibly with marked points, then there exists a metric space $(X_{\infty},d_{\infty})$ with $\diam_{d_{\infty}}X_{\infty}\le R$, such that some subsequence, still denoted by $\{(X_i,d_i)\}$, Gromov-Hausdorff converges to $(X_{\infty},d_{\infty})$.
\end{lemma}
For a sequence of complete non-compact 4-d Ricci shrinkers, we may define the multi-pointed Gromov-Hausdorff convergence to respect the specified base point, i.e. a minimum of the potential function.
\begin{definition}
We say that a sequence of complete non-compact 4-d Ricci shrinkers $(M_i,g_i,f_i,p_i^0)$ with base points $p_i^0$ (a minimum of $f_i$) and $J$ marked points $Mk_i:=\{p_i^1,\cdots,p_i^J\}$ multi-pointed Gromov-Hausdorff converges to a metric space $(X_{\infty},d_{\infty},x_{\infty}^0)$ with $J$ marked points $Mk_{\infty}=\{x_{\infty}^1,\cdots,x_{\infty}^J\}$, if for any $R>0$, $B(p_i^0,R)\rightarrow_{GH}B(x_{\infty}^0,R)$, and there are maps $G_i: M_i\rightarrow X_{\infty}$ and $H_i:X_{\infty}\rightarrow M_i$ such that $G_i(p_i^j)=x_{\infty}^j$ and $H_i(x_{\infty}^j)=p_i^j$ ($j=0,1,\cdots,J$). Moreover, for any $\epsilon>0$, there exists some $i_{\epsilon}(R)$ so that $\forall i>i_{\epsilon}(R)$,
\begin{enumerate}
\item  $\forall p_i,p_i'\in B(p_i^0,R)\backslash Mk_i$, $\left|d_i(p_i,p_i')-d_i(H_i\circ G_i(p_i),H_i\circ G_i(p_i'))\right|\ <\ \epsilon$, and
\item  $\forall x_{\infty},x_{\infty}'\in B(x^0_{\infty},R)\backslash Mk_{\infty}$, $\left|d_{\infty}(x_{\infty},x_{\infty}')-d_{\infty}(G_i\circ H_i(x_{\infty}),G_i\circ H_i(x_{\infty}'))\right|\ <\ \epsilon$.
\end{enumerate}
\end{definition}
\noindent For convenience we will also use the notation $X_i\rightarrow_{pGH}X_{\infty}$ and $Mk_i\rightarrow_{GH}Mk_{\infty}$ for such type of convergence. Also notice, it is possible that $p_i^0\in Mk_i$.


\subsubsection{\textbf{Strong convergence}} Gromov's compactness result provides a weak limit in the category of metric spaces. In order to extract information from a convergent sequence, we need to consider stronger convergence. For a sequence of 4-d Ricci shrinkers $\{M_i,g_i,f_i\}$, suppose $\{(X_i,d_i)\subset (M_i,g_i,f_i)\}$ multi-pointed Gromov-Hausdorff converges to a limit space $(X_{\infty},d_{\infty})$, with marked points $Mk_i\rightarrow_{GH} Mk_{\infty}$. According to (a trivial generalization of) the work of~\cite{ChCoI} and~\cite{CoNa11}, $X_{\infty}\backslash Mk_{\infty}=\mathcal{R}(X_{\infty})\cup \mathcal{S}(X_{\infty})$, with $\dim_H(\mathcal{R}(X_{\infty}))\le 4$, and $\dim_H(\mathcal{S}(X_{\infty}))<\dim_H(\mathcal{R}(X_{\infty}))$. We define the strong convergence as following:
\begin{definition}[Strong convergence]
Let $(M_i,g_i,f_i)$ be a sequence of 4-d Ricci shrinkers, whose subsets $(X_i,d_i)\rightarrow_{pGH} (X_{\infty},d_{\infty})$, with $J$ marked points $Mk_i\rightarrow_{GH}Mk_{\infty}$. We say that the convergence is strong if there is an exhaustion of $X_{\infty}\backslash Mk_{\infty}$ by compact subsets $K_j$ ($j=1,2,3,\cdots$), such that for each $j$, there is an $i_j>0$ and for all $i>i_j$,
\begin{enumerate}
\item if $\dim_H\left(\mathcal{R}(X_{\infty})\right)=4$, then $\mathcal{S}(X_{\infty})=\emptyset$, $X_{\infty}$ is a smooth $4$-manifold, and each $H_i|_{K_j}$ can be chosen as a diffeomorphism onto its image, with $H_i^{\ast}g_i\rightarrow g_{\infty}$ smoothly as symmetric 2-tensor fields; or else,
\item if $\dim_H\left(\mathcal{R}(X_{\infty})\right)<4$, then each $G_i^{-1}(K_j)$ has uniformly bounded curvature $C_j$, and $G_i^{-1}(K_j)\rightarrow_{GH} K_j$ is collapsing with bounded curvature, in the sense of Cheeger-Fukaya-Gromov~\cite{CFG92}.
\end{enumerate}
\label{defn: strong_convergence}
\end{definition}
We notice that the two cases in the above definition are alternatives. Case (1) above is guaranteed to happen if a sequence has uniformly locally bounded curvature and uniform volume ratio lower bound, through the work of~\cite{Cheeger70}. 
See Theorem~\ref{thm: CFGT} for a more detailed description of case (2).

\subsubsection{\textbf{Collapsing with bounded curvature}} When collapsing with bounded curvature, i.e. case (2) in Definition~\ref{defn: strong_convergence}, happens, there is a rich structural theory about the Riemannian metric, mainly developed by Cheeger, Fukaya and Gromov, see~\cite{Gromov78}, \cite{Ruh}, \cite{CGI}, \cite{CGII}, \cite{Fukaya87a}, \cite{Fukaya88} and~\cite{CFG92}.
 The following proposition gives a full account of Fukaya's results in~\cite{Fukaya87b} and~\cite{Fukaya88} that are relevant to our argument in the following subsections:
\begin{proposition}[Structure of collapsing limit]
Let $X_i\subset (M^n_i,g_i)$ be bounded domains in a sequence of $n$-dimensional Riemannian manifolds such that
$$|\nabla^k \Rm_{g_i}|\ \le\ C_k\  (k=0,1,2,3,\cdots)\quad \text{on}\quad X_i.$$
Suppose $X_i\rightarrow_{GH}X_{\infty}$ for some metric space $(X_{\infty},d_{\infty})$, with $\dim_HX_{\infty}=m< n$, then there is a regular-singular decomposition $X_{\infty}=\mathcal{R}(X_{\infty})\cup \mathcal{S}(X_{\infty})$, such that 
\begin{enumerate}
\item $(\mathcal{R}(X_{\infty}),d_{\infty}) \equiv (\mathcal{R}(X_{\infty}),g_{\infty})$, a smooth $m$-dimensional Riemannian manifold, such that $$\sup_{\mathcal{R}(X_{\infty})}|\Rm_{g_{\infty}}|\le C_0;$$
\item $\mathcal{S}(X_{\infty})$ is a closed subset of $X_{\infty}$ with $\dim_H(\mathcal{S}(X_{\infty}))=m'\le m-1$;
\item there is a stratification $\emptyset\subset \mathcal{S}_0\subset \mathcal{S}_1\subset\cdots\subset \mathcal{S}_{m'}=\mathcal{S}(X_{\infty})$, each strata $\mathcal{S}_j$ is by itself a $j$-dimensional smooth Riemannian manifold;
\item there exists some $\iota_{X_{\infty}}\ >\ 0$ such that $\text{inj}_{\mathcal{R}(X_{\infty})}\ x\ =\ \min\{\iota_{X_{\infty}},d_{\infty}(x,\mathcal{S}(X_{\infty}))\}$, for any $x\in \mathcal{R}(X_{\infty})$.
\end{enumerate}

For all $i$ sufficiently large, the Gromov-Hausdorff approximation $G_i:X_i\rightarrow X_{\infty}$ can be chosen such that on $U_i:=G_i^{-1}(\mathcal{R}(X_{\infty}))$,
$$G_i: U_i\rightarrow \mathcal{R}(X_{\infty})$$
is an almost Riemannian submersion, and for each $x\in \mathcal{R}(X_{\infty})$, $G_i^{-1}(x)$ is diffeomorphic to $N$, an infranil-manifold.


\label{prop: Fukaya}
\end{proposition}


\subsection{Collapsing and local scales}
The collapsing of Riemannian manifolds could mean different things in different contexts. Our original concern (as stated in introduction) is about \emph{volume collapsing}, i.e. the manifold admitting a family of Riemannian metrics under which the volume of fix-sized metric balls approaches zero. If we assume uniformly bounded Riemannian curvature, then the volume collapsing is equivalent to \emph{collapsing with uniformly bounded curvature}, meaning that the injectivity radius of each point, under the family of metrics, approaches zero. When collapsing with bounded curvature happens, the structure theory of Cheeger-Fukaya-Gromov~\cite{CFG92} will be of great help in studying the underlying manifold.

\subsubsection{\textbf{Curvature scale}} In general, however, no \emph{a
priori} uniform curvature bound could be assumed. One then realizes that the
above mentioned structural theory about \emph{collapsing with uniformly bounded
curvature} could be localized if the metrics in consideration are regular. This
is because the curvature scale,
is locally 1-Lipschitz. See Section 3 for a detailed discussion about
Cheeger-Tian's localization adopted to the 4-d Ricci shrinkers, and here we
will focus on the basic properties of the curvature scale. See
also~\cite{Cheeger10} for an exposition of the theory of locally bounded
curvature and the curvature scale.
\begin{definition}[Curvature scale] For any $p\in M$, define 
$$r_{\Rm}(p):=\sup\left\{r>0: B(p,s)\ \text{has compact closure in}\ B(p,r),\
\text{and}\ \sup_{B(p,s)}|\Rm|\le s^{-2}\right\}.$$
\end{definition}

 Equivalently, $r_{\Rm}(p)$ is the maximal scale such that if one rescales the
 metric to make it unit size, then the rescaled curvature will have its norm
 uniformly bounded by $1$ on the resulting unit ball around $p\in M$.

In fact, $\forall x\in B(p,r_{\Rm}(p))$, we have
$B\left(x,r_{\Rm}(p)-d(p,x))\subset B(p,r_{\Rm}(p)\right)$, so
\begin{align}
\sup_{B\left(x,r_{\Rm}(p)-d(p,x)\right)}|\Rm|\le r_{\Rm}(p)^{-2}\le
\left(r_{\Rm}(p)-d(p,x)\right)^{-2},
\label{eqn: r_Rm_Lip}
\end{align}
and thus $d(x,p)<r_{\Rm}(p)$ implies that $r_{\Rm}(x)\ge r_{\Rm}(p)-d(p,x)$.
Reversing the role of $x$ and $p$, we have shown that the curvature scale is
locally 1-Lipschitz as mentioned above:
\begin{lemma}
Either $r_{\Rm}\equiv \infty$ and $\Rm\equiv 0$, or $r_{\Rm}$ is locally
Lipschitz with
\begin{align}
\text{Lip}\ r_{\Rm}\le 1.
\label{eqn: Rm_Lip}
\end{align}
\label{lem: Rm_Lip}
\end{lemma}
In order to facilitate our local arguments, it is also convenient to truncate
the curvature scale:
\begin{definition}[Truncated curvature scale]
For any fixed $0<r\le 	1$, we put
$$l_a:=\min\{r_{\Rm},a\}.$$
\end{definition}
\noindent Clearly $l_a$ is locally 1-Lipschitz.
\subsubsection{\textbf{Energy scale}} Associated to Anderson's theorem
(Proposition~\ref{prop: Anderson}) is another local scale, called the energy
scale. This scale is particularly well-adapted to the \emph{analytical} side of
the problem, and its interaction with the curvature scale, responsible for the
\emph{geometric} side of the problem, consists of the technical core of
Cheeger-Tian's argument.
\begin{definition}
The energy scale $\rho_f(p)$ is defined by
$$\rho_f(p):=\min\left\{\sup\left\{r\in (0,R):I^f_{\Rm}(p,r)\le\epsilon_A(R)\right\},1\right\}.$$
\end{definition}
\noindent Moreover, we could assume $\varepsilon_A(R)<4C_A(R)^{-2}$ in
Anderson's theorem (Proposition~\ref{prop: Anderson}), so that
$I^f_{\Rm}(p,\rho_f(p))\le \epsilon_A(R)$, and Proposition~\ref{prop: Anderson}
tells that
\begin{align}
\rho_f(p)\le 2r_{\Rm}(p),
\label{eqn: scale_comparison}
\end{align}
since 
$$
\sup_{B(p,\frac{1}{2}\rho_f(p))}|\Rm|\le C_A(R)\epsilon_A(R)\rho_f(p)^{-2}\le \left(\frac{1}{2}\rho_f(p)\right)^{-2}.
$$
\subsubsection{\textbf{Volume collapsing and collapsing with locally bounded curvature}}
As mentioned above, we are concerned with the phenomenon of volume collapsing defined as:
\begin{definition}[$\delta$-volume collapsing]
$U\subset B(p_0,R)$ is $\delta$-volume collapsing if $\forall p\in U$, $\mu_f(B(p,1))\le \delta$.
\end{definition}
However, volume collapsing does not give much information of the underlying geometry. The concept associated to localizing the structural theory of Cheeger-Fukaya-Gromov in~\cite{CFG92} is $(\delta,a)$-collapsing with locally bounded curvature:
\begin{definition}[$(\delta,a)$-collapsing with locally bounded curvature]
$U\subset B(p_0,R)$ is $(\delta,a)$-collapsing with locally bounded curvature if 
$\forall p\in U,\ \mu_f(B(p,l_a(p)))\le \delta\ l_a(p)^4.$
\label{def: truncated_collapsing}
\end{definition}
Anderson's $\epsilon$-regularity with respect to collapsing bridges these two concepts:
\begin{lemma}
Suppose for some $\delta\in (0,1)$, and $\forall p\in U\subset B(p_0,R)\subset M$, 
$$\mu_f(B(p,1))\le \frac{\delta}{16\volfm_R(1)}\quad \text{and}\quad \int_{B(p,1)}|\Rm|^2\ \dmu_f\le \frac{\epsilon_A(R)\ \delta}{16\volfm_R(1)},$$
then $U$ is $(\delta,a)$-collapsed with locally bounded curvature, i.e. $\forall p\in U$
\begin{align}
\mu_f\left(B(p,l_a(p))\right)\ \le\ \delta\ l_a(p)^4.
 \label{eqn: CT_local_bdd_curvature}
\end{align}
\label{lem: CT_local_bdd_curvature}
\end{lemma}

\begin{proof}[Proof (following Cheeger-Tian)]
Without loss of generality we only need to consider points with $r_{\Rm}\le 1$. 
If $\rho_f(p)=1$, then 
\begin{align*}
\mu_f(B(p,r_{\Rm}(p)))&\le \mu_f(B(p,2r_{\Rm}(p)))\le \frac{16\mu_f(B(p,1))\volfm_{R}(r_{\Rm}(p))}{\volfm_{R}(1)}\ \le\ \frac{\delta\ \volfm_R(r_{\Rm}(p))}{\volfm_{R}(1)^2}\\
&\le\ \frac{\delta\ r_{\Rm}(p)^4}{\volfm_{R}(1)}\ \le\ \delta\ r_{\Rm}(p)^4.
\end{align*}
Otherwise, if $\rho_f(p)<1$, and by continuity of $I^f_{\Rm}(p,r)$ in $r$, $I^f_{\Rm}(p,\rho_f(p))=\epsilon_A(R)$, and we can estimate
\begin{align*}
\mu_f(B(p,r_{\Rm}(p)))\ &\le\ \frac{16\mu_f(B(p,\rho_f(p)))\volfm_{R}(r_{\Rm}(p))}{\volfm_{R}(\rho_f(p))}\\
&=\ \frac{16\volfm_R(r_{\Rm}(p))}{\epsilon_A(R)}\int_{B(p,\rho_f(p))}|\Rm|^2\ \dmu_f\\
&\le\ \frac{\delta\ \volfm_R(r_{\Rm}(p))}{\volfm_R(1)}\ \le\  \delta\ r_{\Rm}(p)^4,
\end{align*}
in the case $r_{\Rm}(p)<a$, and a similar argument for $r_{\Rm}(p)\ge a$ implies (\ref{eqn: CT_local_bdd_curvature}).
\end{proof}
This lemma says that if we have sufficiently small energy, \emph{local volume collapsing} of a region does imply \emph{collapsing with locally bounded curvature}. 


\section{Regularity and collapsing with locally bounded curvature}
When collapsing with bounded curvature happens, Cheeger-Fukaya-Gromov~\cite{CFG92} gives a complete structural theory of the underlying manifold, one important consequence being the vanishing of the Euler characteristics. 
When the metric is locally regular, a similar structural theory could be obtained when collapsing with only \emph{locally} bounded curvature happens on a domain. This observation was essentially discovered in~\cite{CGII}, in the context of $F$-structures, and was made of full use in~\cite{CT05}. The vanishing of the Euler characteristic of the domain and (\ref{eqn: Euler}) then help obtain an improved energy bound (Proposition~\ref{prop: integral_TP_control}), which will be crucial for the iteration argument for the key estimate (Proposition~\ref{prop: KE}) later. In this section we will follow the expositions of Sections 2 and 3 of Cheeger-Tian~\cite{CT05} to see why their theory also works for 4-d Ricci shrinkers. The equivariant good chopping for sets collapsing with locally bounded curvature (Proposition~\ref{prop: GC}), which is the main theorem of Section 3 in~\cite{CT05}, is proved in the Appendix. 

\subsection{Elliptic regularity at the curvature scale}
Besides the fact that $r_{\Rm}$ is locally Lipschitz, another key ingredient in Cheeger-Tian's localization is that the higher regularities of Einstein metrics follow directly from local curvature bounds. This essentially follows from elliptic regularity theory and is independent of non-collapsing assumptions.

In the case of 4-d Ricci shrinkers, equations (\ref{eqn: scalar_f}) and (\ref{eqn: Rm_laplace}) form an elliptic system, which could be bootstrapped to give higher regularities of both the metric and the potential function, once a local curvature bound assumed. Also notice that according to (\ref{eqn: potential_growth}) and (\ref{eqn: gradient_estimate}), we already have a local $C^1$-bound of the potential function $f$. 
\begin{lemma}(Local elliptic regularity)
Let $p\in B(p_0, R)$, then there exists $C_k(R),D_k(R)$ such that 
\begin{align}
\sup_{B\left(p,\frac{1}{2}l_a(p)\right)}|\nabla^k \Rm|\ \le C_k(R)l_a(p)^{-2-k},\quad \text{and}\quad \sup_{B\left(p,\frac{1}{2}l_a(p)\right)}|\nabla^k f|\ \le\ D_k(R)l_a(p)^{-1-k},
\label{eqn: regularity}
\end{align}
for $k=0,1,2,3,\cdots$.
\label{lem: elliptic_regulairty}
\end{lemma}
\begin{proof}
Fix $p\in B(p_0,R)$, then $B(p,l_a(p))\subset B(p_0,R+1)$. Since $\sup_{B(p,l_a(p))}|\Rm|\le l_a(p)^{-2}$, the conjugate radius $r_{\text{conj}}$ has a definite lower bound on $B(p,l_a(p))$:
$$\inf_{B(p,l_a(p))} r_{\text{conj}}\ \ge\ \pi l_a(p).$$
This means that the exponential map $\exp_p:B(\mathbf{0},l_a(p))\rightarrow B(p,r_a(p))$ is well-defined and has no singularity. We can pull the manifold metric back to $B(\mathbf{0},l_a(p))\subset \mathbb{R}^4$, denote $\tilde{g}:=\exp_p^{\ast}g$ and $\tilde{f}:=\exp_x^{\ast}f$. Then the pull-back metric and potential function still satisfy the defining equation (\ref{eqn: defn}) 
$$Rc_{\tilde{g}}+\tilde{\nabla}^2\tilde{f}=\frac{1}{2}\tilde{g},$$
understood as matrix equations on an open subset of $\mathbb{R}^4$, with $\tilde{\nabla}^2$ the Hessian defined by the metric $\tilde{g}$. Notice that the equations (\ref{eqn: scalar_f}) and (\ref{eqn: Rm_laplace}) now become the elliptic system
\begin{align}
\tilde{\Delta} \tilde{f}=2-\Sc_{\tilde{g}}\quad \text{and}\quad
\tilde{\Delta} \Rm_{\tilde{g}}=\tilde{\nabla}\tilde{f}\ast \Rm_{\tilde{g}}+\Rm_{\tilde{g}}+\Rm_{\tilde{g}}\ast \Rm_{\tilde{g}},
\label{eqn: pullback_elliptic_system}
\end{align}
 defined on an open subset of $\mathbb{R}^4$, as equations of functions and of 4-tensors, respectively. Here $\tilde{\nabla}$ is the gradient under $\tilde{g}$ and $\tilde{\Delta}:=tr_{\tilde{g}}\tilde{\nabla}^2$ is the Laplacian of $\tilde{g}$. Moreover, since $\exp_p$ is an isometry, the local $C^1$-bounds (\ref{eqn: potential_growth}) and (\ref{eqn: gradient_estimate}) of $f$ translates as $\|\tilde{f}\|_{C^1(B(\mathbf{0},l_a(p)))}\le (R+1)^2$.

On the other hand, as in~\cite{GW88} and~\cite{AC92}, on $B(\mathbf{0},l_a(p))\subset \mathbb{R}^4$ we can use harmonic coordinates to deduce that $|\Rm_{\tilde{g}}|\le l_a(p)^{-2}$ implies $\|\tilde{g}\|_{C^{1,\alpha}}\le Cl_a(p)^{-1}$ on $B(\mathbf{0},Cl_a(p)\slash 2)$.

Then we can bootstrap to get that $\|\tilde{f}\|_{C^{k,\alpha}}\le D_k(R)l_a(p)^{-1-k}$ and $\|\Rm_{\tilde{g}}\|_{C^{k,\alpha}}\le C_k(R)l_a(p)^{-2-k}$ under harmonic coordinates. Since $\exp_p:(B(\mathbf{0},l_a(p)),\tilde{g})\rightarrow (B(p,l_a(p)),g)$ is an isometry, these estimates prove (\ref{eqn: regularity}). 
\end{proof}
\begin{remark}

As explained in~\cite{AC92}, given the results of~\cite{CheegerThesis}, the passage from a lower bound on the harmonic radius to a corresponding compactness theorem is immediate.
\end{remark}

It is straightforward to obtain the following elliptic regularity under rescaling:
\begin{lemma}[Rescaling]
Given $\lambda \in (0,1)$. The rescaling $g\mapsto \tilde{g}:=\lambda^{-2} g$ gives the equation $Rc_{\tilde{g}}+\nabla^2f=\frac{\lambda^2}{2}\tilde{g}$. Moreover, $r_{\Rm_{\tilde{g}}}=\lambda^{-1}r_{\Rm_g}$ and $\forall p\in B(p_0,R)$ we have:
$$
\sup_{\tilde{B}\left(p,\frac{1}{2\lambda}l_a(p)\right)}|\tilde{\nabla}^k\Rm_{\tilde{g}}|_{\tilde{g}}\ \le\ C_k(R)\left(\frac{l_a(p)}{\lambda}\right)^{-2-k}\quad\text{and}\quad \sup_{\tilde{B}\left(p,\frac{1}{2\lambda}l_a(p)\right)}|\tilde{\nabla}^kf|_{\tilde{g}}\ \le\ D_k(R)\left(\frac{l_a(p)}{\lambda}\right)^{-1-k}.
$$
\label{lem: regularity_rescaling}
\end{lemma} 

Moreover, for a general function solving the Poisson equation on a 4-d Ricci shrinker, we can argue similarly and obtain the following interior estimates under locally bounded curvature:
\begin{lemma}
Suppose $u\in C^2(B(p,l_a(p)))\subset B(p_0,R)$ solves $\Delta^f u=c$ for some constant $c$, then there are constants $C''_k(R,c)$ for $k=1,2,3,\cdots$, such that 
$$\sup_{B(p,\frac{1}{2}l_a(p))}|\nabla^k u|\ \le\ C''_k(R,c)\ l_a(p)^{-k}.$$
\label{lem: interior_regularity}
\end{lemma}

\subsection{Nilpotent structure and locally bounded curvature}
In this subsection, we will discuss why the main theorems of Sections 2 and 3 of~\cite{CT05} also work for 4-d Ricci shrinkers. Also see the Appendix for the proof of Proposition~\ref{prop: GC}.

We start with constructing a good covering, which sees a nice partition into sub-collections that makes the gluing arguments in~\cite{CFG92} and~\cite{CGIII} possible:
\begin{lemma}[Existence of a good covering]
Fix $a\le 1$. There is a covering of $E\subset M$ by geodesic balls with radius being a uniform multiple of the curvature scale, such that it can be partitioned into at most $N$ sub-collections $S_j$ ($j=1,\cdots,N$) of mutually disjoint balls in the covering, with any ball in a sub-collection intersecting at most one ball from another.
\label{lem: covering} 
\end{lemma} 
\begin{proof}
Let $\{p_i\}$ ($i= 1,2,3,\cdots$) be a maximal subset of $E$ satisfying 
\begin{align}
d(p_i,p_j)\ \ge\ \zeta \min\{l_a(p_i),l_a(p_j)\},
\label{eqn: distance>}
\end{align}
then for suitably chosen $\zeta\in (0,1)$, $\{B(p_i,2\zeta l_a(p_i)\}$ is a locally finite covering with uniformly bounded multiplicity. If $B(p_i,2\zeta l_a(p_i))\cap B(p_j,2\zeta l_a(p_j))\not=\emptyset$, then 
\begin{align*}
d(p_i,p_j)\ \le\ 4\zeta\max\{l_a(p_i),l_a(p_j)\}.
\end{align*}
Assuming $\zeta<\frac{1}{4}$, then as done in (\ref{eqn: r_Rm_Lip}), 
$$
\min\{l_a(p_i),l_a(p_j)\}\ \le \max\{l_a(p_i),l_a(p_j)\}\ \le\ \min\{l_a(p_i),l_a(p_j)\}+d(p_i,p_j),
$$
so we can estimate the distance 
\begin{align}
d(p_i,p_j)\ \le\ \frac{4\zeta}{1-2\zeta}\min\{l_a(p_i),l_a(p_j)\},
\label{eqn: distance<}
\end{align}
and thus 
\begin{align}
\min\{l_a(p_i),l_a(p_j)\}\ \le\ \max\{l_a(p_i),l_a(p_j)\}\ \le\ \frac{1+2\zeta}{1-2\zeta}\min\{l_a(p_i),l_a(p_j)\}.
\label{eqn: l_a_Harnack}
\end{align}

Now if $B(p_{i_0},2\zeta l_a(p_{i_0}))\cap B(p_{i_j},2\zeta l_a(p_{i_j}))\not=\emptyset$ for $j=1,\cdots,N(i_0)$, then by (\ref{eqn: distance<}) and (\ref{eqn: l_a_Harnack}), 
$$
d(p_{i_0},p_{i_j})\ \le\ \frac{4\zeta}{1-2\zeta}l_a(p_{i_0})\quad\text{and}\quad l_a(p_j)\ \ge\ \frac{1-2\zeta}{1+2\zeta}l_a(p_{i_0}),
$$
and thus we have the following containment relations: $\forall j=1,\cdots,N(i_0),$
\begin{align}
B\left(p_{i_0},\zeta l_a(p_{i_0})\right)\ \subset\ B\left(p_{i_j},\frac{5\zeta-2\zeta^2}{1-2\zeta}l_a(p_{i_0})\right)\ \subset\ B\left(p_{i_0},\frac{9\zeta-2\zeta^2}{1-2\zeta}l_a(p_{i_0})\right),
\label{eqn: covering}
\end{align}
while for $1\le j_1<j_2\le N(i_0)$, (\ref{eqn: distance>}) gives
\begin{align}
B\left(p_{j_1},\frac{\zeta(1-2\zeta)}{2(1+2\zeta)}l_a(p_{i_0})\right)\bigcap B\left(p_{j_2},\frac{\zeta(1-2\zeta)}{2(1+2\zeta)}l_a(p_{i_0})\right)\ =\ \emptyset.
\label{eqn: disjoint}
\end{align}
Let $\zeta<\frac{1}{20}$, and do the rescaling $g\mapsto l_a(p_{i_0})^{-2}g=:\tilde{g}$, then since $a\le 1$, 
$$
\sup_{\tilde{B}\left(p_{i_0},\frac{9\zeta}{1-2\zeta}\right)}|\Rm_{\tilde{g}}|_{\tilde{g}}\ \le\ 1.
$$
Now apply (\ref{eqn: covering}), (\ref{eqn: disjoint}) and volume comparison on $\tilde{B}(p_{i_0},10\zeta)$ we get
\begin{align*}
\vol_{\tilde{g}}\left(\tilde{B}\left(p_{i_0},\zeta\right)\right)\ &\le\ \frac{1}{N(i_0)}\sum_{j=1}^{N(i_0)}\vol_{\tilde{g}}\left(\tilde{B}\left(p_{i_j},\frac{5\zeta-2\zeta^2}{1-2\zeta}\right)\right)\\
&\le\ \frac{1}{N(i_0)}\Lambda_{-1}\left(\frac{5\zeta-2\zeta^2}{1-2\zeta}\right)\Lambda_{-1}\left(\frac{\zeta(1-2\zeta)}{2(1+2\zeta)}\right)^{-1}\vol_{\tilde{g}}\left(\tilde{B}\left(p_0, \frac{9\zeta-2\zeta^2}{1-2\zeta}\right)\right)\\
&\le\ \Lambda_{-1}\left(\frac{5\zeta-2\zeta^2}{1-2\zeta}\right)\Lambda_{-1}\left(\frac{9\zeta-2\zeta^2}{1-2\zeta}\right)\Lambda_{-1}\left(\frac{\zeta(1-2\zeta)}{2(1+2\zeta)}\right)^{-1}\frac{\vol_{\tilde{g}}\left(\tilde{B}\left(p_{i_0},\zeta\right)\right)}{\Lambda_{-1}(\zeta)\ N(i_0)},
\end{align*}
where $\Lambda_{-1}(r)$ is the volume of radius $r$ ball in a space form of constant curvature $-1$, and thus $N(i_0)\le N'$, a dimensional constant once we fix $\zeta\in (0,\frac{1}{20})$.

Now start with a maximal subset of $\{p_i\}$ with $d(p_i,p_j)>10\zeta \max\{l_a(p_i),l_a(p_j)\}$ denoted by $S_1$; then choose $S_2$ as a maximal subset of $\{p_i\}\backslash S_1$, etc. In this way we could obtain $S_1,\cdots,S_{N}$. Notice that for $k=1,2$, if there exist $p_{i_k}\in S_i$ and $p_{j}\in S_{j}$ satisfying $B(p_{i_k},2\zeta l_a(p_{i_k}))\cap B(p_j,2\zeta l_a(p_j))\not=\emptyset$, then 
by (\ref{eqn: distance<}) we have
$$10\zeta \max\{l_a(p_{i_1}),l_a(p_{i_2})\}\ \le\ d(p_{i_1},p_{i_2})\ \le\ \frac{8\zeta}{1-2\zeta}\max\{l_a(p_{i_1}),l_a(p_{i_2})\},$$
impossible for $\zeta<\frac{1}{20}$. Thus the ball centered at any element of $S_j$ can intersect with at most one ball centered at some element of a different $S_i$. 

On the other hand, by the maximality of each $S_j$ ($j=1,\cdots, N$), if $p_{i_0}\not\in S_1\cup \cdots \cup S_{N}$, then as observed in~\cite{CG85a}, there exist $p_{i_j}\in S_j$ for \emph{each} $j=1,\cdots,N$ (note that there may be more than one $p_{i_j}$ from a single $S_j$, but we just pick one of them), such that 
$$d(p_{i_0},p_{i_j})\ <\ 10\zeta \max\{l_a(p_{i_0}),l_a(p_{i_j})\}\quad \text{(compare (\ref{eqn: distance<}))}$$
implying as before, since $\zeta<\frac{1}{20}$, that
$$\max\{l_a(p_{i_0}),l_a(p_{i_j})\}\ \le\ \frac{\min\{l_a(p_{i_0}),l_a(p_{i_j})\}}{1-10\zeta}\quad\text{and}\quad d(p_{i_0},p_{i_j})\ \le\ \frac{10\zeta}{1-10\zeta}l_a(p_{i_0}).$$
 So we have the following containment relations 
$$B(p_{i_0},\zeta l_a(p_{i_0}))\ \subset\ B\left(p_{i_j},\frac{11\zeta-10\zeta^2}{1-10\zeta}l_a(p_{i_0})\right)\ \subset\ B\left(p_{i_0},\frac{21\zeta-10\zeta^2}{1-10\zeta}l_a(p_{i_0})\right),$$
and by (\ref{eqn: distance>}), the mutual disjointness of $B\left(p_{i_j},\frac{1}{2}\zeta(1-10\zeta)l_a(p_{i_0})\right)$ for $j=1,\cdots,N$. Now we fix some $\zeta\in (0,\frac{1}{40})$, and do the same rescaling as before $g\mapsto l_a(p_{i_0})^{-2}g$. The unit curvature bound on the rescaled unit ball around $p_{i_0}$, the containment relations and mutual disjointness, together with the multiplicity estimate, give a dimensional bound on $N$, as argued by volume comparison within $\tilde{B}(p_{i_0},1)$ above.
\end{proof}

The fact that the number of partitions of the covering is independent of specific manifold, together with the elliptic regularity (\ref{eqn: regularity}), ensure that the work of Cheeger-Fukaya-Gromov~\cite{CFG92} go through. Thus we have arrived at
\begin{theorem}[Cheeger-Fukaya-Gromov~\cite{CFG92}, Cheeger-Tian~\cite{CT05}]
For any $\epsilon>0$ and $r\in (0,1)$, there exists a $\delta_{CFGT}(\epsilon)>0$ and $\alpha_0,k>0$, such if $U\subset M$ is $(\delta,a)$-collapsing with locally bounded curvature, for some $\delta<\delta_{CFGT}$, then there is an approximating metric $g^{\epsilon}$ on some open subset $W$ with $U\subset W\subset B(U,\frac{a}{2})$, together with an $a$-standard $N$-structure on $W$, such that:
\begin{enumerate}
\item $g^{\epsilon}$ is $(\alpha_0\ l_a,k)$-round in the sense of (1.1.1)-(1.1.6) of~\cite{CFG92};
\item the approximation satisfies
\begin{align*}
e^{-\varepsilon }g^{\varepsilon}\ l_a^2\ \le\ &\ g \le\ e^{\varepsilon}g^{\varepsilon}\ l_a^2,\\
|\nabla^{g}-&\nabla^{g^{\epsilon}}|\ <\ \epsilon\ l_a^{-1},\\
\text{and}\quad |\nabla^k\Rm_{g^{\epsilon}}-&\nabla^k\Rm_g|\ <\ \Psi(\epsilon\ |\ k)\ l_a^{-2-k};
\end{align*}
\item $g^{\epsilon}$ is invariant under the local nilpotent actions of the $N$-structure;
\item $\forall x\in W$, its orbit, $\mathcal{N}(x)$ is compact with $\diam_{g^{\epsilon}}\mathcal{N}(x)\ \le\ \epsilon\ l_a(x)$; and
\item $W=\cup_{x\in W}\mathcal{N}(x)$, i.e. $W$ is \emph{saturated}.
\end{enumerate}
\label{thm: CFGT}
\end{theorem}
We immediately have:
\begin{corollary}[Vanishing Euler characteristics]
If $U\subset B(p_0,R)$ is $(\delta,a)$-collapsing with locally bounded curvature, then $\chi(W)=0$.
\label{cor: vanishing_euler}
\end{corollary}
\begin{proof}
By the existence of an $a$-standard $N$-structure of positive rank over $W$, we have a topological fibration $\mathbb{S}^1\hookrightarrow W\rightarrow B$ where $B$ is the collection of all orbits of the $\mathbb{S}^1$ action, induced by the action associated to the $N$-structure. Thus $\chi(W)=\chi(\mathbb{S}^1)\chi(B)=0$. 
\end{proof}
The construction of the $N$-structure and approximating metric $g^{\epsilon}$ starts on geodesic balls of scale $l_a$. Once we do the rescaling $g\mapsto l_a(p)^{-2}g$, we can carry out the constructions of Section 2 and 5 of~\cite{CFG92} to obtain local fibrations. In order to glue the local fibration and group actions, as done in Section 6 and 7 of~\cite{CFG92}, we need Lemma~\ref{lem: covering} which tells, essentially, that one can carry out the gluing procedure by adjusting within a single ball at a time. Finally, notice that once two balls intersect non-trivially, then (\ref{eqn: l_a_Harnack}) is in effect, and rescaling one ball to unit curvature bound will ensure the rescaled metric having curvature norm bounded by $2$ on the union of both balls, and Proposition A2.2 of~\cite{CFG92} works for the gluing. The same principles apply for the following equivariant good chopping only assuming locally bounded curvature:
\begin{proposition}[Good chopping for collapsing sets]
Let $(M,g,f)$ be a 4-d Ricci shrinker and fix $a\in (0,1)$. There exist constants $\delta_{GC}(R)>0$ and $C_{GC}(R)>0$ such that if an open set $U$ is $(\delta,a)$-collapsing with locally bounded curvature with $\delta<\delta_{GC}$, then there exists another open set $W$ such that
\begin{enumerate}
\item $U\subset W\subset B(U,\frac{a}{2})$;
\item $\partial W$ is smooth and $|II_{\partial W}|\ \le\ C_{GC}l_a^{-1}$;
\item $W$ is saturated by some $a$-standard $N$-structure, whose existence is guaranteed by the last theorem above.
\end{enumerate}
\label{prop: GC}
\end{proposition}

\begin{remark}
We notice the inconsistency of estimates (3.7) and (3.10) of~\cite{CT05} when collapsing happens. 
Instead, we should use relations (1.8) and (1.9) of~\cite{CGIII} in place of estimate (3.7) of~\cite{CT05}. 
\end{remark}

When collapsing happens, the basic idea is to smooth the \emph{distance to the orbits} of a given set (generated by the $N$-structure), rather than the \emph{distance to the original set}. However, due to the possible occurrence of a mixed $N$-structure, Cheeger-Gromov's equivariant good chopping theorem~\cite{CGIII} does \emph{not} apply directly (as done in~\cite{CT05}), to the smoothing of the distance function. See Appendix A for a detailed proof, where we will use Fukaya's frame bundle argument~\cite{Fukaya88}.

Combining the above propositions, Cheeger-Tian~\cite{CT05} obtain the following estimates of the boundary Gauss-Bonnet-Chern term:
\begin{proposition}
Let $(M,g,f)$ be a 4-d Ricci shrinker and fix $a\in (0,1)$. There exist positive constants $\delta_{CT}(R)\le \delta_{GC}(R)$ and $C_{CT}(R)>0$ such that for any $K\subset B(p_0,R-a)$ with $B(K,a)$ being $(\delta,a)$-collapsing with locally bounded curvature for some $\delta<\delta_{CT}(R)$, then there exists an open subset $Z$, saturated with respect to the associated $N$-structure of an approximating metric, such that 
\begin{enumerate}
\item $B(K,\frac{1}{4}a)\ \subset\ Z\ \subset\ B(K,\frac{3}{4}a),$
\item $|II_{\partial Z}|\ \le\ C_{CT}(R) (a^{-1}+r_{\Rm}^{-1}),$ \quad and
\item $\left|\int_{\partial Z} \bEuler\right|\ \le\ C_{CT}(R)\ a^{-1}\int_{A(K,\frac{1}{4}a,\frac{3}{4}a)}\left(a^{-3}+r_{\Rm}^{-3}\right)\ \dvol_g.$
\end{enumerate}
\label{prop: integral_TP_control}
\end{proposition}
The proof of this proposition only used, in addition to the previous propositions, the volume comparison, and this is available within $B(p_0,R)$ by Lemma~\ref{lem: volume_comparison}.

\section{Proof of the $\epsilon$-regularity theorem for 4-d Ricci shrinkers}
The foundation of the proof is Anderson's $\epsilon$-regularity with respect to collapsing, which basically asserts that the smallness of the renormalized energy $I_{\Rm}^f$ (see Definition~\ref{defn: renormalized_energy}) at certain scale guarantees the uniform curvature bound at half of that scale. However, the (more natural) input of our theorem is the smallness of the local energy $$E(p,r):=\int_{B(p,r)}|\Rm|^2\ \dmu_f\ <\ \epsilon,$$
which, when collapsing happens, may well be caused by the smallness of $\mu_f(B(p,r))$, and it is not obvious at all that small local energy implies the smallness of the renormalized energy. However, we will follow the strategy of Cheeger-Tian~\cite{CT05} to find that for 4-d Ricci shrinkers, the above smallness of energy indeed implies the smallness of the renormalized energy, at a much smaller, but definite scale.

\subsection{The key estimate for $4$-d Ricci shrinkers} 
Recall that the curvature can be controlled by 
$$|\Rm|^2\ \le\ 8\pi^2|\Euler| + |\mathring{\nabla}^2f|^2.$$
The main task is to obtain an average control of $|\Euler|$. This is done by an induction process, which is based on Proposition~\ref{prop: integral_TP_control} and the vanishing of the Euler characteristics on subsets that are $(\delta,a)$-collapsing with locally bounded curvature. In order to better extract information from Proposition~\ref{prop: integral_TP_control}, we start with a maximal function argument.

For each $u\in L^1(M, g,\dmu_f)$, we can define 
\begin{align*}
M^f_u(x,s):=\sup_{s'\le s} \frac{1}{\mu_f(B(x,s'))}\int_{B(x,s')}u\ \dmu_f.
\end{align*}
Recall the volume doubling property (\ref{eqn: doubling}) and applying Lemma 4.1 of~\cite{CT05}, we get
\begin{lemma}
There is a constant $C_{4.1}(R,\alpha)>0$, for each $R,\alpha>0$, such that for any $\dmu_f$-measurable subset $W\subset B(p_0,R)$,
\begin{align}
\left(\frac{1}{\omega}\int_WM^f_u(x,s)^{\alpha}\ \dmu_f\right)^{\frac{1}{\alpha}}\le \frac{C_{4.1}(R,\alpha)}{\mu_f(W)}\int_{B(W,6s)}|u|\ \dmu_f.
\end{align} 
\label{lem: maximal_function}
\end{lemma}
From Proposition~\ref{prop: integral_TP_control}, we can estimate:
\begin{lemma}
Fix $r\in (0,1)$ and $\delta<\min\{\delta_{CFGT},\delta_{GC}\}$. There exists a
$C_{4.2}(R)>0$, such that if some compact set $K\subset B(p_0,R-r)$ has its
$r$-neighborhood $B(K,r)$ being $(\delta, r)$-collapsing with locally bounded
curvature, then we have some saturated open set $Z\subset B(K,\frac{1}{2}r)$
with smooth boundary, containing $B(K,\frac{1}{4}r)$ such that
\begin{align}
\left|\int_{Z}\mathcal{P}_{\chi}\right|\le C_{4.2}(R)\mu_f(A(K;0,r)) r^{-1}\left(r^{-3}+\left(\frac{1}{\mu_f(A(K;0,r))}\int_{A(K;\frac{1}{4}r,\frac{3}{4}r)}|\Rm|^2\ \dmu_f\right)^{\frac{3}{4}}\right).
\label{eqn: P_bound}
\end{align}
\end{lemma}
\begin{proof}
By the measure equivalence (\ref{eqn: Euclidean_comparison}) and Proposition~\ref{prop: integral_TP_control}, we get
\begin{align}
\begin{split}
\left|\int_{\partial Z}\mathcal{TP}_{\chi}\right|\ &\le\ C_{CT}(R)r^{-1}\int_{A(K;\frac{1}{3}r,\frac{2}{3}r)}(s^{-3}+r_{\Rm}^{-3})\ \dvol_g\\
&\le\ C_{CT}(R)e^{(2R+\sqrt{2})^2}r^{-1} \int_{A(K;\frac{1}{3}r,\frac{2}{3}r)}(s^{-3}+r_{\Rm}^{-3})\ \dmu_f.
\end{split}
\label{eqn: TP_bound}
\end{align}
Now we notice that for $s\in (0,1]$,
$$\rho_f(p)^{-1}\le c_{4.2.2}(R)\max\left\{M^f_{|\Rm|^2}(p,s)^{\frac{1}{4}},s^{-1}\right\}.$$
This is because if $\rho_f(p)<s\le 1$, then 
$$\frac{\volfm_R(\rho_f(p))}{\mu_f(B(p,\rho_f(p)))}\int_{B(p,\rho_f(p))}|\Rm|^2\ \dmu_f=\epsilon_A(R),$$
which gives
\begin{align*}
M^f_{|\Rm|^2}(p,s)&\ge \frac{1}{\mu_f(B(p,\rho_f(p)))}\int_{B(p,\rho_f(p))}|\Rm|^2\ \dmu_f\\
&=\frac{\epsilon_A(R)}{\volfm_R(\rho_f(p))}\ge \frac{e^{-(2R+\sqrt{2})}\epsilon_A(R)}{\volfm_R(1)}\rho_f(p)^{-4},
\end{align*}
and thus 
$$\rho_f(p)^{-1}\le c_{4.2.2}(R)M^f_{|\Rm|^2}(p,s)^{\frac{1}{4}},$$
where $c_{4.2.2}(R):= (e^{-(2R+\sqrt{2})}\epsilon_A(R)\slash\volfm_f(1))^{-\frac{1}{4}}$.

Now for $s\le r\le 1$ we have
\begin{align}
r_{\Rm}(p)^{-3}\le 8\rho_f(p)^{-3}\le c_{4.2.3}(R)\left(s^{-3}+\left(M^f_{|\Rm|^2}(p,s)\right)^{\frac{3}{4}}\right),
\label{eqn: local_scale_below}
\end{align}
with $c_{4.2.3}(R):=8\max\{1,c_{4.2.2}(R)^3\}$. 

Now we can choose $s=\frac{r}{512}$ and apply Lemma~\ref{lem: maximal_function} to the function $|\Rm|^2$ with $\alpha=\frac{3}{4}$ to obtain
\begin{align}
\int_{A(K;\frac{1}{3}r,\frac{2}{3}r)}\left(M^f_{|\Rm|^2}(\cdot,s)\right)^{\frac{3}{4}}\ \dmu_f\ \le\ \mu_f(A(K;0,r))\left(\frac{C_{4.1}(R,\frac{3}{4})}{\mu_f(A(K;0,r))}\int_{A(K;\frac{1}{4}r,\frac{3}{4}r)}|\Rm|^2\ \dmu_f\right)^{\frac{3}{4}}.
\label{eqn: integral_local_scale_below}
\end{align}

Then the estimates (\ref{eqn: TP_bound}), (\ref{eqn: local_scale_below}) and (\ref{eqn: integral_local_scale_below}) together give
\begin{align}
\left|\int_{\partial Z}\mathcal{TP}_{\chi}\right|\ \le\ C_{4.2}(R)\mu_f(A(K;0,r))\left(r^{-4}+\left(\frac{r^{-\frac{4}{3}}}{\mu_f(A(K;0,r))}\int_{A(K;\frac{1}{4}r,\frac{3}{4}r)}|\Rm|^2\ \dmu_f\right)^{\frac{3}{4}}\right).
\label{eqn: PX_iteration}
\end{align}

Since there exists an $r$-standard N-structure on $Z$, $\chi(Z)=0$, and we can employ the Gauss-Bonnet-Chern formula on $Z$ to finish the proof, i.e. $\int_Z\mathcal{P}_{\chi}=-\int_{\partial Z}\mathcal{TP}_{\chi}$.
\end{proof}

Recall that our purpose is to use (\ref{eqn: P_bound}) together with the special relation (\ref{eqn: Pfaffian}) between $\mathcal{P}_{\chi}$ and $|\Rm|^2$ in dimension four to estimate $\|\Rm\|_{L^2_{loc}}$. In the Einstein case $\mathring{\Sc}c\equiv 0$ but for non-trivial 4-d Ricci shrinkers, $\mathring{\Sc}c=\mathring{\nabla}^2f$ does not vanish identically. However, we could employ the good cut-off function constructed in Lemma~\ref{lem: cutoff} to obtain a local $L^2$-control of the full Hessian of $f$ by its energy. This is the content of the following lemma:
\begin{lemma}
Given $K\subset B(p_0,R-r)$, we have estimate (\ref{eqn: Hessian_L2}) for the potential function $f$.
\label{lem: Hessian_control}
\end{lemma}
\begin{proof}
By Lemma~\ref{lem: cutoff}, we have a cut-off function $\varphi$ such that $0\le \varphi \le 1$, $\supp \varphi\subset B(K,r)$, $\varphi\equiv 1$ on $B(K,r\slash 4)$ and $r|\nabla \varphi|+r^2|\Delta^f\varphi|\le C_{2.10}(R)$, then we can use the Weitzenb\"ock formula (\ref{eqn: Bochner}) to compute
\begin{align*}
\int_{B(K,\frac{1}{2}r)}2|\nabla^2 f|^2\ \dmu_f\ &\le\ \int_{B(K,r)}2\varphi |\nabla^2 f|^2\ \dmu_f\\
&=\ \int_{B(K,r)}\varphi \left(\Delta^f|\nabla f|^2+|\nabla f|^2\right)\ \dmu_f\\
&\le\ \int_{A(K;0,r)}|\Delta^f\varphi| |\nabla f|^2\ \dmu_f+\int_{B(K,r)}|\nabla f|^2\ \dmu_f\\
&\le\ c(R)\left(2R+\sqrt{2}\right)^2r^{-2}\mu_f(A(K;0,r))+(2R+\sqrt{2})^2\mu_f(B(K,r)),
\end{align*}
and thus
\begin{align}
\begin{split}
\int_{B(K;\frac{1}{2}r)}2|\nabla^2 f|^2\ \dvol_g\ \le\  C_{4.3}(R)\left(2R+\sqrt{2}\right)^2e^{(2R+\sqrt{2})^2}\left(r^{-2}\mu_f(A(K;0,r))+\mu_f(B(K;r))\right).
\end{split}
\label{eqn: Hessian_L2}
\end{align}
\end{proof}

From now on, we fix $\delta_{KE}:=\frac{1}{2}\{\delta_{CFGT},\delta_{GC}\}$. Now we can generalize the following key estimate of~\cite{CT05} to 4-d Ricci shrinkers: 
\begin{proposition}[Key estimate]
Fix $r\in (0,1)$ and $R>0$. There exist constants $\epsilon_{KE}(R)>0$ and 
$C_{KE}(R)>0$, such that any $B(E,r)\subset B(p_0,R)$ which is $\delta$-volume collapsing for any $\delta<\delta_{KE}$ sufficiently small, and with
\begin{align}
\int_{B(E,r)}|\Rm|^2\ \dmu_f\ \le\ \epsilon_{KE}(R),
\end{align}
has the estimate
$$\int_{E}|\Rm|^2\ \dmu_f\ \le\ C_{KE}(R)\mu_f(B(E;r))\ r^{-4}.$$
\label{prop: KE}
\end{proposition}
\begin{proof}
The estimates (\ref{eqn: PX_iteration}) and (\ref{eqn: Hessian_L2}) (with (\ref{eqn: Pfaffian})) show that $\forall K\subset B(p_0,R-s)$ that is $(\delta,s)$-collapsing with locally bounded curvature (assume $s\in (0,1)$),
\begin{align}
\begin{split}
&\int_{B(K,\frac{1}{4}s)}|\Rm|^2\ \dmu_f\\
 \le\
&C_{4.2}(R)\mu_f(A(K;0,s))\left(s^{-4}+\left(\frac{s^{-\frac{4}{3}}}{\mu_f(A(K;0,s))}\int_{A(K;\frac{1}{4}s,\frac{3}{4}s)}|\Rm|^2\ \dmu_f\right)^{\frac{3}{4}}\right)\\
&+C_{4.3}(R)\mu_f(B(K,s)).
\end{split}
\label{eqn: Rm_L2}
\end{align} 
Here the point is that even in practice we have $\delta\to 0$, but the threshold, $\delta_{KE},$ for the theory developed in Section 3 to be applied to obtain (\ref{eqn: PX_iteration}), is universal.
 
Define $E_1:=B(E,r)$; for $i=2,3,4,\cdots$, set $E_i:=A(E;2^{-i}r,r-2^{-i}r)$,
\begin{align*}
D_i:=\{x\in E_i:r_{\Rm}(x)\le 2^{-(i+1)}r\},\quad \text{and}\quad F_i:=E_i\backslash D_i.
\end{align*}
Clearly $B(D_i,2^{-(i+1)}r)\subset E_{i+1}$ and in fact we have:
\begin{claim}
$B(D_i,2^{-(i+1)}r)$ is $(\delta_{KE},2^{-(i+1)}r)$-collapsing with locally
bounded curvature.
\end{claim} 
\begin{proof}[Proof of claim]
If $x\in B(D_i,2^{-(i+1)}r)$ has $r_{\Rm}(x)< 2^{-(i+1)}r$ then this follows
from Lemma~\ref{lem: CT_local_bdd_curvature}, if we assume
$\varepsilon_{KE}(R)\le \frac{\varepsilon_A(R)\delta_{KE}}{16\volfm_R(1)}$.

 Otherwise, if $x\in B(D_i,2^{-(i+1)}r)$ has $r_{\Rm}(x)\ge 2^{-(i+1)}r$, then
 since $Lip\ r_{\Rm}\le 1$ and $\sup_{B(D_i,2^{-(i+1)}r)} r_{\Rm}\le 2^{-i}r$,
 we have $\rho_f(x)\le 2^{-(i-1)}r$, and
\begin{align*}
\mu_f(B(x,2^{-(i+1)}r))\ &\le\ \mu_f(B(x,2^{-(i-1)}r))\ \le\
\frac{\mu_f(B(x,\rho_f(x)))\volfm_R(1)}{\volfm_R(\rho_f(x))2^{4(i-1)}r^{-4}}\\
&=\ \frac{\volfm_R(1)\ r^4}{\epsilon_A(R)2^{4(i-1)}}\int_{B(x,\rho_f(x))}|\Rm|^2\ \dmu_f\
 \le\ \frac{\volfm_R(1)\ r^4}{\epsilon_A(R)2^{4(i-1)}}\int_{B(x,r)}|\Rm|^2\
 \dmu_f\\
&\le\ \delta_{KE}2^{-4(i+1)} r^{4},
\end{align*}
provided $\epsilon_{KE}(R) \le
\frac{\epsilon_A(R)\ \delta_{KE}}{256\volfm_R(1)}.$
\end{proof}
\noindent Here we could clearly see how the energy threshold
$\varepsilon_{KE}(R)$ is determined by $\delta_{KE}$.

 Now we can apply (\ref{eqn: Rm_L2}) to $K=D_i$, $s=2^{-(i+1)}r$ to obtain
\begin{align}
\begin{split}
\frac{1}{\mu_f(B(E,r))}\int_{B(D_i,2^{-(i+3)}r)}|\Rm|^2\ \dmu_f\ \le\
c(R)\left(\frac{2^{4i}}{r^4}+\frac{2^i}{r}\left(\frac{1}{\mu_f(B(E,r))}\int_{E_{i+1}}|\Rm|^2\ \dmu_f\right)^{\frac{3}{4}}\right),
\end{split}
\label{eqn: Rm_L2i}
\end{align} 
where we need to notice that 
$$A(D_i;r\slash 2^{i+3},3r\slash 2^{i+3})\subset A(D_i;0,2^{-(i+1)}r)\subset B(D_i,2^{-(i+1)}r)\subset E_{i+1}.$$
On $F_i$, we have $|\Rm|\le 4^{i+1}r^{-2}$, so $\int_{F_i}|\Rm|^2\ \dmu_ f\
\le\ c(R) 2^{4(i+1)}r^{-4} \mu_f(A(E;0,r))$. Now we can estimate
\begin{align*}
\int_{E_i}|\Rm|^2\ \dmu_f\ &\le\ \int_{D_i}|\Rm|^2\ \dmu_f+\int_{F_i}|\Rm|^2\
\dmu_f\\
&\le\ \int_{B(D_i,2^{-(i+3)}r)}|\Rm|^2\
\dmu_f+c(R)2^{4(i+1)}r^{-4}\mu_f(B(E,r))\\
&\le\ c(R) \mu_f(B(E,r))\left(\frac{2^{4i}}{r^4}
+\frac{2^i}{r}\left(\frac{1}{\mu_f(B(E,r))}\int_{E_{i+1}}|\Rm|^2\
\dmu_f\right)^{\frac{3}{4}}\right).
\end{align*}
Similarly, (\ref{eqn: Rm_L2}) directly implies that
$$
\int_{E_1}|\Rm|^2\ \dmu_f\ \le\
c(R)\mu_f(B(E,r))\left(16r^{-4}
+2r^{-1}\left(\frac{1}{\mu_f(B(E,r))}\int_{E_2}|\Rm|^2\ \dmu_f\right)^{\frac{3}{4}}\right).
$$
Therefore, we could set $a_i:=c(R)r^{-4}16^i$, $b_i:=c(R)r^{-1}2^i$, and $x_i:=\frac{1}{\mu_f(B(E,r))}\int_{E_i}|\Rm|^2\ \dmu_f$ for $i=1,2,3,\cdots$, then $a_i,b_i,x_i$ satisfy the relations
$$x_i\le a_i+b_ix_{i+1}^{\frac{3}{4}},\quad\text{and}\quad \limsup_{i\rightarrow \infty} x_i^{(\frac{3}{4})^i}=1.$$
Notice that $\sum_{j=0}^{\infty}\left(\frac{3}{4}\right)^j=4$, we can apply Lemma 5.1 of~\cite{CT05} to obtain
$$\frac{1}{\mu_f(B(E,r))}\int_{E}|\Rm|^2\ \dmu_f\ =\ x_1\ \le\ C_{KE}(R)r^{-4}.$$
\end{proof}

As mentioned in the Introduction, (\ref{eqn: Hessian_L2}) gives a bound that blows up in the induction process. However, the blow up rate is of second order in the inductive scale, which is absorbed by the controlling terms, i.e. the right-hand side of (\ref{eqn: Rm_L2}), blowing up of fourth order in the same scale. This observation will also be crucial for our arguments in the next sub-section.


\subsection{The fast decay proposition.}
As the key estimate tells, as long as the energy is sufficiently small at a given scale, the renormalized energy at that scale is bounded.
In order to find a \emph{uniform} scale, reducing to which the renormalized energy is small enough to apply Anderson's $\epsilon$-regularity theorem, we need the following proposition:

\begin{proposition}
Let $(M,g,f)$ be a 4-d Ricci shrinker and fix $R>2\sqrt{2}$. There exists some $r_{FD}(R)>0$, $\varepsilon_{FD}(R)>0$, $\delta_{FD}(R)>0$ and $\eta_R>0$, such that for $B(p,2r)\subset B(p_0,R)$ with $r<r_{FD}(R)$, if
\begin{align}
\frac{\mu_f(B(p,r))}{\volfm_R(r)}\ <\ \delta_{FD}(R),
\label{eqn: vol_collapsing_decay}
\end{align}
 and
\begin{align}
\int_{B(p,2r)}|\Rm|^2\ \dmu_f\ \le \varepsilon_{FD}(R),
\label{eqn: small_energy}
\end{align} 
then 
\begin{align}
\frac{\volfm_R(r)}{\mu_f(B(p,r))}\int_{B(p,r)}|\Rm|^2\ \dmu_f\le (1-\eta_R)\frac{\volfm_R(2r)}{\mu_f(B(p,2r))}\int_{B(p,2r)}|\Rm|^2\ \dmu_f.
\label{eqn: decay}
\end{align}
\label{prop: decay}
\end{proposition}
\begin{remark}
Abusing notations, we will always denote a possible subsequence by the original one.
\end{remark}

In this subsection we will take several steps to prove this proposition. Essentially, the proof reduces the problem, by blowing up the radius $r$, to a situation similar to the Einstein case. But this principle works on two levels: on the level of $|\nabla f|$, its smallness after rescaling will directly give a comparison geometry picture similar to the Einstein case; however, on the level of $|\nabla^2 f|$, we notice that $\int_{B(p,r)}|\mathring{\nabla}^2f|^2\ \dmu_f$ is scaling invariant, and we need to use the Weitzenb\"ock formula (\ref{eqn: Bochner}) to give it a local $L^2$-control of order lower than that of $\int_{B(p,r)}|\Rm|^2\ \dmu_f$. This is in the same spirit as Lemma~\ref{lem: Hessian_control}.

Moreover, our argument avoids appealing to the theory of Cheeger-Colding-Tian, 
see Theorem 3.7 of~\cite{CCT02}. 
This is unavailable in the context of shrinking Ricci solitons since the Ricci curvature lower bound is not satisfied. However, we expect there to be a version of Cheeger-Colding-Tian's theory for manifolds with Bakry-\'Emery Ricci curvature bounded below.

We wish to point out that our argument is under the framework of Cheeger-Tian's in~\cite{CT05}, whose key observation is that the estimates (\ref{eqn: C_0})\ --\ (\ref{eqn: H_2}) of the approximating functions are in the average sense. Our new input is the elliptic regularity (\ref{eqn: u_i_regularity}) of the approximating functions that produces smooth annuli where we have \emph{global point-wise} derivative control, see Sub-sub-section (4.3.10).
We would also like to thank Jeff Cheeger for pointing out the paper~\cite{LiYe10} for an alternative treatment in a different context.



\subsubsection{\textbf{Control of Pfaffian form.}} In fact, we can assume
\begin{align}
\int_{B(p,r)}|\Rm|^2\ \dvol_{g}\ > \epsilon_A(R)\frac{\mu_{f}(B(p,r))}{e^{R^2}\omega_4\ r^4},
\end{align}
because otherwise we could have directly applied Anderson's $\epsilon$-regularity theorem to obtain the desired curvature bound, and there is no need to prove this proposition. Now we use Lemma~\ref{lem: cutoff} to obtain a cut-off function $\varphi$ supported on $B(p,2r)$, constantly equal to $1$ on $B(p,1.6r)$, and having uniform control $r|\nabla \varphi|+r^2|\Delta^f\varphi|\le C_{2.12}(R)$. Then we could estimate as in Lemma~\ref{lem: Hessian_control}:
\begin{align*}
\int_{B(p,1.6r)}|\nabla^2f|^2\ \dvol_g\ &\le\ \frac{e^{R^2}}{2}\int_{B(p,2r)}(|\Delta^f\varphi|+1)|\nabla f|^2\ \dmu_f\\
&\le\ C(R)\mu_f(B(p,r))\ r^{-2}.
\end{align*}
As long as $r<\sqrt{\frac{\epsilon_A(R)e^{-R^2}}{2C(R)\omega_4}}$, for any open set $B(p,r)\subset U\subset B(p,1.6r)$ with smooth boundary, the expression of Pfaffian (\ref{eqn: Pfaffian}) gives 
\begin{align*}
8\pi^2\int_{U}\mathcal{P}_{\chi}\ \ge\ \int_{B(p,r)}|\Rm|^2\ \dvol_{g}-\int_{B(p,1.6r)}|\mathring{\nabla}^2f|^2\ \dvol_g\ >\ 0.
\end{align*}
Let $\varepsilon_{FD}(R)\le \pi^2e^{-R^2}$, then the above inequality, together with (\ref{eqn: small_energy}) , gives 
\begin{align}
0\ <\ \int_U\mathcal{P}_{\chi}\ \le \frac{3e^{R^2}\ \varepsilon_{FD}(R)}{8\pi^2}\ <\ \frac{1}{2}
\label{eqn: bad_Euler}
\end{align}
for any open subset $U$ with smooth boundary such that $B(p,r)\subset U\subset B(p,1.6r)$.
\subsubsection{\textbf{Setting up a contradiction argument.}} We prove the proposition by a contradiction argument. Were the proposition false, then there exist $4$-d Ricci shrinkers $(M_i,g_i,f_i)$, sequences $r_i\rightarrow 0$, $\delta_i\rightarrow 0$ and $\eta_i\rightarrow 0$ as $i\rightarrow \infty$, such that for some $B(p_i,4r_i)\subset B(p^0_i,R)$ ($p^0_i$ denoting the base point of $M_i$), 
\begin{align}
\int_{B(p_i,2r_i)}|\Rm_{g_i}|^2\ \dmu_{f_i}\ &\le\ \varepsilon_{FD}(R),
\label{eqn: contradiction}\\
\text{and}\quad \frac{\mu_{f_i}(B(p_i,2r_i))}{\volfm_R(2r_i)}\ &<\ \delta_i
\label{eqn: volume_collapsing_assumption}
\end{align}
 but (\ref{eqn: decay}) is violated for each $i$. 

We will find, for each $i$ large enough, some open subset $U_i$ with smooth boundary such that $B(p_i,r_i)\subset U_i\subset B(p_i,2r_i)$ and that 
\begin{align}
0<\int_{\partial U_i} \bEuler\ <\ \frac{1}{2}.
\label{eqn: bad_bEuler}
\end{align}
Since $r_i\rightarrow 0$, (\ref{eqn: bad_Euler}) holds for all $i$ sufficiently large, so adding (\ref{eqn: bad_Euler}) and (\ref{eqn: bad_bEuler}) gives $$0<\chi(U_i)<1,$$ contradicting the integrality of $\chi(U_i)$.
\subsubsection{\textbf{Rescaling.}} Consider the rescaled sequence $(M_i,r_i^{-2}g_i,f_i)$. Denote $\tilde{g}_i:=r^{-2}_ig_i$, then the scaling invariance of the energy and (\ref{eqn: small_energy}) implies that for each $i$,
\begin{align}
\int_{\tilde{B}(p_i,2)}|\Rm_{\tilde{g}_i}|^2\ \tdmu_{f_i}\ \le \varepsilon_{FD}(R),
\label{eqn: small_energy_i}
\end{align}
where we add a tilde to an object to denote its rescaled correspondence. Moreover, the scaling invariance of volume ratio and the converse of (\ref{eqn: decay}) implies that 
\begin{align}
\frac{\volfm_{r_iR}(1)}{\tilde{\mu}_{f_i}(\tilde{B}(p_i,1))}\int_{\tilde{B}(p_i,1)}|\Rm_{\tilde{g}_i}|^2\ \tdmu_{f_i}\ >\ (1-\eta_i)\frac{\volfm_{r_iR}(2)}{\tilde{\mu}_{f_i}(\tilde{B}(p_i,2))}\int_{\tilde{B}(p_i,2)}|\Rm_{\tilde{g}_i}|^2\ \tdmu_{f_i}.
\label{eqn: no_decay_i}
\end{align}
These two inequalities will be the starting point of our future arguments. Moreover, the rescaled metrics and potential functions satisfy 
\begin{align}
\Rc_{\tilde{g}_i}+\tilde{\nabla}^2 f_i\ =\ \frac{r_i^2}{2}\tilde{g}_{i},
\end{align}
which implies the non-negativity of the rescaled Bakry-\'Emery-Ricci curvature
\begin{align}
\Rc^{f_i}_{\tilde{g}_i}\ =\ \frac{r_i^2}{2}\tilde{g}_i\ \ge\ 0,
\label{eqn: tBER}
\end{align}
and the potential function has the gradient estimates
\begin{align}
|\tilde{\nabla} f_i|_{\tilde{g}_i}\ \le\ r_iR.
\label{eqn: tgradient_estimate}
\end{align}
Finally, we denote the distance to the given point $p_i$ by $d_{p_i}(x):=d(p_i,x)$, then its rescaled version is denoted by $d_i:=r_i^{-1}d_{p_i}$.
\subsubsection{\textbf{Regularity on annuli.}} On the one hand, since 
$$\frac{\volfm_{r_iR}(1)\tilde{\mu}_{f_i}(\tilde{B}(p_i,2))}{\volfm_{r_iR}(2)\tilde{\mu}_{f_i}(\tilde{B}(p_i,1))}\ \le\ 1$$
by (\ref{eqn: volume_monotone}), we have
\begin{align*}
\int_{\tilde{A}(p_i;1,2)}|\Rm_{\tilde{g}_i}|^2\ \tdmu_{f_i}\ \le\ \frac{\eta_i}{1-\eta_i}\int_{\tilde{B}(p_i,1)}|\Rm_{\tilde{g}_i}|^2\ \tdmu_{f_i}.
\end{align*}
Let $\varepsilon_{FD}(R)>0$ be sufficiently small (and fixed from now on), so
that we can apply the key estimate (notice the correct order of the scaling
there) to obtain $$\int_{\tilde{B}(p_i,1)}|\Rm_{\tilde{g}_i}|^2\ \tdmu_{f_i}\
\le\ c(R)\tilde{\mu}_{f_i}(\tilde{B}(p_i,2)), $$ and it follows that
\begin{align*}
\int_{\tilde{A}(p_i;1,2)}|\Rm_{\tilde{g}_i}|^2\ \tdmu_{f_i}\ \le\
\frac{c(R)\eta_i}{1-\eta_i}\tilde{\mu}_{f_i}(\tilde{B}(p_i,2)).
\end{align*}
Now for any $x\in \tilde{A}(p_i;1.1,1.9)$, $\tilde{B}(x,0.1)\subset
\tilde{A}(p_i;1,2)\subset \tilde{B}(x,4)$ and
\begin{align*}
\int_{\tilde{B}(x,0.1)}|\Rm_{\tilde{g}_i}|^2\ \tdmu_{f_i}\ \le\
\frac{c(R)\eta_i}{1-\eta_i}\tilde{\mu}_{f_i}(\tilde{B}(x,4))\le
\frac{c(R)\eta_i}{1-\eta_i}\frac{\tilde{\mu}_{f_i}(\tilde{B}(x,0.1))}{\volfm_{r_iR}(0.1)}\volfm_{r_iR}(4),
\end{align*}
so by the scaling invariance of the renormalized energy, we have 
\begin{align*}
I_{\Rm_{\tilde{g}_i}}^{f_i}(x,0.1)\ \le\ \frac{c(R)\eta_i}{1-\eta_i}.
\end{align*}
 For all $i$ sufficiently large, Anderson's $\epsilon$-regularity theorem gives $|\Rm_{\tilde{g}_i}|_{\tilde{g}_i}^2(x)\le c(R)\eta_i$, thus 
\begin{align}
\sup_{\tilde{A}(p_i;1.1,1.9)}|\Rm_{\tilde{g}_i}|_{\tilde{g}_i}^2\le c(R)\eta_i\rightarrow 0\quad \text{as}\ i\rightarrow \infty.
\label{eqn: annulus_curvature_vanishing}
\end{align}
Notice that the above curvature estimate enables us to apply Lemma~\ref{lem: regularity_rescaling} and obtain uniform bounds for each $k\ge 0$:
\begin{align}
\sup_{\tilde{A}(p_i;1.2,1.8)}|\tilde{\nabla}^k\Rm_{\tilde{g}_i}|_{\tilde{g}_i}\le c(k,R),\quad \text{and}\quad \sup_{\tilde{A}(p_i;1.2,1.8)}|\tilde{\nabla}^kf_i|_{\tilde{g}_i}\le c'(k,R).
\label{eqn: annulus_regularity}
\end{align}
\subsubsection{\textbf{Almost volume annulus and smoothing distance function.}} On the other hand, since 
$$\int_{\tilde{B}(p_i,1)}|\Rm_{\tilde{g}_i}|^2\ \tdmu_{f_i}\ \le\ \int_{\tilde{B}(p_i,2)}|\Rm_{\tilde{g}_i}|^2\ \tdmu_{f_i},$$
then (\ref{eqn: no_decay_i}) implies that 
\begin{align*}
\frac{\volfm_{r_iR}(1)}{\tilde{\mu}_{f_i}(\tilde{B}(p_i,1))}\ \ge\ (1-\eta_i)\frac{\volfm_{r_iR}(2)}{\tilde{\mu}_{f_i}(\tilde{B}(p_i,2))},
\end{align*}
i.e. $\tilde{A}(p_i;1,2)$ is an annulus in an almost $f_i$-weighted volume cone for $i$ sufficiently large. By weighted volume comparison (\ref{eqn: volume_monotone}), for any $r\in (1.05,1,95)$,
\begin{align}
\frac{\tilde{\mu}_{f_i}(\partial \tilde{B}(p_i,r))}{\volfm'_{r_iR}(r)}\ \ge\ (1-\Psi(\eta_i |\ r))\frac{\tilde{\mu}_{f_i}(\tilde{B}(p_i,r))}{\volfm_{r_iR}(r)},
\label{eqn: volume_cone}
\end{align}
where $\Psi(\eta_i |\ r)$ denotes some positive function that approaches $0$ as $\eta_i\rightarrow 0$.

Now we smooth the square of the distance function $\frac{d_i^2}{2}$. For each $i$, we will solve the Dirichlet problem
$$\Delta^{f_i}_{\tilde{g}_i} u_i=4\quad \text{and}\quad u_i|_{\partial \tilde{A}(p_i;1,2)}=\frac{d^2_{i}}{2}.$$

In view of (\ref{eqn: tBER}), (\ref{eqn: tgradient_estimate}) and (\ref{eqn: volume_cone}), we can estimate $u_i$ and $\tilde{u}_i:=\sqrt{2u_i}$ by applying Lemma~\ref{lem: Cheeger-Colding}:
\begin{align}
\sup_{\tilde{A}(p_i;1.2,1.8)}\left|\tilde{u}_i-d_i\right|\ &\le\ \Psi(\eta_i,r_i\ |\ R); 
\label{eqn: C_0}\\
\fint_{\tilde{A}(p_i;1.1,1.9)}\left|\nabla \tilde{u}_i-\nabla d_i\right|^2\ \tdmu_{f_i}\ &\le\ \Psi(\eta_i,r_i\ |\ R);
\label{eqn: H_1}\\
\fint_{\tilde{A}(p_i;1.3,1.7)}|\nabla^2u_i-\tilde{g}_i|^2\ \tdmu_{f_i}\ &\le\ \Psi(\eta_i,r_i\ |\ R). 
\label{eqn: H_2}
\end{align}

Moreover, the elliptic regularity Lemma~\ref{lem: interior_regularity}, estimates (\ref{eqn: annulus_regularity}) and the $C^0$ bound (\ref{eqn: C_0}) ensures that each $\tilde{u}_i$ and $u_i$ are regular:
\begin{align}
\sup_{\tilde{A}(p_i;1.3,1.7)}|\nabla^k \tilde{u}_i|+|\nabla^k u_i|\ \le\ c''(k; R).
\label{eqn: u_i_regularity}
\end{align}

\subsubsection{\textbf{The collapsing limit}} According to Proposition~\ref{prop: Fukaya},
$\tilde{A}(p_i;1.2,1.8)\rightarrow_{GH}(X,d_{\infty})$ (after passing to a subsequence), with $X=\mathcal{R}(X)\cup \mathcal{S}(X)$. Here $\mathcal{R}(X)$ is a lower dimensional Riemannian manifold equipped with a smooth Riemannian metric $g_{\infty}$ with bounded curvature  (invoking (\ref{eqn: annulus_regularity})), such that $d_{\infty}|_{\mathcal{R}(X)}$ is induced by $g_{\infty}$. $\mathcal{S}(X)$ is a stratified collection of subsets of $X$, each strata of $\mathcal{S}(X)$ by itself being a Riemannian manifold of dimension even lower than that of $\mathcal{R}(X)$. There is a constant $\iota_X>0$ such that 
$$\forall x_{\infty}\in \mathcal{R}(X),\quad \text{inj}\ x_{\infty}\ge \min\left\{d_{\infty}(x_{\infty},\mathcal{S}(X)),d_{\infty}(x_{\infty},\partial X),\iota_X\right\}.$$



\subsubsection{\textbf{Local average control of $u_i$}} We will study the behavior of $u_i$ at each point of $\tilde{A}(p_i;1.3,1.7)$ by taking limit. Fix $x_{\infty}\in\mathcal{R}(X)$ such that $x_i\rightarrow_{GH} x_{\infty}$ for some sequence $x_i\in \tilde{A}(p_i;1.3,1.7)$. Fix a scale $\alpha=\alpha(x_{\infty})<\min\{0.001,\frac{1}{2}d_{\infty}(x_{\infty},\mathcal{S}(X)), \iota_X\}$, we have, by volume comparison,
\begin{align}
\frac{\tilde{\mu}_{f_i}(\tilde{B}(x_i,\alpha))}{\tilde{\mu}_{f_i}(\tilde{A}(p_i;1.3,1.7))}\ \ge\ \frac{\volfm_{r_i R}(\alpha)}{\volfm_{r_iR}(4)}.
\label{eqn: volume_lower_bound}
\end{align} 

Now we can localize the estimates (\ref{eqn: H_1}) and (\ref{eqn: H_2}): 
\begin{align}
\fint_{\tilde{B}(x_i,\alpha)} |\nabla \tilde{u}_i-\nabla d_i|^2\ \tdmu_{f_i}\ &\le\ \Psi(\eta_i,r_i\ |\ R,\alpha);
\label{eqn: local_H_1}\\
\fint_{\tilde{B}(x_i,\alpha)} |\nabla^2 u_i-g_i|^2\ \tdmu_{f_i}\ &\le\ \Psi(\eta_i,r_i\ |\ R,\alpha).
\label{eqn: local_H_2}
\end{align}

\subsubsection{\textbf{Limit local covering geometry}} Let $\pi_i:\tilde{B}_i \rightarrow \tilde{B}(x_i,\alpha)$ be the universal covering of $\tilde{B}(x_i,\alpha)$, with lifted base point $\tilde{x}_i$ and deck transformation group $\Gamma_i$.  Recall that the scale $\alpha=\alpha(x_{\infty})$ is chosen so that $B(x_{\infty},\alpha)$ is away from $\mathcal{S}(X)$ and simply connected. This means, by Fukaya's fibration theorem, that for all $i$ sufficiently large, $\tilde{B}(x_i,\alpha)$ is topologically a torus bundle over $B(x_{\infty},\alpha)$, whence topologically
\begin{align}
\tilde{B}_i\ \approx\ \mathbb{R}^{4-\dim_HX}\times B(x_{\infty},\alpha).
\label{eqn: covering_geometry}
\end{align} 
We equip $\tilde{B}_i$ with the pull-back metric $h_i:=\pi_i^{\ast}\tilde{g_i}$ and potential function $\tilde{f}_i:=\pi_i^{\ast}f_i$. Clearly $(B(\tilde{x}_i,\alpha),h_i)$ is non-collapsing, and on $B(\tilde{x}_i,\alpha)$, estimates (\ref{eqn: annulus_curvature_vanishing}) and (\ref{eqn: annulus_regularity}) hold for $\Rm_{h_i}$ and $\tilde{f}_i$. This ensures that $\{B(\tilde{x}_i,\alpha)\}$ converges, after passing to a subsequence, to $B(\tilde{x}_{\infty},\alpha)$, a 4-dimensional Riemannian manifold with limiting Riemannian metric $h_{\infty}$. Moreover, by (\ref{eqn: annulus_curvature_vanishing}), possibly passing to a subsequence, $h_i$ smoothly converges to the flat metric $h_{\infty}=g_{Euc}$ on $B(\tilde{x},\alpha)$. We will denote the pull-back measure by $\nu_i:=\pi_i^{\ast}(\tdmu_{f_i})$, and by $\tilde{d}_i:=\pi_i^{\ast}d_i$.

Recall that by (\ref{eqn: tgradient_estimate}), $|\nabla \tilde{f}_i|_{h_i}=|\nabla f_i|_{\tilde{g}_i}\le r_iR\rightarrow 0$ as $i\rightarrow \infty$,  and that $\{\tilde{f}_i\}$ has uniform derivative control (\ref{eqn: annulus_regularity}), the drifted Laplace operators $\Delta_{h_i}^{\tilde{f}_i}$ converge smoothly to $\Delta=\sum_{j=1}^4\partial_j\partial_j$, the standard Laplace operator for $(\mathbb{R}^4,g_{Euc})$.

Moreover, each pull-back smooth function $v_i:=\pi_i^{\ast}\tilde{u}_i$ satisfies the elliptic equation 
\begin{align*}
\Delta^{\tilde{f}_i}_{h_i}v_i^2\ =\ 8.
\end{align*}
The smooth convergence of the drifted Laplace operators $\Delta_{h_i}^{\tilde{f}_i}$ further gives, for $i$ large enough, uniform elliptic estimates
\begin{align}
\sup_{B(\tilde{x}_{\infty},0.09)}|\nabla^k v_i^2|_{h_i}\le c''(k,R).
\label{eqn: Poisson_regularity}
\end{align}
The uniform boundedness (\ref{eqn: C_0}) and the regularity estimates (\ref{eqn: Poisson_regularity}) ensure that $v_i\rightarrow v_{\infty}$ in $C^{\infty}(B(\tilde{x}_{\infty},0.9\alpha)$ (after possibly passing to a further subsequence), the limiting equation being 
\begin{align}
\Delta v_{\infty}^2\ =\ 8\quad \text{on}\quad B(\tilde{x}_{\infty},0.9\alpha).
\label{eqn: limit_elliptic}
\end{align} 
To summarize, when $i\rightarrow \infty$ and after passing to subsequences, we have smooth convergence on $B(\tilde{x}_{\infty},0.9\alpha)$, of the sequence of metrics $h_i\rightarrow g_{Euc}$, of the sequence of potential functions $\tilde{f}_i\rightarrow c(R)$ (whence the smooth convergence of the elliptic operators $L_i\rightarrow \Delta$) and of the sequence of Poisson solutions $v_i\rightarrow v_{\infty}$.
\subsubsection{\textbf{Local point-wise control of $u_i$}} Now we will discuss the effect of the estimates (\ref{eqn: local_H_1}) and (\ref{eqn: local_H_2}) on the local coverings. Let $B_i\ni \tilde{x}_i$ be a fundamental domain of $\tilde{B}_i$, then for each sufficiently large $i$, in view of (\ref{eqn: covering_geometry}), we have $B(\tilde{x}_i,\alpha)\subset \tilde{U}_i\subset B(\tilde{x}_i,2\alpha)$ where
\begin{align}
\tilde{U}_i:=\cup\{\gamma B_i:\gamma\in \Gamma_i, \gamma B_i\cap B(\tilde{x},\alpha)\not=\emptyset\}.
\label{eqn: cover}
\end{align}
Notice that estimates (\ref{eqn: local_H_1}) and (\ref{eqn: local_H_2}) on the local covering, for each $\gamma\in \Gamma_i$, read
\begin{align*}
\int_{\gamma B_i} |\nabla v_i-\nabla \tilde{d}_i|^2\ \text{d}\nu_i\ \le\ \Psi(\eta_i,r_i\ |\ R,\alpha) \nu_i(\gamma B_i);\\
\int_{\gamma B_i} |\nabla v_i-h_i|^2\ \text{d}\nu_i\ \le\ \Psi(\eta_i,r_i\ |\ R,\alpha) \nu_i(\gamma B_i).
\end{align*}
Then by Bishop-Gromov volume comparison on $\tilde{B}_i$, we have
\begin{align*}
\int_{B(\tilde{x}_i,\alpha)} |\nabla v_i-\nabla \tilde{d}_i|^2\ \text{d}\nu_i\ &\le\ \int_{\tilde{U}_i} |\nabla v_i-\nabla \tilde{d}_i|^2\ \text{d}\nu_i\\
&\le\ \sum_{\gamma B_i\cap B(\tilde{x}_i,\alpha)\not=\emptyset}\int_{\gamma B_i} |\nabla v_i-\nabla \tilde{d}_i|^2\ \text{d}\nu_i\\
&\le\ \Psi(\eta_i,r_i\ |\ R,\alpha) \sum_{\gamma B_i\cap B(\tilde{x}_i,\alpha)\not=\emptyset} \nu_i(\gamma B_i)\\
&\le\ \Psi(\eta_i,r_i\ |\ R,\alpha)\nu_i(B(\tilde{x}_i,2\alpha));
\end{align*}
whence
\begin{align}
\fint_{B(\tilde{x}_i,\alpha)} |\nabla v_i-\nabla \tilde{d}_i|^2\ \text{d}\nu_i\ \le\ \Psi(\eta_i,r_i\ |\ R,\alpha),
\label{eqn: covering_H_1}
\end{align}
and similarly,
\begin{align}
\fint_{B(\tilde{x}_i,\alpha)} |\nabla^2 v_i^2 - h_i|^2\ \text{d}\nu_i\ &\le\ \Psi(\eta_i,r_i\ |\ R,\alpha).
\label{eqn: covering_H_2}
\end{align}

When passing to the limit, these estimates, together with the regularity (\ref{eqn: Poisson_regularity}) give 
\begin{align}
|\nabla v_{\infty}|\ \equiv\ 1\quad\text{and}\quad
\nabla^2v_{\infty}^2\ \equiv\ 2\ g_{Euc}\quad \text{in}\ B(\tilde{x}_{\infty},0.7\alpha).
\label{eqn: Hessian_rigidity}
\end{align}
Thinking of $B(\tilde{x}_{\infty},0.7\alpha)$ as a region in $\mathbb{R}^4$ with $\tilde{x}_{\infty}=\mathbf{0}$, we see that $v_{\infty}^2(\mathbf{x})=|\mathbf{x}-\mathbf{x}_0|^2$ for $\mathbf{x}\in B(\mathbf{0},0.7\alpha)$ and some $\mathbf{x}_0\in \mathbb{R}^4$. Moreover, $v_{\infty}(\mathbf{0})=\lim_{i\rightarrow \infty}\tilde{u}_i(x_i)$.

For any $i>i_{x_{\infty}}$, since the local covering is equipped with the pull-back metric, the smoothness of the convergence (\ref{eqn: Poisson_regularity}) then gives 
\begin{align}
|\nabla \tilde{u}_i|\ \ge\ 1-10^{-10}\quad\text{and}\quad |\nabla^2 u_i-\tilde{g}_i|\ \le\ 10^{-10}\quad\text{in}\quad B(x_i,0.6\alpha).
\label{eqn: pointwise_u_i}
\end{align}

Now we consider the second fundamental form of $\tilde{u}_i^{-1}(\tilde{u}_i(x_i))$: since at $x_i$,
\begin{align}
\lim_{i\rightarrow \infty} \frac{\nabla^2 v_i}{|\nabla v_i|}(\tilde{x}_i)\ = \frac{1}{v_{\infty}(\tilde{x}_{\infty})}(g_{Euc}-\nabla r\otimes \nabla r),
\end{align}  
where $r$ is the Euclidean distance function to the origin, we have, especially, the principal curvatures of $\tilde{u}_i^{-1}(\tilde{u}_i(x_i))$ at $x_i$, satisfies 
\begin{align}
\left|\kappa_i^k(x_i)-\frac{1}{\tilde{u}_i(x_i)}\right|\ <\ 10^{-10}\quad\text{for all}\ i>i_{x_{\infty}}.
\label{eqn: good_principal_curvature_x_i}
\end{align}
This further implies a control of the boundary Gauss-Bonnet-Chern term for $\tilde{u}^{-1}_i(\tilde{u}_i(x_i))$ at $x_i$: since
\begin{align*}
\bEuler(x_i)\ =\ \frac{1}{4\pi^2}\left(2\prod_{k=1,2,3}\kappa_i^k(x_i)-\sum_{k=1,2,3}\kappa_i^k(x_i)\mathcal{K}_{\tilde{g}_i}^{\hat{k}}(x_i)\right)\ \darea_{\tilde{u}_i^{-1}(\tilde{u}_i(x_i))},
\end{align*}
where $\hat{k}$ is a pair of numbers in $\{1,2,3\}$ not containing $k$, we have, by (\ref{eqn: annulus_curvature_vanishing}) and (\ref{eqn: good_principal_curvature_x_i}), for all $i>i_{x_{\infty}}$,
\begin{align}
\left|\bEuler(x_i)-\frac{1}{2\pi^2\tilde{u}_i(x_i)^3}\ \darea_{\tilde{u}_i^{-1}(\tilde{u}_i(x_i))}\right|\ <\ 10^{-10}.
\label{eqn: good_TPX_x_i}
\end{align}
\subsubsection{\textbf{Global point-wise control of $u_i$}} Notice that (\ref{eqn: pointwise_u_i}) and (\ref{eqn: good_principal_curvature_x_i}) are actually point-wise controls, since the scale $\alpha$ depends on specific $x_{\infty}=\lim_{GH}x_i \in \mathcal{R}(X)$; especially, from the argument above we could not obtain any control as we approach $\mathcal{S}(X)$. Luckily, $u_i$ has very nice regularity (\ref{eqn: u_i_regularity}), so that we can choose a \emph{uniform} scale $\alpha_0>0$ sufficiently small such that for any $x',x''\in \tilde{A}(p_i;1.3,1.7)$, and $\kappa_i^k(x)$ being the $k$-th principal vector of $\tilde{u}_i^{-1}(\tilde{u}_i(x))$,
\begin{align}
d(x',x'')\ <\ 3\alpha_0\ \Rightarrow\ \left||\nabla \tilde{u}_i|(x')-|\nabla \tilde{u}_i|(x'')\right|+\sum_{k=1,2,3}\left|\kappa_i^k(x')-\kappa_i^k(x'')\right|\ <\ 10^{-10}.
\label{eqn: small_change}
\end{align}

Now let $\{x_{\infty}^j\}\subset \mathcal{R}(X)$ be a minimal $\alpha_0$-net of $\mathcal{R}(X)$, and $\{x_i^j\}\subset \tilde{A}(p_i;1.3,1.7)$ such that $x_i^j\rightarrow_{GH}x_{\infty}^j$. Obviously $j<J$ for some absolute constant $J$. For large enough $i$, $\{B(x_i^j,2\alpha_0)\}$ covers $\tilde{A}(p_i;1.3,1.7)$. Then (\ref{eqn: pointwise_u_i}), (\ref{eqn: good_principal_curvature_x_i}) and (\ref{eqn: good_TPX_x_i}) work for each $x_i^j$, when $i>i_0:=\max\{i_{x^j_{\infty}}:\ j=1,\cdots, J\}$ is large enough. We could further estimate, by (\ref{eqn: annulus_curvature_vanishing}) and (\ref{eqn: small_change}), that
\begin{align}
\inf_{\tilde{B}(x_i^j,2\alpha_0)}|\nabla \tilde{u}_i|\ >\ 1-10^{-5}\quad\text{and}\quad \sup_{\tilde{B}(x_i^j,2\alpha_0)}\sum_{k=1,2,3}\left|\kappa_i^k-\frac{1}{\tilde{u}_i}\right|\ <\ 10^{-5},
\label{eqn: u_i_good}
\end{align}
whence the same estimate globally on $\tilde{A}(p_i;1.3,1.7)$, for all $i>J$ sufficiently large.

Especially, since $(1.4,1.6)\subset Image(\tilde{u}_i)$ by (\ref{eqn: C_0}), this implies that $\tilde{u}_i^{-1}(a)$ is a smooth hyper-surface in $\tilde{A}(p_i;1.3,1.7)$, for all $a\in (1.4,1.6)$ and large enough $i$. Furthermore, we can control the boundary Gauss-Bonnet-Chern form of $\tilde{u}^{-1}_i$ by (\ref{eqn: annulus_curvature_vanishing}), (\ref{eqn: annulus_regularity}), (\ref{eqn: u_i_regularity}) and (\ref{eqn: u_i_good}): for all $i>i_0$ large enough and $a\in (1.4,1.6)$,
\begin{align}
\left|\bEuler-\frac{1}{2\pi^2a^3}\ \darea_{\tilde{u}_i^{-1}(a)}\right|\ <\ 10^{-4},
\label{eqn: TPX_good}
\end{align}
 since $|\nabla \bEuler|\ \le\ C(R)$.
\subsubsection{\textbf{Level sets of $u_i$}} From the co-area formula,  (\ref{eqn: H_1}) and the scaling invariance of (\ref{eqn: contradiction}), we can estimate
\begin{align*}
\int_{1.4}^{1.6}\frac{\tilde{\mu}_{f_i}(\tilde{u}_i^{-1}(s))}{2\pi^2s^3}\ \text{d} s\ &\le\ \frac{\int_{1.4}^{1.6}\tilde{\mu}_{f_i}(\tilde{u}_i^{-1}(s))\ \text{d}s}{\int_{1.4}^{1.6}2\pi^2 s^3\ \text{d}s}\\
&=\ C(R)\frac{\int_{\tilde{u}^{-1}([1.4,1.6])}|\nabla \tilde{u}_i|\ \tdmu_{f_i}}{\volfm_{r_iR}(1.6)-\volfm_{r_iR}(1.4)}\\
&\le\ C(R)\frac{\tilde{\mu}_{f_i}(\tilde{A}(p_i;1.1,1.9))(1+\Psi(\eta_i,r_i\ |\ R))}{\volfm_{r_iR}(1.6)-\volfm_{r_iR}(1.4)}\\
&\le\ C(R)\frac{\tilde{\mu}(\tilde{B}(p_i,2))}{\volfm_{r_iR}(2)}\ <\ C(R)\delta_i.
\end{align*}
Thus for all $i>i_0$ sufficiently large, by (\ref{eqn: u_i_regularity}) and (\ref{eqn: u_i_good}), there is some $a_i\in (1.4,1.6)$ such that  
\begin{align}
\frac{\tilde{\mu}_{f_i}(\tilde{u}_i^{-1}(a_i))}{2\pi^2a_i^3}\ <\ \frac{1}{6}e^{-R^2-r_iR},
\label{eqn: small_boundary}
\end{align}
whenever $i$ is large enough (so that $\delta_i<\frac{1}{6}e^{-2R^2}C(R)^{-1}$). 
\subsubsection{\textbf{The contradiction}} We fix any $i>i_0$ sufficiently large, and set $U_i:=\tilde{B}(p_i,1.4)\cup \tilde{u}_i^{-1}(1.3,a_i)$, and notice the smoothness of $\partial U_i =\tilde{u}_i^{-1}(a_i)$ by (\ref{eqn: u_i_good}). Moreover, we have
\begin{align}
0\ <\ \int_{\partial U_i}\bEuler\ <\ \frac{3Vol_{\tilde{g}_i}(\partial U_i)}{4\pi^2a_i^3},
\end{align}
by (\ref{eqn: TPX_good}), but then by (\ref{eqn: small_boundary})
$$
\frac{3\vol_{\tilde{g}_i}(\partial U_i)}{4\pi^2a_i^3}\ \le\ \frac{3}{2}e^{R^2+r_iR}\frac{\tilde{\mu}_{f_i}(\tilde{u}_i^{-1}(a_i))}{2\pi^2a_i^3}\ \le\ \frac{1}{4}.
$$
Further notice that $\int_{\partial U_i}\bEuler$ is a topological constant, invariant under rescaling, so the above two estimates confirm (\ref{eqn: bad_bEuler}). 

\begin{remark}
As kindly pointed out by Ruobing Zhang, (\ref{eqn: cover}) and the estimates that follow do not require the specific topological structure, thus we don't have to work within the injectivity radii at regular points, but instead, estimates (\ref{eqn: covering_H_1}) and (\ref{eqn: covering_H_2}) work for balls centered at any point. We wrote the estimates (\ref{eqn: covering_H_1}) and (\ref{eqn: covering_H_2}) only at the scale of injectivity radii because (\ref{eqn: covering_geometry}) gives a more intuitive explanation.
\end{remark}

\subsection{Conclusion of the proof} With the help of the key estimate (Proposition~\ref{prop: KE}) and the fast decay of renormalized energy (Proposition~\ref{prop: decay}), we can now prove Theorem~\ref{thm: main}:
\begin{proof}[Proof of Theorem~\ref{thm: main}]
Let $r_R:=\frac{1}{10}r_{FD}(R)$ and let $\epsilon_R=\min\{\epsilon_{KE}(R),\epsilon_{FD}(R)\}$. Fix some $r<r_R$, and assume that $B(p,r)\subset B(p_0,R)$ has small energy $\int_{B(p,r)}|\Rm|^2\ \dmu_f<\epsilon_R$. It then follows from Proposition~\ref{prop: KE} that $I_{\Rm}^f(p,r)<C_{KE}(R)$. If $I_{\Rm}^f(p,r)<\epsilon_A(R)$ we can apply Anderson's $\epsilon$-regularity theorem directly, or if 
$$\frac{\mu_f(B(p,r\slash2))}{\volfm_R(r\slash 2)}\ge \delta_{KE}(R),$$
we are reduced to the known non-collapsing case, see~\cite{HM10}. Otherwise, we can apply Proposition~\ref{prop: decay} so that $I_{\Rm}^{f}(p,r\slash 2)<(1-\eta_R) C_{KE}(R)$. Performing the same process at most $k_R:=\log_{1-\eta_R}\frac{\epsilon_A(R)}{2C_{KE}(R)}$ many times, we will have $I_{\Rm}^{f}(p,2^{-k_R}r)<\epsilon_A(R)$, whence 
\begin{align}
\sup_{B(p,2^{-k_R-1}r)}|\Rm|\ \le\ C_A(R)4^{k_R}r^{-2}I_{\Rm}^f(p,2^{-k_R}r)^{\frac{1}{2}}.
\label{eqn: pointwise_small_curvature}
\end{align}

Now cover $B(p,r\slash 4)$ by balls of radius $2^{-k_R-1}r$, we have
\begin{align*}
B\left(p, r\slash 4\right)\ \subset\ \bigcup_{q\in B(p,r\slash 4)} B(q,2^{-k_R-1}r)\ \subset\ \bigcup_{q\in B(p,r\slash 4)} B\left(q,r\slash2\right)\ \subset\  B(p, r),
\end{align*}
applying the argument above for each $q\in B(p,r\slash 4)$, we obtain by (\ref{eqn: pointwise_small_curvature}),
$$
\sup_{B(p,\frac{1}{4}r)} |\Rm|\ \le\ C_R\ r^{-2},
$$
with $C_R:=C_A(R)\sqrt{\epsilon_A(R)}\ 4^{k_R+1}$.
\end{proof}




\section{Strong convergence of 4-d Ricci shrinkers} 

In this section we will apply our $\epsilon$-regularity theorem to obtain structural results concerning the convergence and degeneration of the soliton metrics. We first have a straightforward application of Theorem~\ref{thm: main}: 
\begin{proposition}
Let $\{(M_i,g_i,f_i)\}$ be a sequence of complete non-compact 4-d Ricci shrinkers. Suppose 
\begin{align}
\int_{B(p^0_i,R)}|\Rm_{g_i}|^2\ \dmu_{f_i}\ \le\ C(R).
\label{eqn: local_Rm_L2}
\end{align}
Then for each $R>0$ fixed, it sub-converges to some length space $(X_R,d_{\infty})$ in the strong multi-pointed Gromov-Hausdorff sense (see Definition~\ref{defn: strong_convergence}), with $J\le J(R)$ marked points.
\label{prop: convergence}
\end{proposition}

\begin{proof}
Fix any $R>0$. By the assumption (\ref{eqn: local_Rm_L2}), there are only finitely many points $p_i^1,\cdots,p_i^{J_R} \in B(p_i^0,R)$, around which there is a curvature concentration
\begin{align}
\int_{B(p_i^j,r_i)} |\Rm_{g_i}|^2\ \dmu_{f_i}\ \ge\ \epsilon_{R+1},
\end{align}
with $r_i\rightarrow 0$ and $j_R\le C(R+1)\epsilon_{R+1}^{-1}$. On the other hand, for any $q\in B(p_i^0,R)$ outside $\cup_{j=1}^{J_R}B(p_i^1,2r_i)$, we have
$$|\Rm_{g_i}|(q)\ \le\ C_{R+1} r_i^{-2}.$$

By Lemma~\ref{lem: volume_comparison} and Gromov's compactness~\cite{GLP}, there is a compact length space $(X,d_{\infty})$ such that after passing to a subsequence, $(B(p_i^0,R),g_i)\rightarrow_{GH}(X,d_{\infty})$. Clearly $\diam_{d_{\infty}} X \le R$. Moreover, by compactness of $\overline{B(p_i^0,R)}$, possibly passing to a further subsequence, the set of points $\{p_i^1,\cdots,p_i^{J_R}\}$ also Gromov-Hausdorff converge to a set of marked points $\{x_{\infty}^1,\cdots,x_{\infty}^{J_R}\}\subset X$.

Now fix $x\in X\backslash \{x_{\infty}^1,\cdots,x_{\infty}^{J_R}\}$ and assume $B(p_i^0,R)\ni p_i\rightarrow_{GH} x$. Fix 
$$d_x:=\min_{1\le j\le J_R} d_{\infty}(x,x_{\infty}^j),$$ 
then for any $i$ sufficiently large, $d_x> 10r_i$ ($j=1,\cdots,J_R$), and we can conclude as above 
$$\sup_{B(p_i,\frac{1}{4}d_x)}|\Rm_{g_i}|\ \le\ C_{R+1}d_x^{-2},$$
a uniform constant for the sequence $\{B(p_i^0,R)\}$. Thus the Gromov-Hausdorff convergence to any $x\not=x_{\infty}^j$ ($j=1,\cdots, J_R$) is improved to strong convergence in Definition~\ref{defn: strong_convergence}.
\end{proof}

Presumably, as $R\rightarrow \infty$, $j_R\rightarrow \infty$ and the selection of the subsequence of $\{M_i,g_i,f_i\}$ depends on $R$. This is a feature of Ricci solitons different from the Einstein case. However, assuming weighted $L^2$-bound of curvature is much more realistic for non-compact 4-d Ricci shrinkers, compared to non-compact Ricci flat manifolds. For instance, as we will see in the following proof of Theorem~\ref{thm: moduli}, a global weighted $L^2$-curvature bound by the Euler characteristics could be easily obtained if we further assume a uniform scalar curvature bound, see also~\cite{HM10} and~\cite{MWang15}.

From (\ref{eqn: potential_growth}) and (\ref{eqn: gradient_estimate}), we notice that a uniform bound on the scalar curvature, eliminates singularities of $f$ outside a definite ball. It will then be convenient to use (sub-)level sets of $f$ instead of geodesic balls centered at $p_0$. Therefore we use the following notations:
\begin{definition}
Let $(M,g,f)$ be a 4-d Ricci shrinker such that the normalization condition (\ref{eqn: normalization}) is satisfied, and fix $R>0$, we define
$$\quad D(R):=\{x\in M: f(x)<R\}\quad \text{and}\quad \Sigma(R):=\{x\in M: f(x)=R\}=\partial D(R).$$
\end{definition} 

\begin{proof}[Proof of Theorem~\ref{thm: moduli}]
Fix $R_0>1$ so that there is no critical value of $f$ outside $D(R)$.

\noindent \textbf{\textit{Curvature bound outside some $D(R_{MW}).$}} We will start by examining the work of Munteanu-Wang~\cite{MWang15} carefully, and obtain a \emph{uniform} curvature control outside a fixed sized ball around the base point. (We cannot directly quote their results because their estimates involve the curvature of specific manifolds, but we need uniform estimates.) After a detailed study of the level sets $\Sigma(R)$, Munteanu-Wang observed, in Proposition 1.1 of~\cite{MWang15}, the following fundamental estimate for 4-d Ricci shrinkers: there is an absolute constant $c_{1.0}$ such that 
\begin{align*}
c_{1.0}|\Rm|\le \frac{|\nabla \Rc|}{\sqrt{t}}+\frac{|\Rc|^2+1}{f}+|\Rc|
\end{align*}
outside $D(R_{1.0})$. This estimate then enables them to obtain an elliptic inequality about the positive function $u:=|\Rc|^2\Sc^{-a}$ for some $a\in (0,1)$ (see Lemma 1.2 of~\cite{MWang15}): there exists some absolute constant $c_{1.1}>0$ such that
\begin{align*}
\Delta_fu\ge \left(2a-\frac{c_{1.1}}{1-a}\frac{\Sc}{f}\right)u^2\Sc^{a-1}-c_{1.1}u^{\frac{3}{2}}\Sc^{\frac{a}{2}}-c_{1.1}u
\end{align*}
outside $D(R_{1.1})$. For any $R>2\max\{R_0,R_{1.0},R_{1.1}\}$, as done in Proposition 1.3 of~\cite{MWang15}, one can construct a cut off function $\varphi$ supported on $D(3R)\backslash D(R\slash 2)$ such that $\varphi\equiv 1$ on $D(2R)\backslash D(R)$ and $|\nabla \varphi|+|\Delta^f\varphi|\le c_{1.2}$ ($c_{1.2}>0$ being some absolute constant, especially independent of $R$). Now choose $R_{1.2}>R$ and $a\in (0,1)$ such that 
$$2a-\frac{c_{1.2}}{1-a}\frac{\Sc}{f}\ \ge\ 1$$
outside $D(R)$, for any $R>R_{1.2}$, then $a$ becomes an absolute constant. We then obtain inequality (1.14) of~\cite{MWang15}:
$$\varphi^2\Delta^fG\ \ge\ S^{a-1}G^2-c_{1.3}G^{\frac{3}{2}}-c_{1.3}G+2\nabla G\cdot \nabla \varphi^2,$$
where $G:=u\varphi^2$. Applying maximum principle to this inequality we see $G\le c_{1.4}$, i.e. $|\Rc|\le c_{1.4}\Sc^{a}\le c_{1.5}$ on $D(2R)\backslash D(R)$, for any $R>R_{1.2}$. Munteanu-Wang then applied the cut off function and maximum principle to the elliptic inequality (see (1.17) and (1.18) of~\cite{MWang15})
\begin{align*}
\Delta_f(|\Rm|+|\Rc|^2)\ge |\Rm|^2-c_{1.6}\ge \frac{1}{2}(|\Rm|+|\Rc|^2)^2-c_{1.7},
\end{align*}
 whence
\begin{align}
\sup_{M\backslash D(R_{MW})}|\Rm|+|\Rc|^2\ \le\ C_{MW},
\label{eqn: Rm_bound_outside_DR}
\end{align}
for some absolute constants $R_{MW}>1000$ and $C_{MW}>0$, depending only on $\bar{S}$. From this estimate, we notice (as pointed out in~\cite{MWang15}), that under the assumption of uniform scalar curvature bound, the main concern of controlled geometry is about a bounded region $D(R_{MW})$ around the base point.

\noindent \textbf{\textit{Global weighted $L^2$-curvature bound.}} By the non-degeneration of $f$ outside $D(R)$ for any $R>R_{MW}$, we see that $D(R)$ is a smooth retraction of $M$, hence $\chi(M)=\chi(D(R))$ as the Euler characteristic is a homotopy invariant. Recall that for $\Sigma(R)$, the boundary Gauss-Bonnet-Chern term can be estimated as
$$\left|\bEuler\right| \le\ \frac{1}{4\pi^2}\left(\frac{2\left|\det \nabla^2f\right|}{|\nabla f|^3}+3\frac{|\nabla^2 f|}{|\nabla f|}|\Rm|\right),$$ 
since $|\nabla f|>1$ and $|\Rm|\le C_{MW}$ outside $D(R_{MW})$, we then have
\begin{align}
\int_{D(R)}|\Rm|^2\ \dmu_f\ \le\ \bar{E} +c_{2.0}\int_{\Sigma(R)} (|\nabla^2f|^{3}+|\nabla^2 f|)\ \darea_{\Sigma(R)}.
\end{align}
The defining equation (\ref{eqn: defn}) then gives
$$\int_{\Sigma(R)}|\nabla^2 f|^3+|\nabla^2 f|\ \darea_{\Sigma(R)}\ \le\ c_{2.1}\int_{\Sigma(R)}|\Rc|^3+|\Rc|\ \darea_{\Sigma(R)}\ \le\ c_{2.2}Vol(\Sigma(R)).$$
On the other hand, (\ref{eqn: potential_growth}) and Lemma~\ref{lem: volume_growth} gives control 
$$Vol(D(3R_{MW}))-Vol(D(2R_{MW}))\le c_{2.3}R_{MW}^2.$$
For some $R_2\in [2R_{MW},3R_{MW}]$ such that $Vol(\Sigma(R_2))=\min_{2R_{MW}\le R\le 3R_{MW}}Vol(\Sigma(R)),$ 
we can apply the coarea formula and (\ref{eqn: gradient_estimate}) to estimate
\begin{align*}
Vol(\Sigma(R_2))\ \le\ \frac{1}{R_{MW}}\int_{D(3R_{MW})\backslash D(2R_{MW})} |\nabla f|\ \dvol\ \le\ c_{2.4}R_{MW}^{\frac{3}{2}}.
\end{align*}
These inequalities together give:
\begin{align*}
\int_{M}|\Rm|^2\ \dmu_f\ &\le\ \int_{D(R_2)}|\Rm|^2\ \dvol+ C_{MW}^2\int_{M\backslash D(R_2)}1\ \dmu_f\\
&\le\ \chi(D(R_2))+c_{2.2}Vol(\Sigma(R_2))+c_{2.3}\sum_{k=1}^{\infty}e^{-2^kR_{MW}}8^kR_{MW}^2\\
&\le\ \chi(M)+c_{2.5}R_{MW}^{\frac{3}{2}}+c_{2.6}\\
&\le\ C(\bar{E},\bar{S}),
\end{align*}
since all the constants involved are solely determined by $\bar{E}$ and $\bar{S}$. Here we recall that $\bar{E}>0$ and $\bar{S}>0$ are the prescribed upper bounds of the Euler characteristics (in absolute value) and the scalar curvature, respectively.

With this bound at one hand, we can apply Proposition~\ref{prop: convergence} to $\{D_i(2R_{MW})\subset M_i\}$ and obtain a convergent subsequence, to some metric space $(X_{\infty}(2R_{MW}),d_{\infty})$ with marked points $\{x_{\infty}^1,\cdots,x_{\infty}^J\}$ and $J\le J(2R_{MW})$. On the other hand, we have a uniform curvature bound outside $D_i(2R_{MW})$, whence a non-compact length space $(X_{\infty},d_{\infty})$ as the Gromov-Hausdorff limit. The convergence will preserve the finitely many marked points, and away from these points, the Gromov-Hausdorff convergence is improved, by the locally uniform curvature bound, to strong multi-pointed Gromov-Hausdorff convergence in the sense of Definition~\ref{defn: strong_convergence}.
\end{proof}

\section{Discussion}

As pointed out in the introduction, the ultimate goal of studying the collapsing of 4-d Ricci shrinkers is to rule out the potential collapsing and to obtain a uniform lower entropy bound. At this point, we propose the following conjecture of Bing Wang:
\begin{conjecture}[Bing Wang]
Given $\bar{E}>0$. If $(M,g,f)$ is a complete non-compact four dimensional shrinking Ricci soliton, then $\chi(M)$, the Euler characteristic of $M$, is finite. Assume $|\chi(M)|\le \bar{E}$, then either $(M,g)\equiv \left(\mathbb{R}\times \mathbb{S}^3\slash \Gamma, \text{d}t^2\oplus g_{\mathbb{S}^3}\right)$ for some finite isometry group $\Gamma$, or else there exists some $\bar{\omega}>0$, depending only on $\bar{E}$, such that 
$$\mathcal{W}(M,g,f)\ge -\bar{\omega}.$$
\label{thm: conjecture_entropy}
\end{conjecture}
If this conjecture is confirmed, the task of classifying $4$-d Ricci shrinkers will be reduced to the classification with a given Euler characteristic bound, with the help of the resulting uniform entropy lower bound. See~\cite{LiWang17} for several uniform estimates in this situation.

In the setting of mean curvature flows, Lu Wang~\cite{Wanglu1610} has recently studied the asymptotic behaviors of self-shrinkers of finite topology. Motivated by her result, we may expect a uniform scalar curvature bound for 4-d Ricci shrinkers with bounded Euler characteristics:
\begin{conjecture}
Given $\bar{E}>0$. If $(M,g,f)$ is a four dimensional shrinking Ricci soliton whose Euler characteristic is bounded above by $\bar{E}$ in absolute value, then there exists some $\bar{S}>0$, depending only on $\bar{E}$, such that 
$$\sup_M \Sc_g\le \bar{S}.$$
\label{thm: conjecture_scalar}
\end{conjecture}
Clearly, the confirmation of this conjecture will pave paths towards the proof of Conjecture~\ref{thm: conjecture_entropy}. See also~\cite{MWang16} for recent progress.

\appendix

\section{Collapsing and equivariant good chopping}
The equivariant good chopping theorem when collapsing with locally bounded
curvature happens, as stated and used in~\cite{CT05}, is a generalization of
the original work of Cheeger-Gromov~\cite{CGIII} in two directions: in one
direction, the global curvature bound is relaxed to locally bounded curvature,
as carried out by Cheeger-Tian in the proof of Theorem 3.1 of~\cite{CT05}; in
the other, since the collapsing does not imply the existence of an isometry
group action --- the action being only by a sheaf of local isometries --- more
elaborations are needed to reduce the situation to the case considered
in~\cite{CGIII}. In this appendix, with respect to the proof given
in~\cite{CT05}, we provide additional details that were indicated but not
written out explicitly.

Fix $a\in (0,1)$ throughout this appendix. For the sake of simplicity, we will
assume the given metric to be locally regular under curvature scale, i.e.
\begin{enumerate}
\item[(R)] there exist $A_k>0$ for $k=0,1,2,\cdots$, such that
$$\sup_{B(p,l_a(p))}|\Rm_g|\ \le\ l_a(p)^{-2}\quad \Longrightarrow\quad
\sup_{B(p,\frac{1}{2}l_a(p))}|\nabla^k\Rm_g|\ \le\ A_k\ l_a(p)^{-2-k}.$$
\end{enumerate}
\begin{theorem}
Let $(M,g)$ be an $n$-dimensional Riemannian manifold satisfying property
$(R)$. There exist constants $\delta_{GC}>0$ and $C_{GC}(n)>0$ such that if
$E\subset M$ is $(\delta,a)$-collapsing with locally bounded curvature for some
$\delta<\delta_{GC}$ and $a\in (0,1)$, then there is an open subset $U\subset
B(E,\frac{a}{2})$ that contains $E$, saturated by some $a$-standard
$N$-structure, and has a smooth boundary $\partial U$ with $$|II_{\partial U}|\
\le\ C_{GC}l_a^{-1}.$$
\end{theorem}


Fukaya's frame bundle argument~\cite{Fukaya88} enables us to overcome this
difficulty. Basically, we first lift to the frame bundle, where the collapsing
can only produce mutually diffeomorphic nilpotent orbits with controlled second
fundamental form. Then we apply the equivariant good chopping theorem of
Cheeger-Gromov~\cite{CGIII} to obtain a good neighborhood that is both
invariant under the nilpotent structure and the $O(n)$-actions. Taking the
quotient of this neighborhood by $O(n)$, we get the desired neighborhood on the
original manifold, because the $O(n)$-action commutes with the local actions of
the nilpotent structure.

We remark that the proof of this theorem utilizes Sections 3-7 of
Cheeger-Fukaya-Gromov's structural theory about the geometry of collapsing with
bounded curvature developed in~\cite{CFG92}, and its generalization to the case
of collapsing with locally bounded curvature by Cheeger-Tian~\cite{CT05}: to
begin with, we need the existence of a regular approximating metric on the
frame bundle, invariant under the nilpotent action resulted from the
collapsing. See  for a detailed description.



\subsubsection*{\textbf{Regularity of the frame bundle}} Consider the frame
bundle $FB(E,a)$, with each fiber diffeomorphic to $O(n)$ and $\pi:
FB(E,a)\rightarrow B(E,a)$ the natural projection. We follow the conventions of
Notation 1.3 in~\cite{Fukaya88}. Let $\bar{g}$ denote the Riemannian metric on
$FB(E,a)$, as defined in 1.3 of~\cite{Fukaya88}. Moreover, for any object $o$
associated to $B(E,a)$, we will let $\bar{o}$ denote the corresponding object
associated to $FB(E,a)$.

For any $p\in E$, do the rescaling $\bar{g}\mapsto
l_a(p)^{-2}\bar{g}=:\bar{g}_p$, then by (R) we can control, for $\bar{p}\in
\pi^{-1}(p)$, 
\begin{align*}
\sup_{FB(\bar{p},\frac{1}{2})}|\nabla^k
\Rm_{\bar{g}_p}|_{\bar{g}_p}\ \le\ B'_k(n,A_{\le k},l_a(p))\ \le\ B_k(n,A_{\le
k}),
\end{align*}
where we use $A_{\le k}$ to denote $A_1,\cdots,A_k$. This  because for $a<1$ the
rescaling will stretch the fiber metric on $O(n)$, making it less curved. This
means, in the original metric,
\begin{enumerate}
\item[(R1)] $\quad \sup_{FB(p,\frac{1}{2}l_a(p))}|\nabla^k
\Rm_{\bar{g}}|_{\bar{g}}\ \le\ B_k(n,A_{\le k})\ l_a(p)^{-2-k}$ for
$k=0,1,2,3,\cdots.$
\end{enumerate}

Now we use Lemma~\ref{lem: covering} to construct a good covering of
$B(E,\frac{a}{2})$, by $B_i:=B(p_i,2\zeta l_a(p_i))$ contained in $B(E,a)$.
Clearly $FB(E,\frac{a}{2})\subset \cup_iFB_i$.

\subsubsection*{\textbf{Fibration and invariant metric of the frame bundle}} We
first assume $\delta<\delta_{CFGT}$. Arguing as before, we notice that if
$B_i\cap B_j\not=\emptyset$, then (\ref{eqn: l_a_Harnack}) ensures that the
curvature of the frame bundle also satisfies for $k=0,1,2,\cdots,$
$$\sup_{FB_i\cup FB_j}|\nabla^k\Rm_{\bar{g}}|_{\bar{g}}\ \le\
\left(\frac{1-2\zeta}{1+2\zeta}\right)^{-2-k}B_k(n,A_{\le
k})\max\{l_a(p_i),l_a(p_j)\}^{-2-k}.$$ Thus rescaling $\bar{g}\mapsto
l_a(p_i)^{-2}\bar{g}=:\bar{g}_{ij}$ on $B_i\cup B_j$ will ensure for
$k=0,1,2,\cdots$,
\begin{align*}
\sup_{FB_i\cup FB_j}|\nabla^k \Rm_{\bar{g}_{ij}}|_{\bar{g}_{ij}}\ \le\
B_k(n,A_{\le k})\left(\frac{1-2\zeta}{1+2\zeta}\right)^{-2-k},
\end{align*}
so that we can think as on $FB_i\cup FB_j$ there is a uniformly regular
Riemannian metric $\bar{g}_{ij}$.


 By Lemma~\ref{lem: covering}, we notice that in each step of carrying out the
 procedure of Sections 3-7, especially applying Proposition A2.2
 of~\cite{CFG92}, we only need to deal with the case of smoothing within a
 single $FB_i$. Thus the above regularity of the metric restricted to
 intersecting balls is sufficient, and we can construct the following data:
\begin{enumerate}
\item[(F1)] there is a global fibration $f: FB(U,2a\slash 3)\rightarrow Y$;
\item[(F2)] $Y$ is a smooth Riemannian manifold of dimension $m'<\dim FE$;
\end{enumerate}
There is a simply connected nilpotent Lie group $\bar{N}$ of dimension $n-m$,
and a co-compact lattice $\Lambda$, such that:
\begin{enumerate}
\item[(N1)] $\bar{N}$ acts on $\cup_iFB_i$ so that each orbit
$\mathcal{N}(\bar{x})$ at some $\bar{x}\in \cup_iFB_i$ is a compact
submanifold, and up to a finite covering, $$\bar{N}\slash \Lambda\ \approx\
\mathcal{N}(\bar{x})\ =\ f^{-1}(f(\bar{x}));$$
\item[(N2)] the action of $\bar{N}$ commutes with the $O(n)$-action;
\item[(N3)] the action of $\bar{N}$ on $FB(E,a)$, after taking the $O(n)$
quotient, descends to the $a$-standard $N$-structure on $B(E,\frac{a}{2})$, as
described in Theorem~\ref{thm: CFGT}.
\end{enumerate}
Moreover, for any positive $\varepsilon$ which could be arbitrarily small,
there is a smooth metric $\bar{g}^{\epsilon}$ on $FB(E,a)$ and a constant
$\alpha_0=\alpha_0(n,a)>0$ such that:
\begin{enumerate}
\item[(G1)] $\bar{g}^{\epsilon}$ is a regular $\epsilon$-approximation of
$l_a(p_i)^{-2}\bar{g}|_{FB_i}$ for each $i$, see Theorem~\ref{thm: CFGT};
\item[(G2)] $\bar{g}^{\epsilon}$ is invariant under both the actions of
$\bar{N}$ and of $O(n)$;
\item[(G3)] $\forall \bar{x}\in FB_i$,
$\diam_{\bar{g}^{\epsilon}}\mathcal{N}(\bar{x})\ <\ \epsilon l_a(p_i)$ for each
$i$.
\item[(G4)] $\forall \bar{x}\in FB_i$, the normal injectivity radius
$\text{inj}^{\perp}_{\bar{g}^{\epsilon}}\ \bar{x}\ge \frac{2}{3}\alpha_0
l_a(p_i);$
\item[(G5)] $\forall \bar{x}\in FB_i$, $|II_{\mathcal{N}(\bar{x})}|\ \le\
C(B_{\le 2}(n,A_{\le 2}))\ l_a(p_i)^{-1}$.
\end{enumerate}
Without loss of generality, we may assume that $B_k(n,A_{\le k})\ge 1$ and
$\alpha_0 \le 1$.

Here we make a simple convention: $\forall X\subset FB(E,a\slash 2)$, let
$\mathcal{N}(X)$ and $\mathcal{O}(X)$ denote the orbits of $X$ under the action
of $\bar{N}$ and $O(n)$, respectively. Since both actions are local isometries
(G2), and they commute (N2), we have:
\begin{enumerate}
\item[(G6)] the operations $\mathcal{N}(-)$, $\mathcal{O}(-)$ and $B(-,r)$
(with respect to $\bar{g}^{\varepsilon}$) for $r\in (0,a\slash 2)$ on subsets of
$FB(E,a\slash 2)$ commute.
\end{enumerate}

We need to further notice that for $\epsilon>0$ arbitrarily small, we can
choose $\delta$ small enough so that $B(E,a)$ being $(\delta,a)$-collapsing
with locally bounded curvature implies the existence of the approximating
metric above, with the given $\epsilon$. Notice that as long as
$\delta<\delta_{CFGT}$, the existence of $\alpha_0$ and the $\bar{N}$-structure
is guaranteed. Here we fix $\epsilon=10^{-10}\alpha_0$, and let
$\delta_{GC}<\delta_{CFGT}$ be one that works for the fixed $\epsilon$. In
practice, once there exists some $\delta'<\delta_{CFGT}$, then there exists a
family of Riemannian metrics that are $(\delta,a)$-collapsing with locally
bounded curvature with $\delta\rightarrow 0$ (see~\cite{CGII}
and~\cite{Fukaya89}), so eventually $\delta<\delta_{GC}$.

\subsubsection*{\textbf{Distance to orbits}}Recall that we hope to smooth the
boundary of $E$. This smoothing will be obtained by taking certain level set of
a smoothing of the distance function to $\mathcal{N}(FE)$. Here for any
$O(n)$-invariant $\bar{U}\subset \cup_iFB_i$, we define the ``distance to
orbits of $\bar{U}$'' as following:
$$\bar{\rho}_{\bar{U}}:\cup_i FB_i\rightarrow [0,\infty)\quad \bar{x}\mapsto
d_{\bar{g}^{\epsilon}}\left(\bar{x},\mathcal{N}(\bar{U})\right).$$ Notice that
by (N2), $\mathcal{N}(\bar{U})$ is invariant under the $O(n)$-action:
$$\forall \gamma \in O(n),\quad
\gamma\mathcal{N}(\bar{U})=\mathcal{N}(\gamma\bar{U})=\mathcal{N}(\bar{U}).$$
Then $\bar{\rho}_{\bar{U}}$ immediately satisfies the following properties:
\begin{enumerate}
\item[(D1)] $\bar{\rho}_{\bar{U}}$ is invariant under the actions of
$\bar{N}$ and $O(n)$ by (G2);
\item[(D2)] $\forall \bar{x}\in \cup_iFB_i$, $\bar{\rho}_{\bar{U}}(\bar{x})\
\le\ d_{\bar{g}^{\epsilon}}(\bar{x},\bar{U})$, and thus
$$\bar{\rho}_{\bar{U}}^{-1}([0,a\slash 4])\subset B(\bar{U},a\slash 4).$$
\end{enumerate}
The possible non-smoothness is caused by the behavior of $\partial \bar{U}$,
since the distance to a single orbit is smooth within the normal injectivity
radii, by (R1), (G1) and (G4): defining
$d^{\bar{x}_0}(\bar{x}):=d_{\bar{g}^{\epsilon}}(\bar{x},\mathcal{N}(\bar{x}_0))$
for some fixed $\bar{x}_0\in \cup_iFB_i$, then for $k=0,1,2,\cdots$, we have
\begin{enumerate}
\item[(D3)]
$\sup_{B\left(\mathcal{N}(\bar{x}_0),\frac{\alpha_0}{10}l_a(p_i)\right)}|\nabla^k
d^{\bar{x}_0}|\ \le\ C_kl_a(p_i)^{1-k};$
\item[(D4)] $\bar{d}^{\bar{x}_0}$ is both $\bar{N}$-and $O(n)$-invariant;
\item[(D5)] $\bar{\rho}_{\bar{U}}=\inf_{\bar{x}_0\in \bar{U}}\bar{d}^{\bar{x}_0}$.
\end{enumerate}


\subsubsection*{\textbf{Local parametrization of the frame bundle}} For each
$\bar{q}\in FB_i\cap FE$, we start with setting
$\bar{H}^0:=B\left(\mathcal{N}(\bar{q}),\frac{\alpha_0}{2}l_a(p_i)\right)$ and
$\bar{H}:=\mathcal{O}(\bar{H}^0)$. Notice that by (G4), the normal injectivity
radius, constant on $\mathcal{N}(\bar{q})$, satisfies
$\text{inj}^{\perp}_{\bar{q}}\ge \frac{2\alpha_0}{3}l_a(p_{i})$. We can deduce
that $f(\bar{H}^0)$ is contractible and $\bar{H}^0$ deformation retracts to
$\mathcal{N}(\bar{q})$, therefore we can find, possibly after lifting to a
finite covering, a global orthonormal frame, consisting of left invariant
vector fields $\xi_1,\cdots,\xi_{m'}\perp \mathcal{N}(\bar{q})$, so that
$\forall \gamma \in \bar{N}$, the map
\begin{align*}
\exp^{\perp}_{\gamma\bar{q}}: B\left(\mathbf{0},
\frac{\alpha_0}{2}l_a(p_{i})\right)\ \rightarrow\ \bar{H}^0, \quad
\mathbf{v}=(v^1,\cdots,v^{m'})^T\ \mapsto\ \exp_{\gamma
\bar{q}}\left(\sum_{s=1}^{m'}v^s\xi_s(\gamma\bar{q})\right)
\end{align*}
is injective and diffeomorphic onto its image, where $B\left(\mathbf{0},
\frac{\alpha_0}{2}l_a(p_{i})\right)\subset \mathbb{R}^{m'}$.

According to (G2) and the definition, we immediately notice that
\begin{enumerate}
\item[(P1)] $\forall \gamma \in \bar{N}$ and $\forall \bar{x}\in \bar{H}^0$,
$Image (\exp_{\gamma\bar{q}}^{\perp})\perp \mathcal{N}(\bar{x})$;
\item[(P2)] $\forall \gamma\in \bar{N}$,
$(\exp_{\bar{q}}^{\perp})^{\ast}\bar{g}^{\epsilon}\ =\ (\exp_{\gamma
\bar{q}}^{\perp})^{\ast}\bar{g}^{\epsilon}$ on $B\left(\mathbf{0};
\frac{\alpha_0}{2}l_a(p_{i})\right)$;
\item[(P3)] $\forall \gamma \in \bar{N}$ and $\forall \bar{x}\in \bar{H}$,
$\exists ! \mathbf{v}_{\bar{x}}\in B\left(\mathbf{0},
\frac{\alpha_0}{2}l_a(p_{i})\right)$ such that
$\bar{\rho}_{\bar{U}}(\gamma\bar{x})\ =\
\bar{\rho}_{\bar{U}}(\exp_{\bar{q}}^{\perp}(\mathbf{v}_{\bar{x}}))$.
\end{enumerate}
We can consider the pull-back metric
$h_i:=(\exp_{\bar{q}}^{\perp})^{\ast}\bar{g}^{\epsilon}$ on
$B\left(\mathbf{0},\frac{\alpha_0}{2}l_a(p_i)\right)$, as a positive definite
2-tensor field, so that:
\begin{enumerate}
\item[(P4)]  according to (G1) and (G2), for any multi-index $I$ with
$|I|=k=0,1,2,\cdots$, $$\left|\frac{\partial^{|I|}}{\partial v^I} h_i\right|\
\le\ C_kl_a(p_{i})^{-k};$$
\item[(P5)] $B\left(\mathbf{0},\frac{\alpha_0}{2}l_a(p_{i})\right)$ is
geodesically convex under the metric $d_{h_i}$ defined by $h_i$;
\item[(P6)] $\forall \mathbf{v}\in B(\mathbf{0},\frac{\alpha_0}{2}l_a(p_{i}))$,
$d_{h_i}(\mathbf{v},\mathbf{0})\ =\
d^{\bar{q}}(\exp_{\bar{q}}^{\perp}(\mathbf{v})).$
\end{enumerate}

\subsubsection*{\textbf{Local smoothing and chopping}} In order to smooth
$\bar{\rho}_{\bar{U}}$, we mollify it by a smooth cut-off function within the
normal injectivity radius of $\bar{q}$, following~\cite{CGvonNeumann}. Let
$0\le \varphi_i\le 1$ be a smooth function such that for some $\zeta'>0$ to be
determined later,
\begin{enumerate}
\item[(S1)]  $\varphi_i$ is supported on
$[0,\frac{\zeta'\alpha_0}{100}l_a(p_{i}))$ and $\varphi_i(t)\equiv 1$ for $t\in
[0,\frac{\zeta'\alpha_0}{200}l_a(p_{i})]$;
\item[(S2)] $\varphi_i^{(k)}(t)\le C_k(\zeta')l_a(p_{i})^{-k}$ for
$k=0,1,2,\cdots$ and $t\in [0,\frac{\zeta'\alpha_0}{100}l_a(p_{i}))$.
\end{enumerate}

Now we focus on an $O(n)$-invariant $\bar{U}\subset \bar{H}^0$, and define on
$\bar{H}^0$:
\begin{align*}
\bar{\rho}_{\bar{U}}^{\sharp}(\bar{x}):=\frac{1}{\mu_i(\bar{x})}\int_{B\left(\mathcal{N}(\bar{p}),\frac{\alpha_0}{2}l_a(p_{i})\right)}\bar{\rho}_{\bar{U}}(\bar{z})\varphi_i(d^{\bar{x}}(\bar{z}))\
\dvol_{\bar{g}^{\epsilon}}(\bar{z}),
\end{align*}
where
$$\mu_i(\bar{x}):=\int_{B\left(\mathcal{N}(\bar{q}),\frac{\alpha_0}{2}l_a(p_{i})\right)}\varphi_i(d^{\bar{x}}(\bar{z}))\
\dvol_{\bar{g}^{\epsilon}}(\bar{z}).$$ Notice that in the definition of
$\bar{\rho}_{\bar{U}}^{\sharp}$ we have taken average by dividing
$\mu_i(\bar{x})$, thus the numerical value of $\bar{\rho}_{\bar{U}}^{\sharp}$
is not affected if we lift the original neighborhood to a finite covering.

By the invariance of $\bar{\rho}_{\bar{U}}$ and that $\exp_{\bar{x}}^{\perp}$
being a diffeomorphism onto its image,
we can reduce $\bar{\rho}_{\bar{U}}^{\sharp}$ to a function
$\tilde{\rho}_{\bar{U}}^{\sharp}$ on
$B\left(\mathbf{0},\frac{\alpha_0}{2}l_a(p_{i})\right)$:
$$\tilde{\rho}_{\bar{U}}^{\sharp}(\mathbf{v})\ :=\
\bar{\rho}_{\bar{U}}^{\sharp}(\exp_{\bar{q}}^{\perp}(\mathbf{v})).$$ The most
important property of $\tilde{\rho}_{\bar{U}}^{\sharp}$ is:
\begin{enumerate}
\item[(S3)] $\quad |\nabla_{\perp}^k\ \bar{\rho}_{\bar{U}}^{\sharp}|(\bar{x})\
=\ |\nabla^k
\tilde{\rho}_{\bar{U}}^{\sharp}|\left((\exp_{\bar{q}}^{\perp})^{-1}(\bar{x})\right)\quad$
for $k=0,1,2,3,\cdots$.
\end{enumerate}

Then by Fubini's theorem and the invariance of $\bar{\rho}_{\bar{U}}$ under the
$\bar{N}$-action,
$$\tilde{\rho}_{\bar{U}}^{\sharp}(\mathbf{v})=\frac{1}{\mu_i(\mathbf{v})}\int_{B\left(\mathbf{0},\frac{\alpha_0}{2}l_a(p_{i})\right)}\bar{\rho}_{\bar{U}}(\exp_{\bar{q}}^{\perp}(\mathbf{w}))\varphi_i\left(d_{h_i}\left(\mathbf{v},\mathbf{w}\right)\right)\
\psi(\mathbf{w})\ \dvol_{h_i}(\mathbf{w}),$$ where
$$\psi(\mathbf{w}):=Vol_{\bar{g}^{\epsilon}}(\mathcal{N}(\exp_{\bar{q}}^{\perp}(\mathbf{w})))\quad\text{and}\quad
\mu_i(\mathbf{v}):=\int_{B\left(\mathbf{0},\frac{\alpha_0}{2}l_a(p_{i})\right)}\varphi_i\left(d_{h_i}\left(\mathbf{v},\mathbf{w}\right)\right)\
\psi(\mathbf{w})\ \dvol_{h_i}(\mathbf{w}).$$ The smoothness of
$\tilde{\rho}_{\bar{U}}^{\sharp}(\mathbf{v})$ then follows from differentiating
$\varphi_i\left(d_{h_i}(\mathbf{v},\mathbf{w})\right)$ with respect to
$\mathbf{v}\in B\left(\mathbf{0},\frac{\alpha_0}{2}l_a(p_{i})\right)$, and the
derivative bounds are guaranteed by (S2) and (P4):
\begin{enumerate}
\item[(S4)] $\quad
\sup_{B\left(\mathbf{0},\frac{\alpha_0}{2}l_a(p_{i})\right)}|\nabla^k
\tilde{\rho}_{\bar{U}}^{\sharp}|\ \le\ C_kl_a(p_{i})^{1-k}\quad$ for $k=1,2,3,\cdots.$
\end{enumerate}

Now we can apply Yomdin's quantitative Morse Lemma (see~\cite{Burguet11}
and~\cite{LoiP12} for proofs) to the function
$\tilde{\rho}_{\bar{U}}^{\sharp}$, to find, for small $\eta>0$, some interval
$J_{\bar{U}}\subset [0,\frac{\alpha_0}{4}l_a(p_{i})]$ of length
$|J_{\bar{U}}|\approx \Psi_{YM}(a,n,\eta) l_a(p_{i})>0$ such that $$ \forall t\in
J_{\bar{U}},\quad |\nabla \tilde{\rho}_{\bar{U}}^{\sharp}|\ >\ \eta\quad \text{on}\quad 
(\tilde{\rho}_{\bar{U}}^{\sharp})^{-1}(t).
$$
Here we may assume the definite constant $\Psi_{YM}(a,n,\eta)<10^{-2}$.
 
By the definition of $\tilde{\rho}_{\bar{U}}^{\sharp}$ and (S3), we then have
\begin{enumerate}
\item[(Y1)] $\quad \forall t\in J_{\bar{U}},\quad |\nabla
\bar{\rho}_{\bar{U}}^{\sharp}|\ \ge\ |\nabla_{\perp}\
\bar{\rho}_{\bar{U}}^{\sharp}|\ >\ \eta\quad \text{on}\quad
(\bar{\rho}_{\bar{U}}^{\sharp})^{-1}(t).$
\end{enumerate}
Let $\bar{W}^0_{\bar{U}}:=(\bar{\rho}^{\sharp}_{\bar{U}})^{-1}([0,t])$ for some
$t\in J_{\bar{U}}$. We then have
\begin{enumerate}
\item[(C1)] $\bar{W}^0_{\bar{U}}$ is invariant under the actions of $\bar{N}$,
by (A3.1) above, and $O(n)$ acts as local isometry on $\bar{W}^0_{\bar{U}}$;
\item[(C2)] $\bar{U}\cap \bar{H}^0\subset \bar{W}^0_{\bar{U}}\subset
B(\bar{U},\frac{\alpha_0}{2}l_a(p_{i}))$;
\end{enumerate}
Define
$\Sigma_{\bar{U}}:=\exp_{\bar{q}}^{\perp}\left((\tilde{\rho}_{\bar{U}}^{\sharp})^{-1}([0,t_i])\right)$,
then $\bar{W}^0_{\bar{U}}=\mathcal{N}(\Sigma_{\bar{U}})$.
We immediately have the bound
\begin{enumerate}
\item[(C3)] $\quad |II_{\partial \Sigma_{\bar{U}}}|\ \le\ \frac{|\nabla^2
\tilde{\rho}_{\bar{U}}^{\sharp}|}{|\nabla \tilde{\rho}_{\bar{U}}^{\sharp}|}\
\le\ C(\eta)l_a(p_{i})^{-1}.$
\end{enumerate}
This, together with $|II_{\mathcal{N}(\bar{x})}|\le Cl_a(p_{i})^{-1},$ and the
fact that $\forall \gamma\in \bar{N}$, $\Sigma_{\bar{U}}\mapsto \gamma
\Sigma_{\bar{U}}$ is an isometry, give
\begin{enumerate}
\item[(C4)] $\quad |II_{\partial
\bar{W}^0_{\bar{U}}}|
\le C(\eta)l_a(p_{i})^{-1}.$
\end{enumerate}

We notice that the function $\bar{\rho}_{\bar{U}}^{\sharp}$ is locally
$O(n)$-invariant, therefore it extends smoothly from $\bar{H}^0$ to $\gamma
\bar{H}^0$ for any $\gamma\in O(n)$.
Thus we see that $\bar{W}_{\bar{U}}:=\mathcal{O}(\bar{W}^0)$ has a smooth
boundary on $\bar{H}=\mathcal{O}(\bar{H}^0)$, whence an $O(n)$-invariant
neighborhood of $\bar{U}\subset \bar{H}$. Moreover, since each $\gamma\in O(n)$
acts as an isometry, we have
\begin{enumerate}
\item[(C5)]$\quad |II_{\partial \bar{W}_{\bar{U}}}|=|II_{\partial
\bar{W}^0_{\bar{U}}}|\le C(\eta)l_a(p_{i})^{-1}$.
\end{enumerate}



\subsubsection*{\textbf{A refined good covering of the frame bundle}} We start
with fixing $\zeta=10^{-2}$ (see
Lemma~\ref{lem: covering}) and
$$\zeta':=\min\{0.1,\zeta\slash \alpha_0\}.$$

Choose a maximal set of points $\{q_j\}\subset E$, such that 
$$d_g(q_j,q_{j'})\ \ge\ \zeta'\alpha_0\min\{l_a(p_{i_j}),l_a(p_{i_{j'}})\},$$
and obtain a covering of $E$ by $\left\{B\left(q_j,2\zeta'\alpha_0
l_a(p_{i_j})\right)\right\}$ (with respect to the original metric $g$ on $M$).
Here for each $q_j$, $p_{i_j}$ is chosen as any $B_i$ containing $q_j$. Then by
Lemma~\ref{lem: covering}, we can find a finite number of sub-collections
$S_j'$ ($j=1,\cdots,N$), such that $E_{j,k}:=B(q_{j,k},2\zeta'\alpha_0
l_a(p_{i_{j,k}}))$ is disjoint from any $E_{j,k'}$, and intersects with at most
one $E_{j',k''}$ for $j'\not=j$.
 
Now the sets $\bar{E}_{j,k}=\pi^{-1}(E_{j,k})$ cover $FE$, and each
$\bar{E}_{j,k}$ is obviously $O(n)$ invariant. Fix $\bar{q}_{j,k}\in
\pi^{-1}(q_{j,k})$ for each $(j,k)$. Since by (G1), $\bar{g}^{\epsilon}$ is a
regular $\epsilon$ approximation of the original metric on $FB(E,a)$, with
$\epsilon<10^{-5}\zeta'\alpha_0$ as defined, we can \emph{redefine}
$$\bar{E}_{j,k}:=\mathcal{O}\left(B\left(\bar{q}_{j,k},
2\zeta'\alpha_0l_a(p_{i_{j,k}})\right)\right)\subset FB(E,a),$$ so that the
covering property and the partition into finitely many sub-collections are
still satisfied.

We further define
$\bar{D}_{j,k}^0:=\mathcal{N}(\bar{E}_{j,k})$,
then by (G6) and (G3), we have $$\bar{D}_{j,k}^0\ =\
\mathcal{O}\left(B\left(\mathcal{N}(\bar{q}_{j,k}),
2\zeta'\alpha_0l_a(p_{i_{j,k}})\right)\right)\ \subset\
B\left(\bar{E}_{j,k},\frac{(1+\zeta)\
\alpha_0}{10^{10}(1-\zeta)}l_a(p_{i_{j,k}})\right),$$ therefore
$\{\bar{D}^0_{j,k}\}$ still forms a covering, and could be divided into
finitely many sub-collections $S'_1,\cdots,S'_N$ as obtained above.

The point of constructing this new covering is that the original covering is
with respect to the original metric $g$, and we need to refinish it so that each
open set of the new covering is saturated by the nilpotent and orthogonal
group actions, yet the whole collection of open sets could still be divided
into $N$ disjoint sub-collections, a necessity for our future step-by-step
gluing to obtain the global chopping.

Now we define
$\bar{D}_{j,k}^m:=B\left(\mathcal{N}(\bar{q}_{j,k}),r_{j,k}^m\right)$, with
$r^m_{j,k}:=(2+\frac{1}{6}m)\zeta'\alpha_0l_a(p_{i_{j,k}})$, for each $m\in
\{0,1,2,3,4,5,6\}$. 
We need this fattening of open sets in the covering since later we will need
to ``glue'' the local smoothings, see the forthcoming claim.

\subsubsection*{\textbf{Global chopping}} We now do the final step, the global
chopping. The method we follow is briefly given in~\cite{CGIII}, where the
curvature is assumed to be uniformly bounded, here we take the (changing)
truncated curvature scale into consideration.

 For the collections
$S'_1,\cdots, S'_N$, we first do the above local chopping for each $FE\cap
\bar{D}_{j,k}$ to obtain $\bar{W}_{j,k}$ with $t_{1,k}\approx
2^{-1}\Psi_{YM}(a,n,\eta)l_a(p_{i_{1,k}})$, and define
$\bar{U}_1:=\cup_{k}\bar{W}_{1,k}$ as an open subset of $\cup_iFB_i$.

For the second step, we modify members of $S'_2$. Notice that if some
$\bar{D}_{2,k}$ intersects some $\bar{W}_{1,k'}$ non-trivially, then we have
the estimates of the truncated curvature scales as before
\begin{enumerate}
\item[(H1)] $\quad \min\{l_a(p_{i_{1,k'}}),l_a(p_{i_{2,k}})\}\ \le\
\max\{l_a(p_{i_{1,k'}}),l_a(p_{i_{2,k}})\}\ \le\
\frac{1+\zeta'}{1-\zeta'}\min\{l_a(p_{i_{1,k'}}),l_a(p_{i_{2,k}})\}.$
\end{enumerate}
Renormalizing $g\mapsto l_a(p_{i_{2,k}})^{-2}g=:g_{2,k}$ will ensure that 
$$\sup_{\bar{W}_{1,k'}\cup
\bar{D}^5_{2,k}}|\Rm_{\bar{g}_{2,k}}|_{\bar{g}_{2,k}}\ \le
C\frac{1+\zeta'}{1-\zeta'},$$ with corresponding bounds on $|\nabla^k
\Rm_{\bar{g}_{2,k}}|_{\bar{g}_{2,k}}$.

Now we can chop locally within $\bar{D}_{2,k}^6$ (see~\cite{CGIII}): first chop
$\bar{Z}_{2,k}:=(\bar{W}_{1,k'}\cup \bar{D}_{2,k}^0)\cap \bar{D}^3_{2,k}$ to
obtain some $\bar{Q}^0_{2,k}$, then choose a smooth interpolation to glue the
newly chopped piece to the previously chopped ones. More specifically, we have
the following
\begin{claim}[Gluing the local choppings]
There is a smooth interpolation between the boundaries $\partial
\bar{W}_{1,k'}\cap (\bar{D}_{2,k}^4\backslash \bar{D}_{2,k}^3)$ and $\partial
\bar{Q}_{2,k}^0\cap (\bar{D}_{2,k}^2\backslash \bar{D}_{2,k}^1)$, so that we
could obtain some $\bar{R}_{2,k}^0\subset \bar{D}_{2,k}^4$, with the property
\begin{align*}\bar{R}_{2,k}^0\
=\ \bar{W}^0_{1,k'}\quad \text{on}\quad \bar{D}_{2,k}^4\backslash
\bar{D}_{2,k}^3,\quad \text{and}\quad \bar{R}_{2,k}^0\ =\ \bar{Q}_{2,k}^0\quad
\text{on}\quad \bar{D}_{2,k}^2.
\end{align*}
\end{claim}

\begin{proof}[Proof of claim]
By the proof of the local smoothing, within $\bar{D}_{2,k}^6$, we have
$\bar{\rho}_{\bar{Z}_{2,k}}^{\sharp}$ as the smoothed distance to
$\mathcal{N}(\bar{Z}_{2,k})$, and
$\bar{Q}^0_{2,k}=(\bar{\rho}_{\bar{Z}_{2,k}}^{\sharp})^{-1}([0,t_{2,k}])$ for
some $t_{2,k}\in I_{2,k}$ with $t_{2,k}\approx 2^{-2}\Psi_{YM}(a,n,\eta)\
l_a(p_{i_{2,k}})$. In addition, we could set
$\bar{Z}_{2,k}':=\bar{W}_{1,k'}\cap (\bar{D}^5_{2,k}\backslash
\bar{D}^0_{2,k})$, with the smoothing of the distance to the orbit of which
being $\bar{\rho}^{\sharp}_{\bar{Z}'_{2,k}}$. Notice that
$$\bar{\rho}_{\bar{Z}_{2,k}}\ \equiv\ \bar{\rho}_{\bar{Z}'_{2,k}}\quad
\text{on}\quad \bar{D}_{2,k}^3\backslash \bar{D}_{2,k}^0,$$ therefore
$$\bar{\rho}^{\sharp}_{\bar{Z}_{2,k}}\ \equiv\
\bar{\rho}^{\sharp}_{\bar{Z}'_{2,k}}\quad \text{on}\quad
\bar{D}_{2,k}^2\backslash \bar{D}_{2,k}^1,$$ and thus $$\bar{Q}^0_{2,k}\cap
(\bar{D}^2_{2,k}\backslash \bar{D}^1_{2,k})\ =\
(\bar{\rho}_{\bar{Z}_{2,k}}^{\sharp})^{-1}([0,t_{2,k}])\ =\
(\bar{\rho}_{\bar{Z}'_{2,k}}^{\sharp})^{-1}([0,t_{2,k}]).$$ On the other hand,
$$\bar{W}_{1,k'}\cap (\bar{D}^4_{2,k}\backslash \bar{D}^1_{2,k})\ =\
(\bar{\rho}_{\bar{Z}'_{2,k}}^{\sharp})^{-1}(0),$$ and now the existence of a
controlled interpolation required above is easily seen: choose a smooth cut-off
function $\lambda_{2,k}:[r^1_{2,k},r^4_{2,k}]\rightarrow [0,t_{2,k}]$ with
controlled derivatives, such that
$\lambda_{2,k}|_{[r^1_{2,k},r^2_{2,k}]}=t_{2,k}$ and
$\lambda_{2,k}|_{[r^3_{2,k},r^4_{2,k}]}=0$, and the desired region
$\bar{R}^0_{2,k}$ is defined as $$\bar{R}_{2,k}^0\ :=\
\left(\bar{D}_{2,k}^2\cap\bar{Q}_{2,k}^0\right)\ \cup\
\left(\bar{\rho}_{\bar{Z}'_{2,k}}^{\sharp}
\left(\lambda_{2,k}(d^{\bar{q}_{2,k}})\right)\right)^{-1}([r^1_{2,k},r^4_{2,k}]).$$
\end{proof}
Clearly $\bar{R}_{2,k}^0$ is $\bar{N}$-invariant and has the expected smooth
boundary whose second fundamental form has control $|II_{\partial
\bar{R}^0_{2,k}}|\ \le\ C\ l_a(p_{i_{2,k}})^{-1}$.

Let $\bar{R}_{2,k}:=\mathcal{O}(\bar{R}_{2,k}^0)$, then the isometric action of
$O(n)$ and the invariance of $\bar{\rho}_{\bar{Z}_{2,k}}^{\sharp}$,
$\bar{\rho}_{\bar{Z}'_{2,k}}^{\sharp}$ under such actions ensure that $\partial
\bar{R}_{2,k}$ is smooth with controlled second fundamental form $|II_{\partial
\bar{R}_{2,k}}|\le C\ l_a(p_{i_{2,k}})^{-1}$. Do such adjustments for each
$\bar{D}^0_{2,k}\in S_2$ and let $\bar{U}_2:=\bar{U}_1\cup
(\cup_k\bar{R}_{2,k})$, we have finished the second step.

Iterate the above procedure for $N$ steps. At the $j$-th step ($j\ge 2$), we
modify members of $\bar{S}_j$ with $t_{j,k}\approx 2^{-j}\Psi_{YM}(a,n,\eta)\
l_a(p_{i_{j,k}})$ for each $k$. By the Harnack inequality of (H1) for
intersecting balls, we could produce a neighborhood $\bar{U}_j$ of $FE$, which
is contained in $FB(E,\frac{a}{2})$, invariant under both $\bar{N}$- and
$O(n)$-actions, and has a smooth, controlled boundary $$|II_{\partial
\bar{U}_j}|\le C\ l_a^{-1}.$$

By (N3) and the invariance of $\bar{U}_N$ under the $O(n)$-action, define
$U:=\bar{U}_N\slash O(n)$, then $E\subset U\subset B(E,\frac{a}{2})$, and $U$
is saturated by the $a$-standard $N$-structure on $B(E,\frac{a}{2})$, with a
smooth and controlled boundary $$|II_{\partial U}|\le C_{GC}\ l_a^{-1}.$$

\quad

\quad
\end{document}